\documentclass[a4paper,11pt,twoside]{amsart}
\usepackage{amsmath}
\usepackage{geometry}
\usepackage{amssymb, latexsym}
\usepackage{verbatim}
\usepackage{enumerate}
\usepackage{comment}
\usepackage{txfonts}

\usepackage{mathtools}
\usepackage[tableposition=top]{caption}
\usepackage{booktabs,dcolumn}

\newtheorem{theorem}{Theorem}[section]
\newtheorem{proposition}[theorem]{Proposition}

\newtheorem{lemma}[theorem]{Lemma}
\newtheorem{corollary}[theorem]{Corollary}

\newtheorem{remark}[theorem]{Remark}

\newcommand\R{\mathbb{R}}
\newcommand\Z{\mathbb{Z}}

\newcommand\C{\mathbb{C}}

\renewcommand\Re{{\operatorname{Re\,}}}
\renewcommand\Im{{\operatorname{Im\,}}}
\newcommand\Log{{\operatorname{Log}}}
\newcommand\eps{\varepsilon}

\begin{document}
\title[Upper bound for de Bruijn-Newman constant]{Effective approximation of heat flow evolution of the Riemann $\xi$ function, and a new upper bound for the de Bruijn-Newman constant}

\author{D.H.J. Polymath}
\address{\tt{http://michaelnielsen.org/polymath1/index.php}}

\begin{abstract}
For each $t \in \R$, define the entire function
$$ H_t(z) \coloneqq \int_0^\infty e^{tu^2} \Phi(u) \cos(zu)\ du$$
where $\Phi$ is the super-exponentially decaying function
$$ \Phi(u) \coloneqq \sum_{n=1}^\infty (2\pi^2  n^4 e^{9u} - 3\pi n^2 e^{5u} ) \exp(-\pi n^2 e^{4u} ).$$
This is essentially the heat flow evolution of the Riemann $\xi$ function.
From the work of de Bruijn and Newman, there exists a finite constant $\Lambda$ (the \emph{de Bruijn-Newman constant}) such that the zeroes of $H_t$ are all real precisely when $t \geq \Lambda$.  The Riemann hypothesis is equivalent to the assertion $\Lambda \leq 0$; recently, Rodgers and Tao established the matching lower bound $\Lambda \geq 0$.  Ki, Kim and Lee established the upper bound $\Lambda < \frac{1}{2}$.

In this paper we establish several effective estimates on $H_t(x+iy)$ for $t \geq 0$, including some that are accurate for small or medium values of $x$.  By combining these estimates with numerical computations, we are able to obtain a new upper bound $\Lambda \leq 0.22$ unconditionally, as well as improvements conditional on further numerical verification of the Riemann hypothesis.  We also obtain some new estimates controlling the asymptotic behavior of zeroes of $H_t(x+iy)$ as $x \to \infty$.
\end{abstract}

\maketitle

\section{Introduction}

Let $H_0 \colon \C \to \C$ denote the function
\begin{equation}\label{hoz}
 H_0(z) \coloneqq \frac{1}{8} \xi\left(\frac{1}{2} + \frac{iz}{2}\right),
\end{equation}
where $\xi \colon \C \to \C$ denotes the Riemann $\xi$ function
\begin{equation}\label{sas}
 \xi(s) \coloneqq \frac{s(s-1)}{2} \pi^{-s/2} \Gamma\left(\frac{s}{2}\right) \zeta(s)
\end{equation}
(which is an entire function after removing all singularities) and $\zeta$ is the Riemann $\zeta$ function.
Then $H_0$ is an entire even function with functional equation $H_0(\overline{z}) = \overline{H_0(z)}$, and the Riemann hypothesis (RH) is equivalent to the assertion that all the zeroes of $H_0$ are real.

It is a classical fact (see \cite[p. 255]{titch}) that $H_0$ has the Fourier representation
$$ H_0(z) = \int_0^\infty \Phi(u) \cos(zu)\ du$$
where $\Phi$ is the super-exponentially decaying function
\begin{equation}\label{phidef}
 \Phi(u) \coloneqq \sum_{n=1}^\infty (2\pi^2  n^4 e^{9u} - 3\pi n^2 e^{5u} ) \exp(-\pi n^2 e^{4u} ).
\end{equation}
The sum defining $\Phi(u)$ converges absolutely for negative $u$ also.  From Poisson summation one can verify that $\Phi$ satisfies the functional equation $\Phi(u) = \Phi(-u)$ (i.e., $\Phi$ is even); this fact is of course closely related to the functional equation for $\zeta$. 

De Bruijn \cite{debr} introduced (with somewhat different notation) the more general family of functions $H_t \colon \C \to \C$ for $t \in \R$, defined by the formula
\begin{equation}\label{htdef}
 H_t(z) \coloneqq \int_0^\infty e^{tu^2} \Phi(u) \cos(zu)\ du.
\end{equation}
As noted in \cite[p.114]{csv}, one can view $H_t$ as the evolution of $H_0$ under the backwards heat equation $\partial_t H_t(z)= -\partial_{zz} H_t(z)$.
As with $H_0$, each of the $H_t$ are entire even functions with functional equation $H_t(\overline{z}) = \overline{H_t(z)}$; from the super-exponential decay of $e^{tu^2} \Phi(u)$ we see that the $H_t$ are in fact entire of order $1$.  It follows from the work of P\'olya \cite{polya} that if $H_t$ has purely real zeroes for some $t$, then $H_{t'}$ has purely real zeroes for all $t'>t$; de Bruijn showed that the zeroes of $H_t$ are purely real for $t \geq 1/2$.  Newman \cite{newman} strengthened this result by showing that there is an absolute constant $-\infty < \Lambda \leq 1/2$, now known as the \emph{De Bruijn-Newman constant}, with the property that $H_t$ has purely real zeroes if and only if $t \geq \Lambda$.  The Riemann hypothesis is then clearly equivalent to the upper bound $\Lambda \leq 0$.  Recently in \cite{brad} the complementary bound $\Lambda \geq 0$ was established, answering a conjecture of Newman \cite{newman}, and improving upon several previous lower bounds for $\Lambda$ \cite{cnv,nrv,crv,cosv,odlyzko,saouter}.  Furthermore, Ki, Kim, and Lee \cite{kkl} sharpened the upper bound $\Lambda \leq 1/2$ of de Bruijn \cite{debr} slightly to $\Lambda < 1/2$.  

In this paper we improve the upper bound:

\begin{theorem}[New upper bound]\label{new-upper}  We have $\Lambda \leq 0.22$.
\end{theorem}

The proof of Theorem \ref{new-upper} combines numerical verification with some new asymptotics and observations about the $H_t$ which may be of independent interest.  Firstly, by analyzing the dynamics of the zeroes of $H_t$, we establish in Section \ref{dynamics-sec} the following criterion for obtaining upper bounds on $\Lambda$:

\begin{theorem}[Upper bound criterion]\label{ubc-0}  Suppose that $t_0, X > 0$ and $0 < y_0 \leq 1$ obey the following hypotheses:
\begin{itemize}
\item[(i)]  (Numerical verification of RH at initial time $0$) There are no zeroes $\zeta(\sigma+iT) = 0$ with $\frac{1+y_0}{2} \leq \sigma \leq 1$ and $0 \leq T \leq \frac{X}{2}$.
\item[(ii)]  (Asymptotic zero-free region at final time $t_0$) There are no zeroes $H_{t_0}(x+iy)=0$ with $x \geq X+\sqrt{1-y_0^2}$ and $y_0 \leq y \leq \sqrt{1-2t_0}$.
\item[(iii)]  (Barrier at intermediate times) There are no zeroes $H_{t}(x+iy)=0$ with $X \leq x \leq X+\sqrt{1-y_0^2}$, $\sqrt{y_0^2+2(t_0-t)} \leq y \leq \sqrt{1-2t}$, and $0 \leq t \leq t_0$.
\end{itemize}
Then $\Lambda \leq t_0 + \frac{1}{2} y_0^2$.
\end{theorem}

Informally, hypothesis (i) implies that at time $t=0$, there are no zeroes $H_t(x+iy)=0$ with large values of $y$ to the left of the barrier region in (iii).  The absence of zeroes in that barrier, together with a continuity argument and an analysis of the time derivative of each zero, can then be used to show that for later times $0 < t \leq t_0$, there continue to be no zeroes $H_t(x+iy)=0$ with large values of $y$ to the left of the barrier; see Figure \ref{barrier-fig}.  Hypothesis (ii) then gives the complementary assertion to the right of the barrier, and one can use an existing theorem of de Bruijn (Theorem \ref{debr-bound}) to conclude.

In practice, we have found it convenient numerically to replace the barrier region in Theorem \ref{ubc-0} with the larger and simpler region
$$ X \leq x \leq X+1; \quad y_0 \leq y \leq 1; \quad 0 \leq t \leq t_0.$$

\begin{figure}[ht!]
  \includegraphics[width=0.9\linewidth]{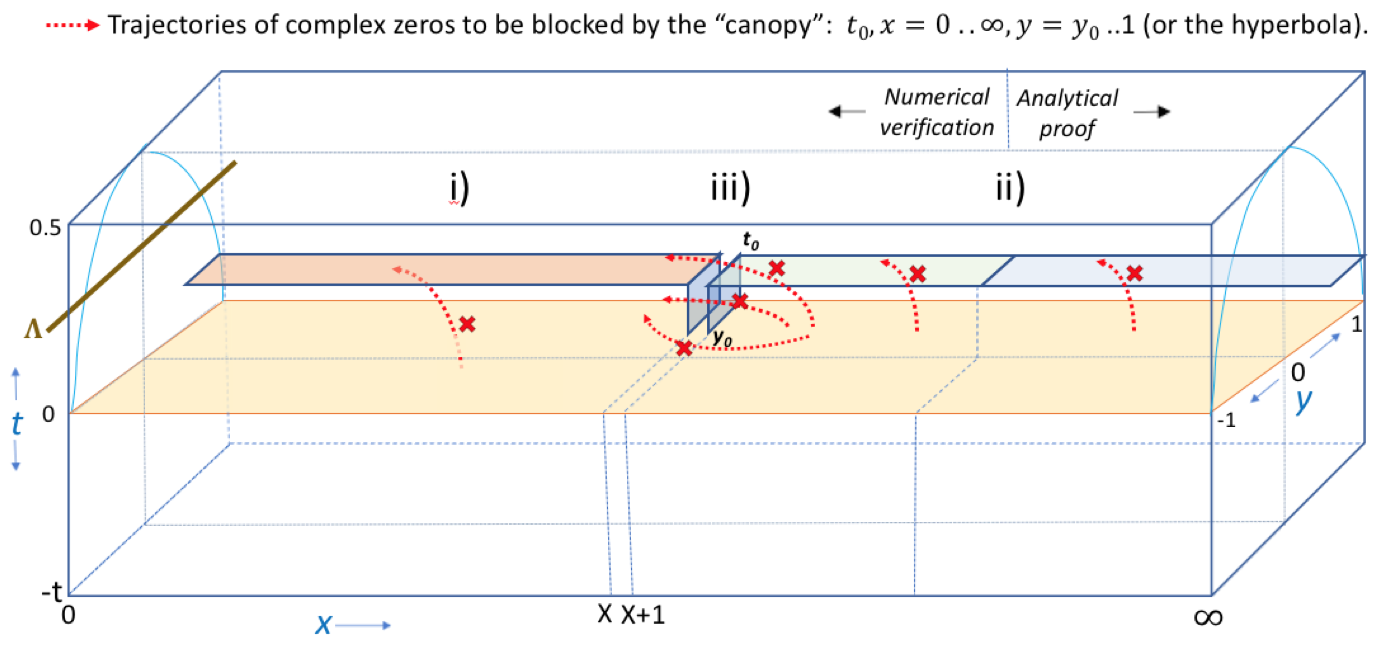}
  \caption{A visualization of Theorem \ref{ubc-0}.  If at time $t_0$ one can show that there are no zeroes $H_t(x+iy)=0$ in the ``canopy'' $0 \leq x <\infty$, $y_0 \leq y \leq 1$, then a theorem of de Bruijn allows one to conclude the desired bound $\Lambda \leq t_0 + \frac{1}{2} y_0^2$.  The hypothesis (i) prevents zeroes hitting this canopy from an initial position to the left of the barrier; the hypothesis (ii) prevents zeroes from lying in the canopy to the right of the barrier; and the hypothesis (iii) prevents zeroes from starting to the right of the barrier and reaching the canopy to the left of the barrier.}
\label{barrier-fig}
\end{figure}

We will obtain Theorem \ref{new-upper} by applying Theorem \ref{ubc-0} with the specific numerical choices $t_0 = 0.2$, $X = 6 \times 10^{10} + 83952 - 0.5$, and $y_0 = 0.2$.  The reason we choose $X$ close to $6 \times 10^{10}$ is that this is near the limit of known numerical verifications of the Riemann hypothesis such as \cite{platt}, which we need for the hypothesis (i) of the above theorem; the shift $83952 - 0.5$ is in place to make the partial Euler product $\prod_{p \leq 11} \left(1 - \frac{1}{p^{\frac{1-iX}{2}}}\right)^{-1}$ large, which helps in keeping the functions $H_t(x+iy), H_{t_0}(x+iy)$ large in magnitude, which in turn is helpful for numerical verifications of (ii) and (iii); see also Figure \ref{euler}.  The choices $t_0=0.2, y_0=0.2$ are then close to the limit of our ability to numerically verify hypothesis (ii) for this choice of $X$.  (The hypothesis (iii) is also verified numerically, but can be done quite quickly compared to (ii), and so does not present the main bottleneck to further improvements to Theorem \ref{new-upper}.)  Further upper bounds to $\Lambda$ can be obtained if one assumes the Riemann hypothesis to hold up to larger heights than that in \cite{platt}: see Section \ref{further-sec}.

To verify (ii) and (iii), we need efficient approximations (of Riemann-Siegel type) for $H_t(x+iy)$ in the regime where $t,y$ are bounded and $x$ is large.  For sake of numerically explicit constants, we will focus attention on the region
\begin{equation}\label{region}
0 < t \leq \frac{1}{2}; \quad 0 \leq y \leq 1; \quad x \geq 200,
\end{equation}
though the results here would also hold (with different explicit constants) if the numerical quantities $\frac{1}{2}, 1, 200$ were replaced by other quantities.

A key difficulty here is that $H_t(x+iy)$ decays exponentially fast in $x$ (basically because of the Gamma factor in \eqref{sas}); see Figure \ref{loghtandlogbt}.  This means that any direct attempt to numerically establish a zero-free region for $H_t(x+iy)$ for large $x$ would require enormous amounts of numerical precision.  To get around this, we will first renormalise the function $H_t(x+iy)$ by dividing it by a nowhere vanishing explicit function $B_t(x+iy)$ (basically a variant of the aforementioned Gamma factor) that removes this decay.  To describe this function, we first introduce the function $M_0: \C \backslash (-\infty,1] \to \C \backslash \{0\}$ defined by the formula
\begin{equation}\label{M-def}
 M_0(s) \coloneqq \frac{1}{8} \frac{s(s-1)}{2} \pi^{-s/2} \sqrt{2\pi} \exp\left( \left(\frac{s}{2}-\frac{1}{2}\right)\Log \frac{s}{2} - \frac{s}{2} \right),
\end{equation}
where $\Log$ denotes the standard branch of the complex logarithm, with branch cut at the negative axis and imaginary part in $(-\pi,\pi]$. One may interpret $M_0(s)$ as the Stirling approximation to the factor $\frac{1}{8} \frac{s(s-1)}{2} \pi^{-s/2} \Gamma\left(\frac{s}{2}\right)$ appearing in \eqref{hoz}, \eqref{sas}; it decays exponentially as one moves to infinity $s \to \pm i \infty$ along the critical strip.  We may form a holomorphic branch $\log M_0: \C \backslash (-\infty,1] \to \C$ of the logarithm of $M_0$ by the formula
\begin{equation}\label{logM}
 \log M_0(s) \coloneqq \Log s + \Log(s-1) - \frac{s}{2} \log \pi + \log \frac{\sqrt{2\pi}}{16} + 
 \left(\frac{s}{2}-\frac{1}{2}\right)\Log \frac{s}{2} - \frac{s}{2};
\end{equation}
differentiating this, we see that the logarithmic derivative $\alpha: \C \backslash (-\infty,1] \to \C$ of this function, defined by
\begin{equation}\label{alpha-def}
\alpha \coloneqq (\log M_0)' = \frac{M'_0}{M_0}
\end{equation}
is given explicitly by the formula
\begin{equation}\label{alpha-form}
\begin{split}
 \alpha(s) &= \frac{1}{s} + \frac{1}{s-1} - \frac{1}{2} \log \pi + \frac{1}{2} \Log \frac{s}{2} - \frac{1}{2s} \\
&= \frac{1}{2s} + \frac{1}{s-1} + \frac{1}{2} \Log \frac{s}{2\pi}.
\end{split}
\end{equation}
For any time $t \in \R$, we then define the deformation $M_t: \C \backslash (-\infty,1]$ of $M_0$ by the formula
\begin{equation}\label{Mt-def}
M_t(s) \coloneqq \exp\left( \frac{t}{4} \alpha(s)^2 \right) M_0(s)
\end{equation}
for any $t \geq 0$.
In the region \eqref{region}, we introduce the quantity
\begin{equation}\label{bo-def} 
B_t(x+iy) \coloneqq M_t\left(\frac{1+y-ix}{2}\right).
\end{equation} 
For fixed $t \geq 0$ and $y>0$, $B_t(x+iy)$ is non-vanishing, and it is easy to verify the asymptotic $|B_t(x+iy)| = e^{-(\frac{\pi}{8} + o(1)) x}$.  As it turns out, $B_t(x+iy)$ is an asymptotic approximation to $H_t(x+iy)$ in the region \eqref{region}, in the sense that
\begin{equation}\label{limx}
 \lim_{x \to \infty} \frac{H_t(x+iy)}{B_t(x+iy)} = 1
\end{equation}
for any fixed $t>0$ and $y>0$; see Figure \ref{ht-bt}.  (However, the convergence of \eqref{limx} is \emph{not} uniform as $t$ approaches zero.)

\begin{figure}[ht!]
  \includegraphics[width=1.0\linewidth]{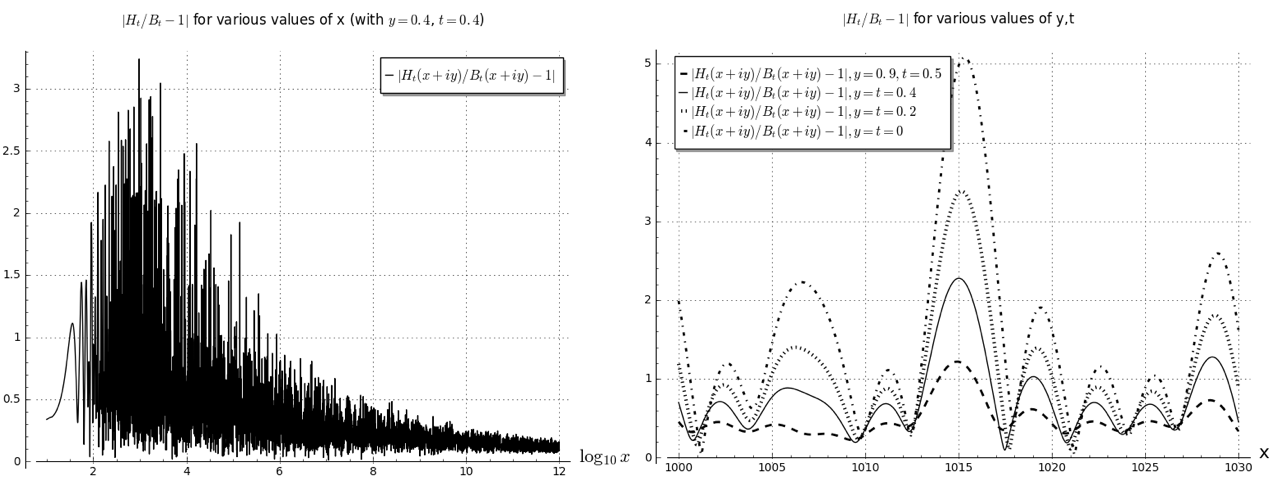}
  \caption{The quantity $|H_t(x+iy)/B_t(x+iy) - 1|$ when $y=t=0.4$ and $\log_{10} x \leq 12$ (left) and when for $1000 \leq x \leq 1030$ and various choices of $(y,t)$ (right).}
\label{ht-bt}
\end{figure}

\begin{figure}[ht!]
  \includegraphics[width=0.8\linewidth]{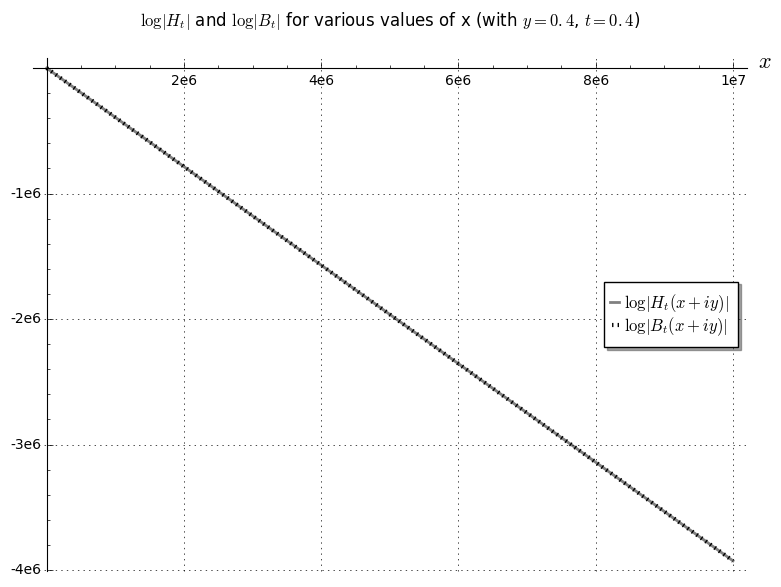}
  \caption{The quantities $\log|H_t(x+iy)|$ and $\log|B_t(x+iy)|$ when $y=t=0.4$ and $\log_{10} x \leq 7$.   Both quantities decay like $- \frac{\pi}{8} x \approx -0.393 x$.}
\label{loghtandlogbt}
\end{figure}

In fact we have the following significantly more accurate approximation (of Riemann-Siegel type) with effective error estimates.
For any real number $X$, let $O_{\leq}(X)$ denote a quantity that is bounded in magnitude by $X$. We also use $x_+ = \max(x,0)$ to denote the positive part of a real number $x$.

\begin{theorem}[Effective Riemann-Siegel approximation to $H_t(x+iy)$]\label{eff}  Let $t,x,y$ lie in the region \eqref{region}.  Then we have
\begin{equation}\label{ratio-form-eff}
\frac{H_t(x+iy)}{B_t(x+iy)} = f_t(x+iy) + O_{\leq}\left( e_A + e_B + e_{C,0} \right)
\end{equation}
where
\begin{align}
f_t(x+iy) &\coloneqq \sum_{n=1}^N \frac{b_n^t}{n^{s_*}} + \gamma \sum_{n=1}^N n^y \frac{b_n^t}{n^{\overline{s_*} + \kappa}}\label{ft-def} \\
b_n^t &\coloneqq \exp( \frac{t}{4} \log^2 n ) \label{bn-def}\\
\gamma = \gamma(x+iy) &\coloneqq \frac{M_t\left(\frac{1-y+ix}{2}\right)}{M_t\left(\frac{1+y-ix}{2}\right)} \label{lambda-def} \\
s_* = s_*(x+iy) &\coloneqq \frac{1+y-ix}{2} +\frac{t}{2} \alpha\left(\frac{1+y-ix}{2}\right) \label{sn-def}\\
\kappa = \kappa(x+iy) &\coloneqq \frac{t}{2} \left(\alpha\left(\frac{1-y+ix}{2}\right) - \alpha\left(\frac{1+y+ix}{2}\right)\right) \label{kappa-def}\\
N &\coloneqq \left\lfloor \sqrt{\frac{x}{4\pi} + \frac{t}{16}} \right\rfloor \label{N-def-main} 
\end{align}
and $e_A, e_B, e_{C,0}$ are certain explicitly computable positive quantities\footnote{See \eqref{ea-def}-\eqref{ec-def} for the precise definition of these quantities.} depending on $t$ and $x+iy$.  Furthermore, we have the following bounds:
\begin{align}
|\gamma| &\leq e^{0.02 y} \left( \frac{x}{4\pi} \right)^{-y/2}  \label{gamma-bound} \\
\Re s_* &\geq \frac{1+y}{2} +\frac{t}{4} \log \frac{x}{4\pi} - \frac{t}{2x^2} \left(1-3y+\frac{4y(1+y)}{x^2}\right)_+  \label{res-bound} \\
|\kappa| &\leq  \frac{ty}{2(x-6)} \label{kappa-bound} \\
e_A + e_B &\leq \sum_{n=1}^N (1 + |\gamma| N^{|\kappa|} n^y) \frac{b_n^t}{n^{\Re s_*}} \left( \exp\left( \frac{\frac{t^2}{16} \log^2 \frac{x}{4\pi n^2} + 0.626}{x-6.66} \right)-1 \right) \label{eab-bound} \\
e_{C,0} &\leq \left(\frac{x}{4\pi}\right)^{-\frac{1+y}{4}} \exp\left( - \frac{t}{16} \log^2 \frac{x}{4\pi} + \frac{1.24 \times (3^y+3^{-y})}{N-0.125} + \frac{3 |\log \frac{x}{4\pi} + i \frac{\pi}{2}|+10.44}{x-12} \right) \label{ec-bound}
\end{align}
\end{theorem}

This theorem will be proven in Section \ref{initial-sec}; see Figures \ref{htft}, \ref{htft-2} for a numerical illustration of the approximation.  The strategy is to express $H_t$ as a convolution of $H_0$ with a gaussian heat kernel, then apply an effective Riemann-Siegel expansion to $H_0$ to rewrite $H_t$ as the sum of various contour integrals; see Section \ref{heatflow-sec} for details.  One then uses the saddle point method to shift each such contour to a location that is suitable for effective estimation.   We remark that $f_t(x+iy)$ is a holomorphic function of $x+iy$ in the region \eqref{region} as long as $N$ is constant, but has jump discontinuities when $N$ is incremented.

\begin{figure}[ht!]
  \includegraphics[width=\linewidth]{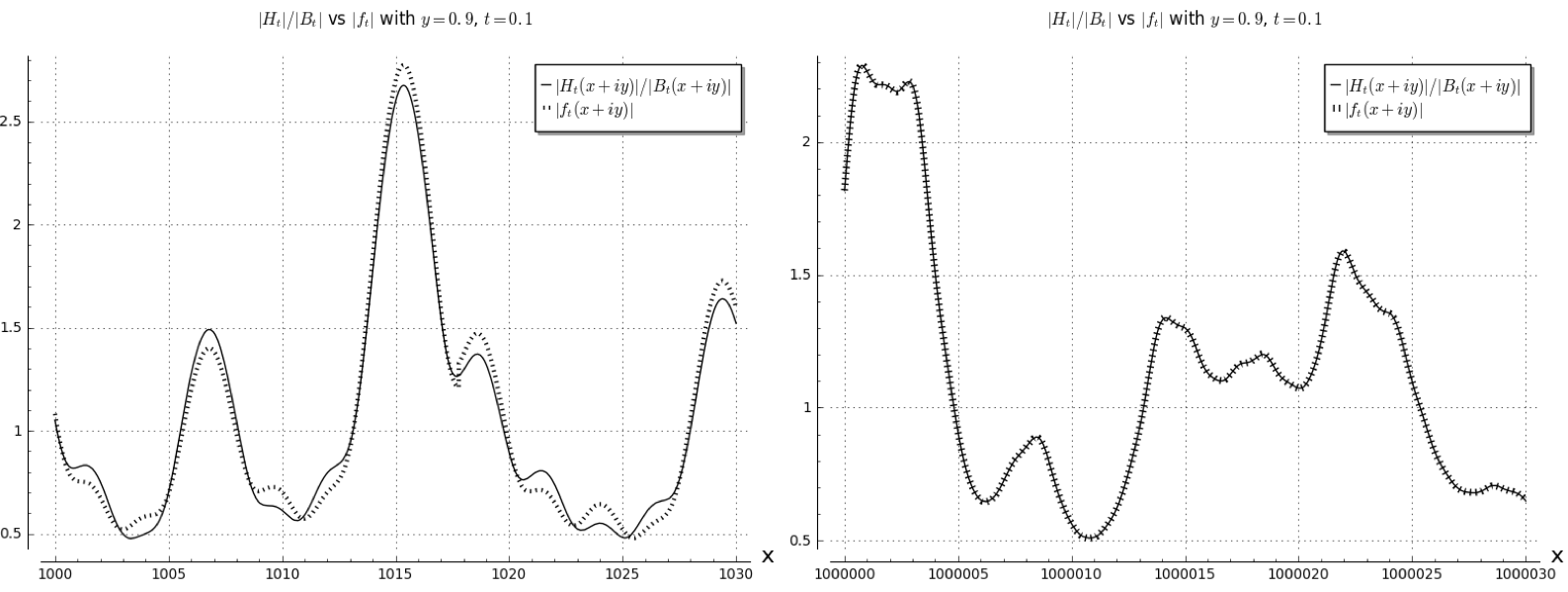}
  \caption{Comparison of $|f_t(x+iy)|$ and $|H_t(x+iy)|/|B_t(x+iy)|$ for $y=0.9$, $t=0.1$, $1000 \leq x \leq 1030$, (left) and when $10^6 \leq x \leq 10^6 + 30$ (right). The approximation improves as $x$ gets larger.}
	\label{htft}
\end{figure}

\begin{figure}[ht!]
  \includegraphics[width=\linewidth]{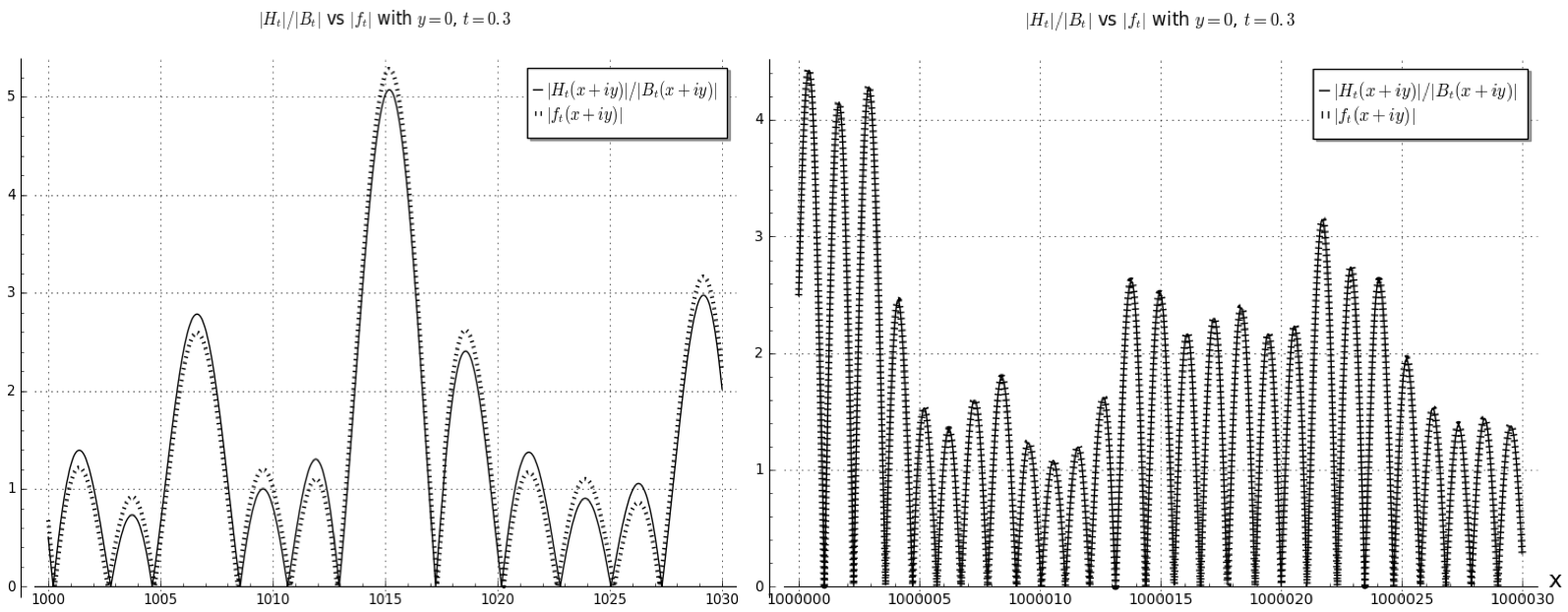}
  \caption{Comparison of $|f_t(x+iy)|$ and $|H_t(x+iy)|/|B_t(x+iy)|$ for $y=0$, $t=0.3$, $1000 \leq x \leq 1030$ (left), and when $10^6 \leq x \leq 10^6 + 30$ (right). Again notice the improving approximation with $x$.}
	\label{htft-2}
\end{figure}

\begin{figure}[ht!]
  \includegraphics[width=0.8\linewidth]{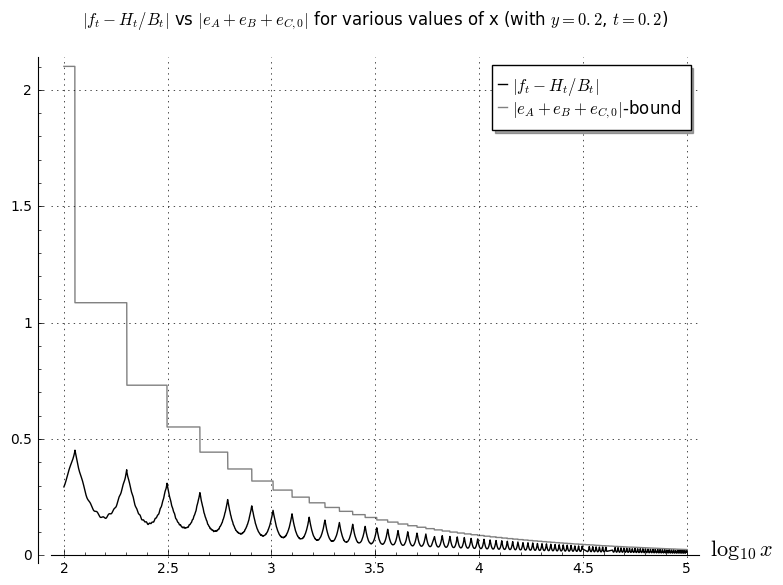}
  \caption{The error upper bound $|e_A+e_B+e_{C,0}|$ versus $|f_t - H_t/B_t|$ when $y=t=0.2$ and $2 \leq \log_{10} x \leq 5$.}
\label{errorboundtot}
\end{figure}

\begin{figure}[ht!]
  \includegraphics[width=1.0\linewidth]{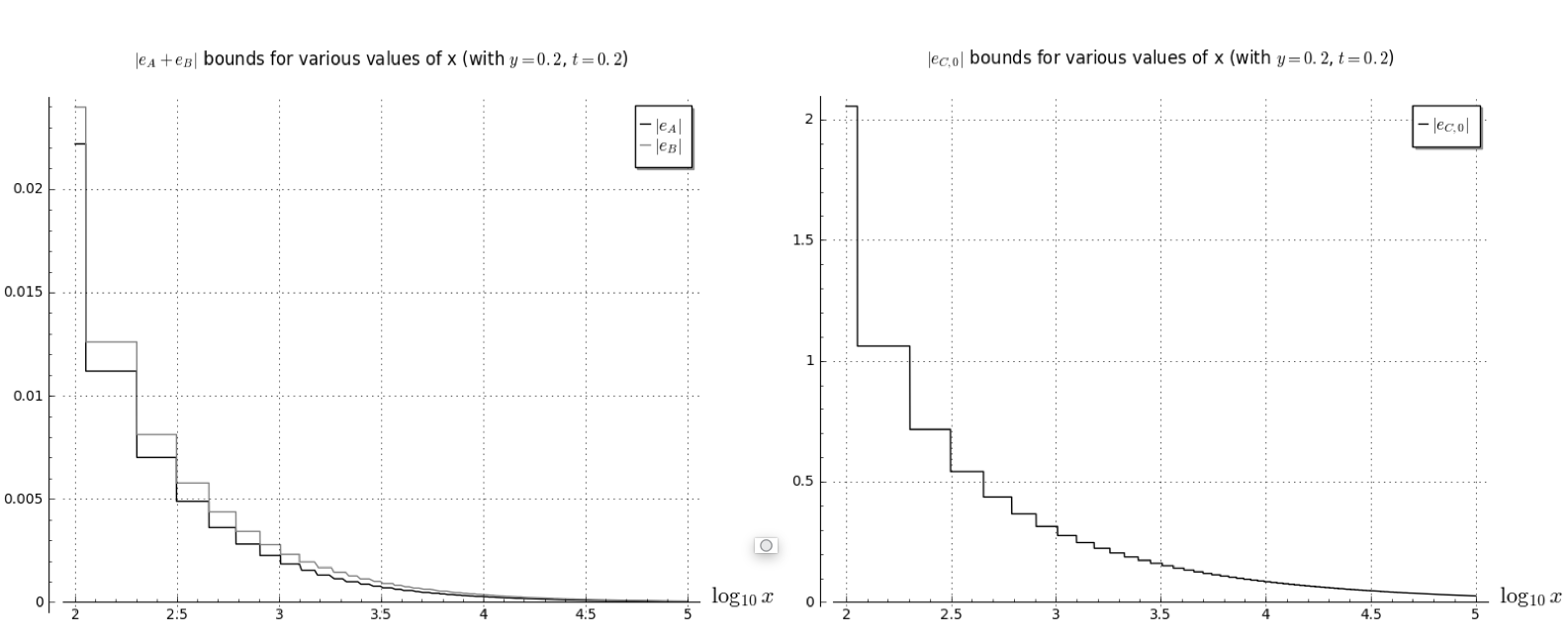}
  \caption{The individual error upper bounds $|e_A|$, $|e_B|$ and $|e_{C,0}|$ with $y=t=0.2$ and $2 \leq \log_{10} x \leq 5$. The $e_{C,0}$-term clearly dominates.}
\label{ind_errorbounds}
\end{figure}

From \eqref{ratio-form-eff} and the triangle inequality, we have a numerically verifiable criterion to establish non-vanishing of $H_t$ at a given point:

\begin{corollary}[Criterion for non-vanishing]\label{zero-test}  Let $t,x,y$ lie in the region \eqref{region}, and let $f_t, e_A, e_B, e_{C,0}$ be as in Theorem \ref{eff}.  If one has the inequality
\begin{equation}\label{criterion}
|f_t(x+iy)| > e_A + e_B + e_{C,0} 
\end{equation}
then $H_t(x+iy) \neq 0$.  
\end{corollary}

Actually, for some regions of $x,y,t$ we will use a more complicated criterion than \eqref{criterion}, in order to exploit the argument principle.  To numerically estimate $f_t(x+iy)$ in a feasible amount of time, we will use Taylor expansion to be able to efficiently compute many values of $f_t(x+iy)$ simultaneously (see Section \ref{multiple-sec}), and for some ranges of the parameters $t,x,y$ we will also use an Euler product mollifier to reduce the amount of oscillation in the sum $f_t(x+iy)$ (see Section \ref{b-bound}).

For $y,t$ fixed and $x$ sufficiently large, we have the asymptotics
\begin{align*}
e_A,e_B, e_{C,0} &= O( x^{-ct} )\\
f_t(x+iy) &= 1 + O(x^{-ct} )
\end{align*}
for some absolute constant $c>0$; see Proposition \ref{asymp}(i) and its proof.  This gives the crude asymptotic \eqref{limx} in the region \eqref{region} at least.   In practice, the $e_{C,0}$ term numerically dominates the $e_A+e_B$ term, although both errors will be quite small in the ranges of $x$ under consideration; in particular, for the ranges needed to verify conditions (ii) and (iii) of Theorem \ref{ubc-0}, we can make $e_A+e_B$ and $e_{C,0}$ both significantly smaller than $|f_t(x+iy)|$.  In the spirit of expanding the Riemann-Siegel approximation to higher order, we also obtain an even more accurate explicit approximation in which a correction term $-\frac{C_t}{B_t}$ is added to $f_t$, and the error term $e_{C,0}$ is replaced by a smaller quantity $e_C$; see \eqref{ratio-form-refined} and Figure \ref{htft-c}.

\begin{figure}[ht!]
  \includegraphics[width=0.6\linewidth]{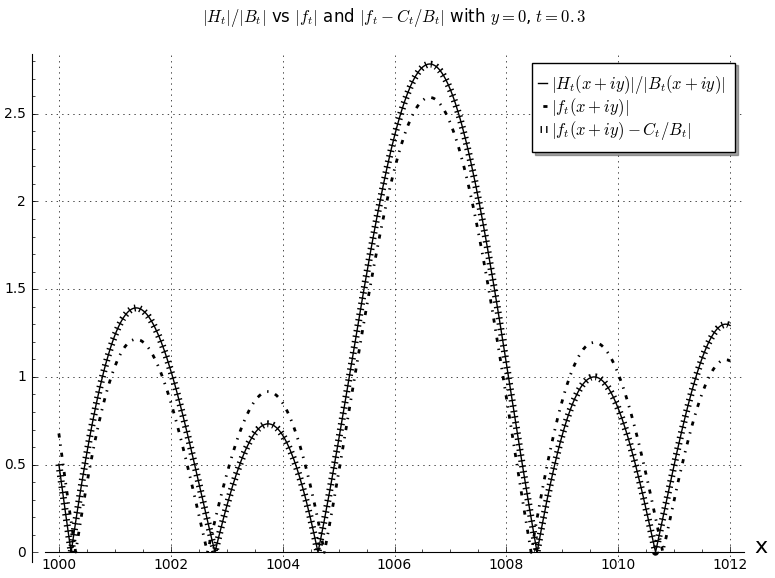}
  \caption{Comparison of $|f_t(x+iy)|$, $|f_t(x+iy)-\frac{C_t(x+iy)}{B_t(x+iy)}|$ and $|H_t(x+iy)|/|B_t(x+iy)|$ for $y=0$, $t=0.3$, $1000 \leq x \leq 1012$.}
	\label{htft-c}
\end{figure}

In addition to establishing upper bounds such as Theorem \ref{new-upper}, one can use Theorem \ref{eff} and Corollary \ref{zero-test} (together with variants in slightly larger regions than \eqref{region}, for instance if $y$ is allowed to be as large as $10$) to obtain asymptotic control on the zeroes of $H_t$, refining previous work of Ki, Kim, and Lee \cite{kkl}.  Indeed, in Section \ref{asymptotic-sec} we will establish

\begin{theorem}[Distribution of zeroes of $H_t$]\label{Zero}  Let $0 < t \leq 1/2$, let $C>0$ be a sufficiently large absolute constant, and let $c>0$ be a sufficiently small absolute constant.  For $x \geq 4\pi$, define
$$ g(x,t) \coloneqq  \frac{x}{4\pi} \log \frac{x}{4\pi} - \frac{x}{4\pi} + \frac{11}{8} + \frac{t}{16} \log \frac{x}{4\pi} $$
and for all $n \geq C$, let $x_n$ be the unique real number greater than $4\pi$ such that
\begin{equation}\label{lip}
 g(x_n,t) = n.
\end{equation}
(This is well-defined since the $g(x,t)$ is increasing in $x$ for $x \geq 4\pi$.)
\begin{itemize}
\item[(i)]  If $x \geq \exp(\frac{C}{t})$ and $H_t(x+iy)=0$, then $y=0$, and
$$ x = x_n + O(x^{-ct})$$
for some $n$.  
\item[(ii)]  Conversely, for each $n \geq \exp( \frac{C}{t} )$ there is exactly one zero $H_t$ in the disk $\{ x+iy: |x+iy - x_n| \leq \frac{c}{\log x_n} \}$ (and by part (i), this zero will be real and lie within $O(x^{-ct})$ of $x_n$).
\item[(iii)]  If $X \geq \exp(\frac{C}{t})$, the number $N_t(X)$ of zeroes with real part between $0$ and $X$ (counting multiplicity) is
$$ N_t(X) = g(X,t) + O(1).$$
\item[(iv)]  For any $X \geq 0$, one has
$$ N_t(X+1) - N_t(X) \leq O( \log(2+X) )$$
and
$$ N_t(X) = g(X,t) + O( \log(2+X) ).$$
\end{itemize}
Here and in the sequel we use $X = O(Y)$ to denote the estimate $|X| \leq AY$ for some constant $A$ that is absolute (in particular, $A$ is independent of $t$ and $C$).
\end{theorem}

Roughly speaking, these estimates tell us that the zeroes of $H_t$ behave (on macroscopic scales) like those of $H_0$ in the region $x = O(\exp(O(1/t)))$, and are very evenly spaced (and on the real axis) outside of this range.  The factor $\frac{t}{16} \log \frac{x_n}{4\pi}$ in \eqref{lip} indicates that as time $t$ advances, the zeroes (or at least those with large values of $x$) will tend to move towards the origin at a speed of approximately $\frac{\pi}{4}$.  Although we will not prove this here, the conclusions (i) and (iii) suggest that one in fact has an asymptotic of the form
$$ N_t(X) = \left\lfloor g(X,t) + O( X^{-ct} ) \right\rfloor$$
when $X \geq \exp(C/t)$; in particular (since the sawtooth function $x - \lfloor x \rfloor$ has average value $\frac{1}{2}$) one would have the heuristic approximation
$$ N_t(X) \approx \frac{X}{4\pi} \log \frac{X}{4\pi} - \frac{X}{4\pi} + \frac{7}{8} + \frac{t}{16} \log \frac{X}{4\pi} $$
 after performing some averaging in $X$, thus recovering the familiar $\frac{7}{8}$ term in the usual averaged asymptotics for $N_0(X)$.

\begin{figure}[ht!]
  \includegraphics[width=1.0\linewidth]{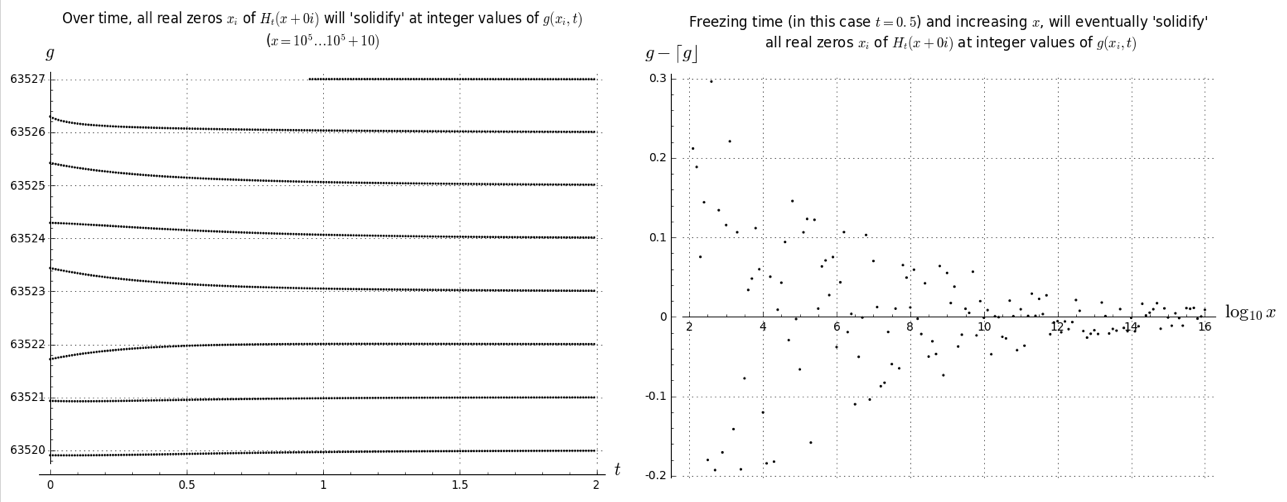}
  \caption{The real zeroes $x_i$ of $H_t$ will converge to integer values of $g(x_i,t)$ when $t$ (left) and/or $x$ (right) increases.}
\label{integer_conv_zeros}
\end{figure}

The results in Theorem \ref{Zero} refine previous results of Ki, Kim, and Lee \cite[Theorems 1.3, 1.4]{kkl}, which gave similar results but with constants that depended on $t$ in a non-uniform (and ineffective) fashion, and error terms that were of shape $o(1)$ rather than $O(x^{-ct})$ in the limit $x \to \infty$ (holding $t$ fixed).  The results may also be compared with those in \cite{arias-lune}, who (in our notation) show that assuming RH, the zeroes of $H_0$ are precisely the solutions $x_n$ to the equation
$$ \frac{1}{2\pi} \mathrm{arg}\left( - e^{2 i \vartheta(x_n/2)} \frac{\zeta'(\frac{1-ix_n}{2})}{\zeta'(\frac{1+ix_n}{2})}\right) = n $$
for integer $n$, where $-\vartheta(t)$ is the phase of $\zeta(\frac{1}{2}+it)$ and one chooses a branch of the argument so that the left-hand side is $-\frac{1}{2}$ when $x_n=0$.

\begin{remark}  One can draw an analogy between the various potential behaviours of zeroes of $H_t$ and the three classical states of matter.  A ``gaseous'' state corresponds to the situation in which some fraction of the zeroes of $H_t$ are strictly complex.  A ``liquid'' state corresponds to a situation in which the zeroes are real, but disordered (with highly unequal spacings between zeroes).  A ``solid'' state corresponds to a situation in which the zeroes are real and arranged roughly in an arithmetic progression.  Thus for instance the Riemann hypothesis and the GUE hypothesis assert (roughly speaking) that the zeroes of $H_0$ should exhibit liquid behaviour everywhere, while Theorem \ref{Zero} asserts that the zeroes of $H_t$, $t>0$ ``solidify'' in the region $x \geq \exp(C/t)$.  Below this region we expect liquid behaviour.  In general, as the parameter $t$ increases, the zeroes appear\footnote{This is the picture for positive $t$ at least.  As $t$ becomes very negative, it appears that the ``gaseous'' zeroes become more ordered again, for instance organizing themselves into curves in the complex plane.  See \cite{sharkfin} for further discussion of this phenomenon.} to ``cool'' down, transitioning from gaseous to liquid to solid type states; see \cite{brad} for some formalisations of this intuition. 
\end{remark}

\subsection{About this project}

This paper is part of the \emph{Polymath project}, which was launched
by Timothy Gowers in February 2009 as an experiment to see if research
mathematics could be conducted by a massive online collaboration.
The current project (which was administered by Terence Tao) is the fifteenth
project in this series.  Further information on the Polymath project can be
found on the web site {\tt michaelnielsen.org/polymath1}.  Information
about this specific project may be found at
\begin{center}
\small{{\tt michaelnielsen.org/polymath1/index.php?title=De\_Bruijn-Newman\_constant}}
\end{center}
and a full list of participants and their grant acknowledgments may be
found at
\begin{center}
\small{{\tt michaelnielsen.org/polymath1/index.php?title=Polymath15\_grant\_acknowledgments}}
\end{center}

We thank the anonymous referees for their careful reading of the paper and for several useful corrections and suggestions.

\section{Notation}

We use the standard branch $\Log$ of the logarithm to define the standard complex powers $z^w \coloneqq \exp( w \Log z)$, and in particular define the standard square root $\sqrt{z} \coloneqq z^{1/2} = \exp( \frac{1}{2} \Log z)$.  We record the familiar gaussian identity
\begin{equation}\label{gaussian}
 \int_\R \exp\left(-(au^2+bu+c)\right)\ du = \sqrt{\frac{\pi}{a}} \exp\left( \frac{b^2}{4a} - c\right)
\end{equation}
for any complex numbers $a,b,c$ with $\Re a > 0$.

When using order of magnitude notation such as $O_{\leq}(X)$, any expression of the form $A=B$ using this notation should be interpreted as the assertion that any quantity of the form $A$ is also of the form $B$, thus for instance $O_{\leq}(1) + O_{\leq}(1) = O_{\leq}(3)$.  (In particular, the equality relation is no longer symmetric with this notation.)

If $F$ is a meromorphic function, we use $F'$ to denote its derivative.  We also use $F^*$ to denote the reflection $F^*(s) := \overline{F(\overline{s})}$ of $F$.  Observe from analytic continuation that if $F:\Omega \to \C$ is holomorphic on a connected open domain $\Omega \subset \C$ containing an interval in $\R$, and is real-valued on $\Omega \cap \R$, then it is equal to its own reflection: $F = F^*$ (since the holomorphic function $F - F^*$ has an uncountable number of zeroes).


\section{Dynamics of zeroes}\label{dynamics-sec}

In this section we control the dynamics of the zeroes of $H_t$ in order to establish Theorem \ref{ubc-0}.  As $H_t$ is even with functional equation $H_t = H_t^*$, the zeroes are symmetric around the origin and the real axis; from \eqref{htdef} and the positivity of $\Phi$, we also see that $H_t(iy) > 0$ for all $y \in \R$, so there are no zeroes on the imaginary axis.  From the super-exponential decay of $\Phi$ and \eqref{htdef} we see that the entire function $H_t$ is of order $1$; by Jensen's formula, this implies that the number of zeroes in a large disk $D(0,R)$ is at most $O( R^{1+o(1)})$ as $R \to \infty$.

We begin with the analysis of the dynamics of a single zero of $H_t$:

\begin{proposition}[Dynamics of a single zero]\label{dynam}  Let $t_0 \in \R$, and let $(z_k(t_0))_{k \in \Z \backslash \{0\}}$ be an enumeration of the zeroes of $H_{t_0}$ in $\C$ (counting multiplicity), with the symmetry condition $z_{-k}(t_0) = -z_k(t_0)$.
\begin{itemize}
\item[(i)]  If $j \in \Z \backslash \{0\}$ is such that $z_j(t_0)$ is a simple zero of $H_{t_0}$, then there exists a neighbourhood $U$ of $z_j(t_0)$, a neighbourhood $I$ of $t_0$ in $\R$, and a smooth map $z_j: I \to U$ such that for every $t \in I$, $z_j(t)$ is the unique zero of $H_t$ in $U$.  Furthermore one has the equation
\begin{equation}\label{zjk}
 \frac{d z_j}{d t}(t_0) = 2 \sum^{\prime}_{k \neq j} \frac{1}{z_j(t_0) - z_k(t_0)} 
\end{equation}
where the sum is over those $k \in \Z \backslash \{0\}$ with $k \neq j$, and the prime means that the $k$ and $-k$ terms are summed together (except for the $k=-j$ term, which is summed separately) in order to make the sum convergent.
\item[(ii)]  If $j \in \Z \backslash \{0\}$ is such that $z_j(t_0)$ is a repeated zero of $H_{t_0}$ of order $m \geq 2$, then there is a neighbourhood $U$ of $z_j(t_0)$ such that for $t$ sufficiently close to $t_0$, there are precisely $m$ zeroes of $H_t$ in $U$, and they take the form
$$ z_j(t_0) + \sqrt{2} (t-t_0)^{1/2} \lambda_j + O( |t-t_0|)$$
for $j=1,\dots,m$ as $t \to t_0$, where $\lambda_1 < \dots < \lambda_m$ are the roots of the $m^{\operatorname{th}}$ Hermite polynomial
\begin{align}
\operatorname{He}_m(z) &\coloneqq (-1)^m \exp\left(\frac{z^2}{2}\right) \frac{d^m}{dz^m} \exp\left(-\frac{z^2}{2}\right)\label{heform}\\
&= \sum_{0 \leq l \leq m/2} \frac{m!}{l! (m-2l)!} (-1)^l \frac{z^{m-2l}}{2^l}\label{heform2}
\end{align}
and the implied constant in the $O()$ notation can depend on $t_0$, $j$, and $m$.
\end{itemize}
\end{proposition}

The differential equation \eqref{zjk} was previously derived in \cite[Lemma 2.4]{csv} in the case $t > \Lambda$ (in which all zeroes are real and simple); however, in our applications we also need to consider the regime $t \leq \Lambda$ in which the zeroes are permitted to be complex or repeated.
The roots $\lambda_1,\dots,\lambda_m$ appearing in Proposition \ref{dynam}(ii) can be given explicitly for small values of $m$ as
$$ \lambda_1 = -1; \quad \lambda_2 = +1$$
when $m=2$,
$$ \lambda_1 = -\sqrt{3}; \quad x_2 = 0; \quad \lambda_3 = +\sqrt{3}$$
when $m=3$, and
$$ x\lambda1 = -\sqrt{3+\sqrt{6}}; \quad \lambda_2 = -\sqrt{3-\sqrt{6}}; \quad \lambda_3 = \sqrt{3 - \sqrt{6}}; \quad  \lambda_4 = \sqrt{3 + \sqrt{6}}$$
when $m=4$.  From \eqref{heform} and iterating Rolle's theorem we see that all the roots $\lambda_1,\dots,\lambda_m$ of $\operatorname{He}_m$ are real; from the Hermite equation $\left(\frac{d^2}{dz^2} - z \frac{d}{dz} + m\right) \operatorname{He}_m(z) = 0$ and the Picard uniqueness theorem for ODE we see that the zeroes are all simple.

\begin{proof}  First suppose we are in the situation of (i).  As $z_j(t_0)$ is simple, $\frac{\partial}{\partial z} H_t$ is non-zero at $z_j(t_0)$; since $H_t(z)$ is a smooth function of both $t$ and $z$, we conclude from the implicit function theorem that there is a unique solution $z_j(t) \in U$ to the equation
$$ H_t( z_j(t) ) = 0$$
with $z_j(t)$ in a sufficiently small neighbourhood $U$ of $z_j(t_0)$, if $t$ is in a sufficiently small neighbourhood $I$ of $t_0$; furthermore, $z_j(t)$ depends smoothly on $t$, and agrees with $z_j(t_0)$ when $t=t_0$.  Differentiating the above equation at $t_0$, we obtain
$$ \frac{\partial H_t}{\partial t}|_{t=t_0}( z_j(t_0) ) + \frac{d z_j}{d t}(t_0) H'_{t_0}(z_j(t_0)) = 0,$$
where the primes denote differentiation in the $z$ variable.
On the other hand, from \eqref{htdef} and differentiation under the integral sign (which can be justified using the rapid decrease of $\Phi$) we have the backwards heat equation
\begin{equation}\label{back}
\frac{\partial H_t}{\partial t} = -H''_t
\end{equation}
for all $t \geq 0$.  Inserting this into the previous equation, we conclude that
\begin{equation}\label{tzj}
\frac{d z_j}{d t}(t_0)  = \frac{H''_t}{H'_t}( z_j(t_0) ),
\end{equation}
noting that the denominator $H'_t(z_j(t_0))$ is non-vanishing by the hypothesis that the zero at $z_j(t_0)$ is simple.  Henceforth we omit the dependence on $t_0$ for brevity.  From Taylor expansion of $H_t$, $H'_t$, and $H''_t$ around the simple zero $z_j$ we see that
\begin{equation}\label{htz-eq}
 \frac{H''_t}{H'_t}( z_j) = 2 \lim_{z \to z_j}\left(  \frac{H'_t}{H_t}( z) - \frac{1}{z-z_j} \right).
\end{equation}
On the other hand, as $H_t$ is even, non-zero at the origin (as follows from \eqref{htdef} and the positivity of $\Phi$), and entire of order $1$, we see from the Hadamard factorization theorem that
$$ H_t(z) = H_t(0) \prod_k^{\prime} \left(1 - \frac{z}{z_k}\right),$$
where the prime indicates that the $k$ and $-k$ factors are multiplied together.  The product is locally uniformly convergent, so we may take logarithmic derivatives and conclude that
$$ \frac{H'_t}{H_t}(z) = \sum_k^{\prime} \frac{1}{z-z_k}.$$
Inserting this into \eqref{tzj}, \eqref{htz-eq} and using the continuity of $z \mapsto \sum_{k: k \neq j}^\prime \frac{1}{z-z_k}$ at $z_j$ (which follows from the growth in the number of zeroes, either from the dominated convergence theorem or the Weierstrass $M$-test), we obtain the claim (i).

Now we prove (ii).  We abbreviate $z_j(t_0)$ as $z_j$.  By Taylor expansion we have
$$ \frac{\partial^{2k} H_{t_0}}{\partial z^{2k}}(z) = m (m-1) \dots (m-2k+1) a_m (z-z_j)^{m-2k} + O( |z-z_j|^{\max(m-2k+1,0)} )$$
as $z \to z_j$ for any fixed integer $k \geq 0$ and some non-zero complex number $a_m = a_m(z_j, t_0)$ (with the implied constant in the $O()$ notation allowed to depend on $k$, $z_j$, $t_0$); applying the backwards heat equation \eqref{back} we thus have
$$ \frac{\partial^k H_t}{\partial t^k}|_{t=t_0}(z) = (-1)^k m (m-1) \dots (m-2k+1) a_m (z-z_j)^{m-2k} + O( |z-z_j|^{\max(m-2k+1,0)} ).$$
Performing Taylor expansion in time and using \eqref{heform2}, we conclude that in the regime $z - z_j = O( |t-t_0|^{1/2} )$, one has the bound
$$ H_t(z) = 2^{\frac{m}{2}} a_m ((t-t_0)^{1/2})^m \left( \operatorname{He}_m\left( \frac{z-z_j}{\sqrt{2} (t-t_0)^{1/2}} \right) + O\left( |t-t_0|^{1/2} \right) \right)$$
as $t \to t_0$, using (say) the standard branch of the square root.  By the inverse function theorem (and the simple nature of the zeroes of $\operatorname{He}_m$), we conclude that for $t$ sufficiently close but not equal to $t_0$, we have $m$ zeroes of $H_t$ of the form
$$ z_j + \sqrt{2} (t-t_0)^{1/2} \lambda_j + O( |t-t_0| ).$$
By Rouche's theorem, if $U$ is a sufficiently small neighborhood of $z_j$ then these are the only zeroes of $H_t$ in $U$ for $t$ sufficiently close to $t_0$.  The claim follows.
\end{proof}

Next, we recall the following bound of de Bruijn:

\begin{theorem}\label{debr-bound}  Suppose that $t_0 \in \R$ and $y_0 > 0$ is such that there are no zeroes $H_{t_0}(x+iy)=0$ with $x \in \R$ and $y > y_0$.  Then for any $t>t_0$, there are no zeroes $H_{t}(x+iy)=0$ with $x \in \R$ and $y > \max( y_0^2 - 2(t-t_0), 0)^{1/2}$.  In particular one has $\Lambda \leq t_0 + \frac{1}{2} y_0^2$.
\end{theorem}

\begin{proof} See \cite[Theorem 13]{debr}.
\end{proof}

We are now ready to prove Theorem \ref{ubc-0}.  The main step is to establish

\begin{proposition}[Zero-free region criterion]\label{ubc}  Suppose that $t_0, X > 0$ and $0 < y_0 \leq 1$ obey the following hypotheses:
\begin{itemize}
\item[(i)]  There are no zeroes $H_0(x+iy)=0$ with $0 \leq x \leq X$ and $\sqrt{y_0^2+2 t_0} \leq y \leq 1$.
\item[(ii)]  There are no zeroes $H_{t_0}(x+iy)=0$ with $x \geq X+\sqrt{1-y_0^2}$ and $y_0 \leq y \leq \sqrt{1-2t_0}$.
\item[(iii)]  There are no zeroes $H_{t}(x+iy)=0$ with $X \leq x \leq X+\sqrt{1-y_0^2}$, $\sqrt{y_0^2 + 2(t_0-t)} \leq y \leq \sqrt{1-2t}$, and $0 \leq t \leq t_0$.
\end{itemize}
Then there are no zeroes $H_{t_0}(x+iy) = 0$ with $x \in \R$ and $y \geq y_0$.
\end{proposition}

\begin{proof}  It is well known that the Riemann $\xi$ function has no zeroes outside of the strip $\{ 0 \leq \Re s \leq 1 \}$, hence there are no zeroes $H_0(x+iy)=0$ with $y > 1$.  By Theorem \ref{debr-bound}, we may thus remove the upper bound constraints $y \leq 1$, $y \leq \sqrt{1-2t_0}$, and $y \leq \sqrt{1-2t}$ from (i), (ii), and (iii) respectively.

By hypotheses (ii), (iii) and the symmetry properties of $H_t$, it suffices to show that for every $0 \leq t \leq t_0$, there are no zeroes $H_t(x+iy) = 0$ with $0 \leq x \leq X$ and $y \geq Y(t)$, where $Y(t) \coloneqq \sqrt{y_0^2 + 2(t_0-t)}$.  By hypothesis (i), this is true at time $t=0$.  Suppose the claim failed for some time $0 < t \leq t_0$.  Let $t_1 \in (0,t_0]$ be the minimal time in which this occurred (such a time exists because $H_t$ varies continuously in $t$, and there are no zeroes $H_t(x+iy)=0$ with (say) $y>1$).  From Rouche's theorem (or Proposition \ref{dynam}) we conclude that there is a zero $H_{t_1}(x+iy)=0$ with $x+iy$ on the boundary of the region $\{ x+iy: 0 \leq	 x \leq X, y \geq Y(t_1) \}$.  The right side $x=X$ of this boundary is ruled out by hypothesis (iii), and (as mentioned at the start of the section) the left side $x=0$ is ruled out by \eqref{htdef} and the positivity of $\Phi$.  Thus by the symmetry properties of $H_{t_1}$ we must have
$$ H_{t_1}(x+iY(t_1)) = 0$$
for some $0 < x < X$.

Suppose first that $H_{t_1}$ has a repeated zero at $x+iy_0$.  Using Proposition \ref{dynam}(ii) and observing (from the symmetry of $\operatorname{He}_m$) that at least one of the roots $x_1,\dots,x_m$ is positive, we then see that for $t<t_1$ sufficiently close to $t_1$, $H_t$ has a zero in the region $\{ x+iy: 0 \leq x \leq X, y \geq Y(t) \}$, contradicting the minimality of $t_1$.  Thus the zero $x+i Y(t_1)$ of $H_{t_1}$ must be simple.  In particular, by Proposition \ref{dynam}(i) we can write $x+i Y(t_1) = z_j(t_1)$ for some smooth function $z_j$ in a neighbourhood of $t_1$ obeying \eqref{zjk}, such that $z_j(t)$ is a zero of $H_t$ for all $t$ close to $t_1$.  We will prove that
\begin{equation}\label{im}
\Im \frac{d}{d t} z_j( t_1 ) < \frac{d}{dt} Y(t_1),
\end{equation}
which implies that there is a zero of $H_t$ in the region $\{ x+iy: 0 \leq x \leq X, y \geq Y(t) \}$  for $t<t_1$ sufficiently close to $t_1$, giving the required contradiction.  

The right-hand side of \eqref{im} is
$$ \frac{d}{d t} Y(t_1) = -\frac{1}{Y(t_1)}.$$
By Proposition \ref{dynam}(i), the left-hand side of \eqref{im} is
$$ -2 \sum_{k \neq j}^{\prime} \frac{Y(t_1) - y_k}{(x-x_k)^2 + (Y(t_1)-y_k)^2}$$
where we write $z_k = x_k + i y_k$.  Clearly any zero $x_k+iy_k$ with imaginary part $y_k$ in $[-Y(t_1),Y(t_1)]$ gives a non-positive contribution to this sum, the contribution of the zero $x- iY(t_1)$ is $-\frac{1}{Y(t_1)}$, the contribution of the zero $-x+iY(t_1)$ vanishes, and the contribution of $-x-iY(t_1)$ is negative.  Grouping the remaining zeroes with their complex conjugates, it then suffices to show that
$$ \frac{Y(t_1) - y_k}{(x-x_k)^2 + (Y(t_1)-y_k)^2} - \frac{Y(t_1) + y_k}{(x-x_k)^2 + (Y(t_1)+y_k)^2} \leq 0$$
whenever $y_k > Y(t_1)$.  Cross-multiplying and canceling like terms, this inequality eventually simplifies to
$$ y_k^2 \leq (x-x_k)^2 + Y(t_1)^2.$$
But from the hypothesis (iii) and the assumption $y_k > Y(t_1)$, we have $|x_k| \geq X+\sqrt{1-Y(t_1)^2}$, so $(x-x_k)^2 \geq 1-Y(t_1)^2$.  On the other hand from Theorem \ref{debr-bound} one has $y_k < 1$, giving the required contradiction.
\end{proof}

By combining Proposition \ref{ubc} with Theorem \ref{debr-bound}, we obtain Theorem \ref{ubc-0}, noting from \eqref{hoz}, \eqref{sas} that condition (i) of Proposition \ref{ubc} is implied by condition (i) of Theorem \ref{ubc-0}.

\section{Applying the fundamental solution for the heat equation}\label{heatflow-sec}

As discussed in the introduction, we will establish Theorem \ref{eff} by writing $H_t$ in terms of $H_0$ using the fundamental solution to the heat equation.  Namely, for any $t>0$, we have from \eqref{gaussian} that
$$
e^{tu^2} = \int_\R e^{\pm 2 \sqrt{t} vu} \frac{1}{\sqrt{\pi}} e^{-v^2}\ dv$$
for any complex $u$ and either choice of sign $\pm$. Multiplying by $e^{\pm i zu}$ and averaging, we conclude that
$$
e^{tu^2} \cos(zu) = \int_\R \cos\left(\left(z - 2 i \sqrt{t} v\right)u\right) \frac{1}{\sqrt{\pi}} e^{-v^2}\ dv$$
for any complex $z,u$.  Multiplying by $\Phi(u)$ and using Fubini's theorem, we conclude the heat kernel representation
$$ H_t(z) = \int_\R H_0( z - 2i \sqrt{t} v ) \frac{1}{\sqrt{\pi}} e^{-v^2}\ dv $$
for any complex $z$.  Using \eqref{hoz}, we thus have
\begin{equation}\label{htz}
 H_t(z) = \int_\R \frac{1}{8} \xi\left( \frac{1+iz}{2} + \sqrt{t} v \right) \frac{1}{\sqrt{\pi}} e^{-v^2}\ dv.
\end{equation}

\begin{remark}  We have found numerically that the formula \eqref{htz} gives a fast and accurate means to compute $H_t(z)$ when $z$ is of moderate size, e.g., if $z = x+iy$ with $|x| \leq 10^6$ and $|y| \leq 1$.  However, we will not need to directly compute the right-hand side of \eqref{htz} for our application to bounding $\Lambda$, as we will only need to control $H_t(x+iy)$ for large values of $x$, and we will shortly develop tractable approximations of Riemann-Siegel type that are more suitable for this regime.
\end{remark}

We now combine this formula with expansions of the Riemann $\xi$-function.  From \cite[(2.10.6)]{titch} we have the Riemann-Siegel formula
\begin{equation}\label{xio}
 \frac{1}{8} \xi(s) = R_{0,0}(s) + R_{0,0}^*(1-s) 
\end{equation}
for any complex $s$ that is not an integer (in order to avoid the poles of the Gamma function), where $R_{0,0}(s)$ is the contour integral
$$ R_{0,0}(s) := \frac{1}{8} \frac{s(s-1)}{2} \pi^{-s/2} \Gamma\left(\frac{s}{2}\right) \int_{0 \swarrow 1} \frac{w^{-s} e^{i\pi w^2}}{e^{\pi i w} - e^{-\pi i w}}\ dw$$
with $0 \swarrow 1$ any infinite line oriented in the direction $e^{5\pi i/4}$ that crosses the interval $[0,1]$.  From the residue theorem (and the gaussian decrease of $e^{i\pi w^2}$ along the $e^{\pi i/4}$ and $e^{5\pi i/4}$ directions) we may expand
$$ R_{0,0}(s) = \sum_{n=1}^N r_{0,n}(s)+ R_{0,N}(s)$$
for any non-negative integer $N$, where $r_{0,n}, R_{0,N}$ are the meromorphic functions
\begin{align}
 r_{0,n}(s) &:= \frac{1}{8} \frac{s(s-1)}{2} \pi^{-s/2} \Gamma\left(\frac{s}{2}\right) n^{-s},\label{ron-def} \\
R_{0,N}(s) &:= \frac{1}{8} \frac{s(s-1)}{2} \pi^{-s/2} \Gamma\left(\frac{s}{2}\right) \int_{N \swarrow N+1} \frac{w^{-s} e^{i\pi w^2}}{e^{\pi i w} - e^{-\pi i w}}\ dw\label{RON-def}
\end{align}
and $N \swarrow N+1$ denotes any infinite line oriented in the direction $e^{5\pi i /4}$ that crosses the interval $[N,N+1]$.  For any $z$ that is not purely imaginary, we see from Stirling's approximation that the functions $r_{0,n}(\frac{1+iz}{2} + \sqrt{t} v)$ and $R_{0,N}(\frac{1+iz}{2} + \sqrt{t} v)$ grow slower than gaussian as $v \to \pm \infty$ (indeed they grow like $\exp( O( |v| \log |v| ) )$, where the implied constants depend on $t,z$).  From this and
\eqref{htz}, \eqref{xio} we conclude that
\begin{equation}\label{htz-expand}
 H_t(z) = \sum_{n=1}^N r_{t,n}\left(\frac{1+iz}{2}\right) + \sum_{n=1}^N r_{t,n}^*\left(\frac{1-iz}{2}\right) + R_{t,N}\left(\frac{1+iz}{2}\right) + R_{t,N}^*\left(\frac{1-iz}{2}\right)
\end{equation}
for any $t>0$, any $z$ that is not purely imaginary, and any non-negative integer $N$, where $r_{t,n}(s), R_{t,N}(s)$ are defined for non-real $s$ by the formulae
\begin{align*}
 r_{t,n}(s) &:= \int_\R r_{0,n}\left( s + \sqrt{t} v \right) \frac{1}{\sqrt{\pi}} e^{-v^2}\ dv\\
 R_{t,N}(s) &:= \int_\R R_{0,N}\left( s + \sqrt{t} v \right) \frac{1}{\sqrt{\pi}} e^{-v^2}\ dv;
\end{align*}
these can be thought of as the evolutions of $r_{0,n}, R_{0,N}$ respectively under the forward heat equation.

The functions $r_{0,n}(s), R_{0,N}(s)$ grow slower than gaussian as long as the imaginary part of $s$ is bounded and bounded away from zero.  As a consequence, we may shift contours (replacing $v$ by $v + \frac{\sqrt{t}}{2} \alpha_n$) and write
\begin{equation}\label{rtn-def}
 r_{t,n}(s) = \exp\left( - \frac{t}{4} \alpha_n^2\right) \int_\R \exp\left( - \sqrt{t} v \alpha_n\right) r_{0,n}\left( s + \sqrt{t} v + \frac{t}{2} \alpha_n\right) \frac{1}{\sqrt{\pi}} e^{-v^2}\ dv
\end{equation}
for any complex number $\alpha_n$ with $\Im(s), \Im(s + \frac{t}{2} \alpha_n)$ having the same sign.  Similarly we may write
\begin{equation}\label{RTN-def}
 R_{t,N}(s) = \exp\left( - \frac{t}{4} \beta_N^2\right) \int_\R \exp\left( - \sqrt{t} v \beta_N\right) R_{0,N}\left( s + \sqrt{t} v + \frac{t}{2} \beta_N\right) \frac{1}{\sqrt{\pi}} e^{-v^2}\ dv
\end{equation}
for any complex number $\beta_N$ with $\Im s, \Im(s + \frac{t}{2} \beta_N)$ having the same sign.  In the spirit of the saddle point method, we will select the parameters $\alpha_n, \beta_N$ later in the paper in order to make the integrands in \eqref{rtn-def}, \eqref{RTN-def} close to stationary in phase at $v=0$, in order to obtain good estimates and approximations for these terms.

\section{Elementary estimates}

In order to explicitly estimate various error terms arising in the proof of Theorem \ref{eff}, we will need the following elementary estimates:

\begin{lemma}[Elementary estimates]\label{elem-lem} Let $x > 0$.
\begin{itemize}
\item[(i)] If $a > 0$ and $b \geq 0$ are such that $x > b/a$, then
$$O_{\leq}\left(\frac{a}{x}\right) + O_{\leq}\left( \frac{b}{x^2}\right ) = O_{\leq}\left( \frac{a}{x-b/a} \right).$$
More generally, if $a > 0$ and $b,c \geq 0$ are such that $x > b/a, \sqrt{c/a}$, then
$$O_{\leq}\left(\frac{a}{x}\right) + O_{\leq}\left( \frac{b}{x^2} \right) + O_{\leq}\left( \frac{c}{x^3}\right) = O_{\leq}\left( \frac{a}{x-\max(b/a,\sqrt{c/a})} \right).$$
\item[(ii)]  If $x > 1$, then
$$\log\left(1 + O_{\leq}\left(\frac{1}{x}\right) \right) = O_{\leq}\left(\frac{1}{x-1}\right).$$
or equivalently
$$1 + O_{\leq}\left(\frac{1}{x}\right) = \exp\left( O_{\leq}\left(\frac{1}{x-1}\right) \right).$$
\item[(iii)]  If $x > 1/2$, then
$$\exp\left( O_{\leq}\left(\frac{1}{x}\right) \right) = 1 + O_{\leq}\left( \frac{1}{x-0.5} \right).$$
\item[(iv)]  We have
$$ \exp\left(O_{\leq}(x)\right) = 1 + O_{\leq}(e^x-1).$$
\item[(v)] If $z$ is a complex number with $|\Im z| \geq 1$ or $\Re z \geq 1$, then
$$ \Gamma(z) = \sqrt{2\pi} \exp\left( \left(z-\frac{1}{2}\right) \log z - z + O_{\leq}\left( \frac{1}{12(|z| - 0.33)} \right)\right).$$
\item[(vi)] If $a,b > 0, y \geq 0$ and $x \geq x_0 \geq \exp(a/b)$ and $x_0 > c \geq 0$, then
$$\frac{\log^a|x+iy|}{(x-c)^b} \leq \frac{\log^a |x_0+iy|}{(x_0-c)^b}.$$
\end{itemize}
\end{lemma}

\begin{proof}  Claim (i) follows from the geometric series formula
$$ \frac{a}{x-t} = \frac{a}{x} + \frac{at}{x^2} + \frac{at^2}{x^3} + \dots$$
whenever $0 \leq t < x$.

For Claim (ii), we use the Taylor expansion of the logarithm to note that
$$\log\left( 1 + O_{\leq}\left(\frac{1}{x}\right) \right) = O_{\leq}\left(\frac{1}{x} + \frac{1}{2x^2} + \frac{1}{3x^3} + \dots\right)$$
which on comparison with the geometric series formula
$$\frac{1}{x-1} = \frac{1}{x} + \frac{1}{x^2} + \frac{1}{x^3} + \dots$$
gives the claim.  Similarly for Claim (iii), we may compare the Taylor expansion
$$\exp\left( O_{\leq}\left(\frac{1}{x}\right) \right) = 1 + O_{\leq}\left(\frac{1}{x} + \frac{1}{2! x^2} + \frac{1}{3! x^3} + \dots\right)$$
with the geometric series formula
$$ \frac{1}{x-0.5} = \frac{1}{x} + \frac{1}{2x^2} + \frac{1}{2^2 x^3} + \dots$$
and note that $k! \geq 2^k$ for all $k \geq 2$.

Claim (iv) follows from the trivial identity $e^x = 1 + (e^x-1)$ and the elementary inequality $e^{-x} \geq 1 - (e^x-1)$.
For Claim (v), we may use the functional equation $\Gamma = \Gamma^*$ to assume that $\Im z \geq 0$.  From the work of Boyd \cite[(1.13), (3.1), (3.14), (3.15)]{boyd} we have the effective Stirling approximation
$$ \Gamma(z) = \sqrt{2\pi} \exp\left( \left(z-\frac{1}{2}\right) \log z - z \right) \left(1 + \frac{1}{12 z} + R_2(z) \right)$$
where the remainder $R_2(z)$ obeys the bound
$$ |R_2(z)| \leq (2 \sqrt{2}+1) \frac{C_2 \Gamma(2)}{(2\pi)^3 |z|^2} $$
for $\Re z \geq 0$ and
$$|R_2(z)| \leq (2 \sqrt{2}+1) \frac{C_2 \Gamma(2)}{(2\pi)^3 |z|^2 |1 - e^{2\pi i z}|} $$
for $\Re z \leq 0$, where $C_2$ is the constant
$$ C_2 := \frac{1}{2} (1 + \zeta(2)) = \frac{1}{2} \left(1 + \frac{\pi^2}{6}\right).$$
In the latter case, we have $\Im z \geq 1$ by hypothesis, and hence $|1 - e^{2\pi i z}| \geq 1 - e^{-2\pi}$.  We conclude that in all ranges of $z$ of interest, we have
$$|R_2(z)| \leq (2 \sqrt{2}+1) \frac{C_2 \Gamma(2)}{(2\pi)^3 |z|^2 (1 - e^{-2\pi})} \leq \frac{0.0205}{|z|^2}$$
and hence by Claim (i)
$$ \Gamma(z) = \sqrt{2\pi} \exp\left( \left(z-\frac{1}{2}\right) \log z - z \right) \left(1 + O_{\leq}\left( \frac{1}{12(|z| - 0.246)} \right)\right) $$
and the claim then follows by Claim (ii).  

For Claim (vi), it suffices to show that the function $x \mapsto \frac{\log^a |x+iy|}{(x-c)^b}$ is non-increasing for $x \geq \exp(a/b)$.  Since $\log |x+iy| = (\log x) (1 + \frac{\log(1 + \frac{y^2}{x^2})}{2 \log x})$ and the second factor is monotone decreasing in $x$, it suffices to show that $x \mapsto \frac{\log^a x}{(x-c)^b}$ is non-increasing in this region.  Taking logarithms and differentiating, we wish to show that $\frac{a}{x \log x} - \frac{b}{x-c} \leq 0$.  But this is clear since $\frac{b}{x-c} \geq \frac{b}{x}$ and $\log x \geq a/b$.
\end{proof}

\section{Proof of Theorem \ref{eff}}\label{initial-sec}

In this section we establish Theorem \ref{eff}.  
The strategy is to use the expansion \eqref{htz-expand}, which turns out to be an effective approximation in the region \eqref{region}, since we will be able to ensure that quantities such as $s + \sqrt{t} v + \frac{t}{2} \alpha_n$ or $s + \sqrt{t} v + \frac{t}{2} \beta_N$, with $s = \frac{1+i(x+iy)}{2}$, stay away from the real axis where the poles of $\Gamma$ are located (and also where the error terms in the Riemann-Siegel approximation deteriorate).  

Accordingly, we will need effective estimates on the functions $r_{t,n}, R_{t,N}$ appearing in Section \ref{heatflow-sec}.  We will treat these two functions separately.

\subsection{Estimation of $r_{t,n}$}

We recall the function $\alpha(s)$ defined in \eqref{alpha-def}.  From differentiating \eqref{alpha-form} we see that
\begin{equation}\label{alpha-deriv}
 \alpha'(s) = -\frac{1}{2s^2} - \frac{1}{(s-1)^2} + \frac{1}{2 s}
\end{equation}
whenever $s \in \C \backslash (-\infty,1]$.  If $\Im s > 3$, we conclude in particular the useful bound
\begin{equation}\label{alpha-deriv-bound}
\begin{split}
 \alpha'(s) &= O_{\leq}\left( \frac{1}{2\Im(s)^2} \right) + O_{\leq}\left( \frac{1}{\Im(s)^2} \right) + O_{\leq}\left( \frac{1}{2\Im(s)} \right) \\
 &= O_{\leq}\left( \frac{1}{2 \Im(s)-6} \right)
\end{split}
\end{equation}
thanks to Lemma \ref{elem-lem}(i).  

We also recall the function $M_t$ and the coefficients $b_n^t$ from \eqref{Mt-def}, \eqref{bn-def} respectively.  It turns out we have a good approximation
$$r_{t,n}(\sigma+iT) \approx M_t(\sigma+iT) \frac{b_n^t}{n^{\sigma+iT+\frac{t}{2} \alpha(\sigma+iT)}}.$$
More precisely, we have

\begin{proposition}[Estimate for $r_{t,n}$]\label{rtn-prop}  Let $\sigma$ be real, let $T>10$, let $n$ be a positive integer, and let $0 < t \leq 1/2$.  Then 
$$ r_{t,n}(\sigma+iT) = M_t(\sigma+iT) \frac{b_n^t}{n^{\sigma+iT+\frac{t}{2} \alpha(\sigma+iT)}} \left(1 + O_{\leq}(\eps_{t,n}(\sigma+iT))\right)$$
where 
\begin{equation}\label{eps-def}
 \eps_{t,n}(\sigma+iT) \coloneqq \exp\left( \frac{\frac{t^2}{8} |\alpha(\sigma+iT) - \log n|^2 + \frac{t}{4} + \frac{1}{6}}{T-3.33} \right)-1.
\end{equation}
\end{proposition}

\begin{proof}  From \eqref{ron-def}, \eqref{M-def} and Lemma \ref{elem-lem}(v) one has
$$ 
r_{0,n}(s) = M_0(s) n^{-s} \exp\left( O_{\leq}\left( \frac{1}{6(|s|-0.66)} \right) \right)
$$
whenever $\Im s > 2$.  Let $\alpha_n$ denote the quantity
\begin{equation}\label{alphan-def}
\alpha_n \coloneqq \alpha(\sigma+iT) - \log n;
\end{equation}
this is the logarithmic derivative of $M(s) n^{-s}$ at $s=\sigma+iT$.  
By \eqref{alpha-form} and the hypothesis $T \geq 10$, the imaginary part of $\alpha_n$ may be lower bounded by
\begin{equation}\label{iman}
 \Im \alpha_n \geq -\frac{1}{2T} - \frac{1}{T} \geq -0.15;
\end{equation}
in particular, $\sigma+iT$ and $\sigma+iT+\frac{t}{2}\alpha_n$ have imaginary parts of the same sign.
We can now apply \eqref{rtn-def} to obtain
\begin{align*}
 r_{t,n}(\sigma+iT) &= \exp\left( - \frac{t}{4} \alpha_n^2\right) \int_\R \exp\left( - \sqrt{t} v \alpha_n\right) M_0\left( \sigma + iT + \sqrt{t} v + \frac{t}{2} \alpha_n\right) \times \\
&\quad \times \exp\left( -\left(\sigma+iT+\sqrt{t} v + \frac{t}{2} \alpha_n \right) \log n + O_{\leq}\left( \frac{1}{6(|\sigma+iT+\sqrt{t} v + \frac{t}{2} \alpha_n|-0.66)} \right) \right)
\frac{1}{\sqrt{\pi}} e^{-v^2}\ dv.
\end{align*}
From \eqref{iman} we see that $\sigma+iT+\sqrt{t}v+\frac{t}{2}\alpha_n$ has imaginary part at least $T-0.08$. Thus
$$ O_{\leq}\left( \frac{1}{6(|\sigma+iT+\sqrt{t} v + \frac{t}{2} \alpha_n|-0.66)} \right)  = 
O_{\leq}\left( \frac{1}{6(T - 0.74)} \right) = O_{\leq}\left( \frac{1}{6(T - 3.08)} \right).$$
From \eqref{alpha-deriv-bound} we have
$$ \alpha'(s) = O_{\leq}\left( \frac{1}{2(T-3.08)} \right)$$
for all $s$ on the line segment between $\sigma+iT$ and $\sigma + iT + \sqrt{t} v + \frac{t}{2} \alpha_n$.  Applying Taylor's theorem with remainder to the branch of the logarithm $\log M_0$ defined in \eqref{logM}, we conclude that
$$ M_0( \sigma + iT + \sqrt{t} v + \frac{t}{2} \alpha_n) = M_0(\sigma+iT) \exp\left( \alpha(\sigma+iT) (\sqrt{t} v + \frac{t}{2} \alpha_n) + O_{\leq}\left( \frac{|\sqrt{t} v + \frac{t}{2} \alpha_n|^2}{4(T-3.08)}  \right)\right).$$
Combining these estimates, writing $\alpha(\sigma+iT) = \alpha_n + \log n$, estimating $|\sqrt{t} v + \frac{t}{2} \alpha_n|^2$ by $2tv^2 + \frac{t^2}{2} |\alpha_n|^2$, and simplifying, we conclude that
\begin{align*}
 r_{t,n}(s) &= M_0(\sigma+iT) \exp\left( \frac{t}{4} \alpha_n^2 - (\sigma+iT) \log n\right)  \\
&\quad \times  \int_\R \exp\left(  O_{\leq}\left( \frac{\frac{t}{2}v^2 + \frac{t^2}{8} |\alpha_n|^2 + \frac{1}{6}}{T-3.08} \right) \right) \frac{1}{\sqrt{\pi}} e^{-v^2}\ dv.
\end{align*}
Using \eqref{alphan-def}, \eqref{Mt-def}, \eqref{bn-def} we see that
$$ M_0(\sigma+iT) \exp\left( \frac{t}{4} \alpha_n^2 - (\sigma+iT) \log n\right) = M_t(\sigma+iT) \frac{b_n^t}{n^{\sigma+iT+\frac{t}{2} \alpha(\sigma+iT)}} $$
and so it suffices to show that
$$  \int_\R \exp\left( O_{\leq}\left( \frac{\frac{t}{2}v^2 + \frac{t^2}{8} |\alpha_n|^2 + \frac{1}{6}}{T-3.08} \right) \right)\frac{1}{\sqrt{\pi}} e^{-v^2}\ dv = 1 + O\left(\exp\left( \frac{\frac{t^2}{8} |\alpha_n|^2 + \frac{t}{4} + \frac{1}{6}}{T-3.33} \right)-1\right).$$
Since $\frac{1}{\sqrt{\pi}} e^{-v^2}\ dv $ integrates to one, it suffices by Lemma \ref{elem-lem}(iv) to show that
$$
  \int_\R \exp\left( \frac{\frac{t}{2} v^2+ \frac{t^2}{8} |\alpha_n|^2 + \frac{1}{6}}{T-3.08} \right) \frac{1}{\sqrt{\pi}} e^{-v^2}\ dv \leq \exp\left( \frac{\frac{t^2}{8} |\alpha_n|^2 + \frac{t}{4} + \frac{1}{6}}{T-3.33} \right).
$$
Since $\frac{1}{T-3.08} < \frac{1}{T-3.33}$, we can remove the $\frac{t^2}{8} |\alpha_n|^2 + \frac{1}{6}$ terms from both sides and reduce to showing that
\begin{equation}\label{ax}
  \int_\R \exp\left( \frac{t v^2}{2(T-3.08)} \right) \frac{1}{\sqrt{\pi}} e^{-v^2}\ dv \leq \exp\left( \frac{t}{4(T-3.33)} \right).
	\end{equation}
Using \eqref{gaussian}, the left-hand side may be calculated exactly as
$$ \left(1 - \frac{t}{2(T-3.08)}\right)^{-1/2}.$$
Applying Lemma \ref{elem-lem}(ii) and using the hypotheses $t \leq 1/2$, $T \geq 10$, one has
$$ 1 - \frac{t}{2(T-3.08)} = \exp\left( O_{\leq}\left( \frac{t}{2(T-3.33)} \right)\right)$$
and the claim follows.
\end{proof}

\subsection{Estimation of $R_{t,N}$}

We begin with the following estimates of Arias de Reyna \cite{arias} on the term $\int_{N \swarrow N+1} \frac{w^{-s} e^{i\pi w^2}}{e^{\pi i w} - e^{-\pi i w}}$ appearing in \eqref{RON-def}:

\begin{proposition}\label{arias-prop}  Let $\sigma$ be real and $T'>0$, and define the quantities
\begin{align}
s &\coloneqq \sigma + iT' \label{s-def}\\
a &\coloneqq \sqrt{\frac{T'}{2\pi}} \label{a-def}\\
N &\coloneqq \lfloor a \rfloor \label{N-def}\\
p &\coloneqq 1 - 2(a-N) \label{p-def}\\
U &\coloneqq \exp\left( -i\left(\frac{T'}{2} \log \frac{T'}{2\pi} - \frac{T'}{2} - \frac{\pi}{8}\right) \right)\label{U-def}.
\end{align}
Let $K$ be a positive integer.  Then we have the expansion
$$ \int_{N \swarrow N+1} \frac{w^{-s} e^{i\pi w^2}}{e^{\pi i w} - e^{-\pi i w}}\ dw = (-1)^{N-1} U a^{-\sigma} \left(\sum_{k=0}^K \frac{C_k(p,\sigma)}{a^k} + RS_K(s)\right) $$
where $C_0(p,\sigma) = C_0(p)$ is independent of $\sigma$ and is given explicitly by the formula
\begin{equation}\label{C0-def}
C_0(p) \coloneqq \frac{e^{\pi i (\frac{p^2}{2} +\frac{3}{8})} - i \sqrt{2} \cos \frac{\pi p}{2}}{2 \cos(\pi p)}
\end{equation}
(removing the singularities at $p = \pm 1/2$), while for $k \geq 1$ the $C_k(p,\sigma)$ are complex numbers obeying the bounds
\begin{equation}\label{ck-bound-1}
|C_k(p,\sigma)| \leq \frac{\sqrt{2}}{2\pi} \frac{9^\sigma \Gamma(k/2)}{2^k}
\end{equation}
for $\sigma>0$ and
\begin{equation}\label{ck-bound-2}
|C_k(p,\sigma)| \leq \frac{2^{\frac{1}{2}-\sigma}}{2\pi} \frac{\Gamma(k/2)}{2\pi ((3-2\log 2)\pi)^{k/2}}
\end{equation}
for $\sigma \leq 0$, while the error term $RS_K(s)$ is a complex number obeying the bounds
\begin{equation}\label{rsk-bound-1}
|RS_K(s)| \leq \frac{1}{7} 2^{3\sigma/2} \frac{\Gamma((K+1)/2)}{(a/1.1)^{K+1}}
\end{equation}
for $\sigma \geq 0$, and
\begin{equation}\label{rsk-bound-2}
|RS_K(s)| \leq \frac{1}{2} \left(\frac{9}{10}\right)^{\lceil -\sigma \rceil} \frac{\Gamma((K+1)/2)}{(a/1.1)^{K+1}}
\end{equation}
if $\sigma < 0$ and $K + \sigma \geq 2$.
\end{proposition}

\begin{proof} This follows from \cite[Theorems 3.1, 4.1, 4.2]{arias} combined with \cite[(3.2), (5.2)]{arias}.  The dependence of $C_k(p,\sigma), k \geq 1$ on $\sigma$ and the dependence of $RS_K(s)$ on $s$ is suppressed in \cite{arias}, but can be discerned from the definitions of these quantities (and the related quantities $g(\tau,z), P_k(z) = P_k(z,\sigma), Rg_K(\tau,z)$) in \cite[(3.9), (3.10), (3.7), (3.6)]{arias}.
\end{proof}

Note that $p$ ranges in the interval $[-1,1]$.  One can show that 
\begin{equation}\label{cop}
|C_0(p)| \leq \frac{1}{2}
\end{equation}
for all $p \in [-1,1]$; this follows for instance from the $n=0$ case of \cite[Theorem 6.1]{arias}.  See also Figure \ref{c0p}.

\begin{figure}[ht!]
  \includegraphics[width=0.7\linewidth]{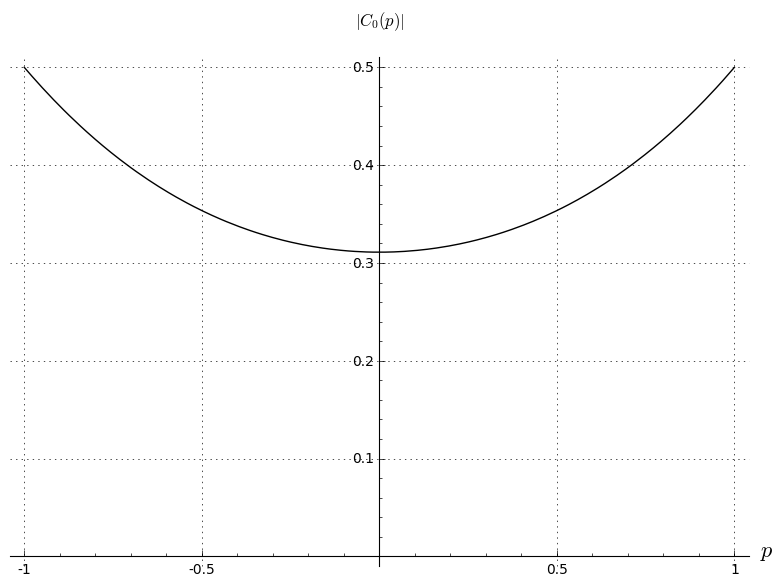}
  \caption{Plot of $|C_0(p)|$ for $-1 \leq p \leq 1$.}
	\label{c0p}
\end{figure}

Informally, the above proposition (and \eqref{RON-def}, \eqref{M-def}) yield the approximation
\begin{align*}
 R_{0,N}(s) &\approx \frac{1}{8} \frac{s(s-1)}{2} \pi^{-s/2} \Gamma\left(\frac{s}{2}\right) (-1)^{N-1} U a^{-\sigma} C_0(p) \\
&\approx (-1)^{N-1} U M_0(s) a^{-\sigma} C_0(p).
\end{align*}
If one writes $s = \sigma+iT$, then by using the approximation $\alpha(s) \approx \frac{1}{2} \log \frac{iT}{2\pi}$ for the log-derivative of $M_0$, one can then obtain the approximate formula
$$ R_{0,N}(s) \approx (-1)^{N-1} U e^{\pi i \sigma/4} M_0(iT) C_0(p).$$
In fact we have the more general approximation
$$ R_{t,N}(s) \approx (-1)^{N-1} U e^{\pi i \sigma/4} \exp\left( \frac{t \pi^2}{64}\right) M_0(iT') C_0(p)$$
where $T' \coloneqq T + \frac{\pi t}{8}$.  More precisely, we have

\begin{proposition}[Estimate for $R_{t,N}$]\label{RTN-prop}  Let $0 \leq \sigma \leq 1$, let $T \geq 100$, and let $0 < t \leq 1/2$.  Set
$$ T' \coloneqq T + \frac{\pi t}{8} $$
and then define $a,N,p,U,C_0(p)$ using \eqref{a-def}, \eqref{N-def}, \eqref{U-def}, \eqref{C0-def}.
Then 
$$
 R_{t,N}(\sigma+iT) = (-1)^{N-1} U e^{\pi i \sigma/4} \exp\left( \frac{t \pi^2}{64}\right) M_0(iT') \left( C_0(p) + O_{\leq}( \tilde \eps(\sigma+iT))\right)$$
where
\begin{equation}\label{epsp-def}
 \tilde \eps(\sigma+iT) \coloneqq \left(\frac{0.397 \times 9^\sigma}{a-0.865} + \frac{5}{3(T-6)}\right) \exp\left( \frac{3.49}{T-4} \right).
\end{equation}
\end{proposition}

\begin{proof}  We apply \eqref{RTN-def} with $\beta_N := \pi i/4$ to obtain
$$ R_{t,N}(\sigma+iT) = \exp\left( \frac{t \pi^2}{64}\right) \int_\R \exp\left( - \frac{\sqrt{t} v \pi i}{4}\right) R_{0,N}( \sigma+iT' + \sqrt{t} v) \frac{1}{\sqrt{\pi}} e^{-v^2}\ dv.$$
From \eqref{RON-def} we have
$$ R_{0,N}( \sigma+iT' + \sqrt{t} v) = \frac{1}{8} \frac{s_v(s_v-1)}{2} \pi^{-s_v/2} \Gamma\left(\frac{s_v}{2}\right) (-1)^{N-1} U a^{-\sigma-\sqrt{t} v}
\left(\sum_{k=0}^{K_v} \frac{C_k(p,\sigma + \sqrt{t} v)}{a^k} + RS_{K_v}(s_v)\right) $$
for any positive integer $K_v$ that we permit to depend (in a measurable fashion) on $v$, where $s_v \coloneqq \sigma + iT' + \sqrt{t} v$. From \eqref{M-def} and Lemma \ref{elem-lem}(v) we thus have
$$ R_{0,N}( \sigma+iT' + \sqrt{t} v) = M_0(s_v) \exp\left( O_{\leq}\left(\frac{1}{12(\frac{T'}{2}-0.33)}\right) \right) (-1)^{N-1} U a^{-\sigma-\sqrt{t} v}
\left(\sum_{k=0}^{K_v} \frac{C_k(p,\sigma + \sqrt{t} v)}{a^k} + RS_{K_v}(s_v)\right).$$
From \eqref{alpha-deriv-bound} and Taylor expansion of the logarithm $\log M_0$ defined in \eqref{logM}, we have
$$ M_0(s_v) = M_0(iT') \exp\left( \alpha(iT') (\sigma + \sqrt{t} v) + O_{\leq}\left( \frac{(\sigma + \sqrt{t} v)^2}{4(T-3)} \right) \right).$$
From \eqref{alpha-form}, \eqref{a-def} one has
$$ \alpha(iT') = O_{\leq}\left( \frac{1}{2T'}\right) + O_{\leq}\left( \frac{1}{T'}\right) + \frac{1}{2} \Log \frac{iT'}{2\pi} = \log a + \frac{i\pi}{4} + O_{\leq}\left( \frac{3}{2T'}\right)$$
and hence (bounding $\frac{3}{2T'}$ by $\frac{6}{4(T-3)}$)
$$ \alpha(iT') (\sigma + \sqrt{t} v)  = (\sigma + \sqrt{t} v)  \log a + \frac{\pi i \sigma}{4} + \frac{\sqrt{t} v \pi i}{4} + O_{\leq}\left( \frac{6 |\sigma+\sqrt{t} v|}{4(T-3)} \right).$$ 
We conclude (bounding $\frac{1}{12(\frac{T'}{2}-0.33)} \leq \frac{1/3}{4(T-3)}$) that
\begin{align*}
\exp\left( - \frac{\sqrt{t} v \pi i}{4}\right) R_{0,N}( \sigma+iT' + \sqrt{t} v) &= 
M_0(iT') \exp\left( O_{\leq}\left(\frac{(\sigma + \sqrt{t} v)^2+6|\sigma+\sqrt{t} v|+\frac{1}{3}}{4(T-3)} \right)\right) \times \\
&\quad \times (-1)^{N-1} U  e^{\pi i \sigma/4} \left(\sum_{k=0}^{K_v} \frac{C_k(p, \sigma+\sqrt{t} v)}{a^k} + RS_{K_v}(s_v)\right).
\end{align*}
Bounding $6|\sigma+\sqrt{t} v| \leq 3 (\sigma + \sqrt{t} v)^2 + 3$, we have
$$ \frac{(\sigma + \sqrt{t} v)^2+6|\sigma+\sqrt{t} v|+\frac{1}{3}}{4(T'-0.33)}  \leq \frac{(\sigma + \sqrt{t} v)^2 + \frac{5}{6}}{T-3}.$$
Putting all this together, we obtain
\begin{align*}
 R_{t,N}(\sigma+iT) &= (-1)^{N-1} U e^{\pi i \sigma/4} \exp\left( \frac{t \pi^2}{64}\right) M_0(iT') \times \\
&\quad \times \int_\R \exp\left( O_{\leq}\left(\frac{(\sigma + \sqrt{t} v)^2 + \frac{5}{6}}{T-3} \right) \right) \left(\sum_{k=0}^{K_v} \frac{C_k(p,\sigma + \sqrt{t} v)}{a^k} + RS_{K_v}(s_v)\right) \frac{1}{\sqrt{\pi}} e^{-v^2}\ dv.
\end{align*}
We separate the $k=0$ term from the rest.
By Lemma \ref{elem-lem}(iv) and the fact that $\frac{1}{\sqrt{\pi}} e^{-v^2}$ integrates to one, we can write the above expression as
\begin{equation}\label{rtnst}
 R_{t,N}(\sigma+iT) = (-1)^{N-1} U e^{\pi i \sigma/4} \exp\left( \frac{t \pi^2}{64}\right) M_0(iT') \left( C_0(p) (1 + O_{\leq}(\epsilon)) + O_{\leq}(\delta) \right)
\end{equation}
where
$$ \epsilon := \int_\R \left( \exp\left( \frac{(\sigma + \sqrt{t} v)^2 + \frac{5}{6} }{T-3} \right)  - 1\right) \frac{1}{\sqrt{\pi}} e^{-v^2}\ dv$$
and
$$ \delta := \int_\R \exp\left( \frac{(\sigma + \sqrt{t} v)^2 + \frac{5}{6}}{T-3} \right) \left(\sum_{k=1}^{K_v} \frac{|C_k(p,\sigma+\sqrt{t}v)|}{a^k} + |RS_{K_v}(s_v)|\right) \frac{1}{\sqrt{\pi}} e^{-v^2}\ dv.$$
Bounding $(\sigma + \sqrt{t} v)^2 \leq 2 \sigma^2 + 2 t v^2$ and using \eqref{gaussian} we obtain
$$ \epsilon \leq \exp\left( \frac{2\sigma^2 + \frac{5}{6}}{T-3} \right) \left(1 - \frac{2t}{T'-0.33}\right)^{-1/2} - 1.$$
Applying Lemma \ref{elem-lem}(ii) and using the hypotheses $t \leq 1/2$, $T \geq 100$, one has
$$ 1 - \frac{2t}{T-3} = \exp\left( O_{\leq}\left( \frac{2t}{T-6} \right)\right)$$
and hence
$$
 \epsilon \leq \exp\left( \frac{2\sigma^2 + t + \frac{5}{6}}{T-6} \right) - 1.
$$
With $t \leq 1/2$ and $0 \leq \sigma \leq 1$, one has $2\sigma^2 + t + \frac{5}{6} \leq \frac{10}{3}$.  By the mean value theorem we then have
\begin{equation}\label{eep}
 \epsilon \leq \frac{10}{3(T-6)} \exp\left( \frac{10}{3(T-6)}\right).
\end{equation}

Now we work on $\delta$.  Making the change of variables $u \coloneqq \sigma + \sqrt{t} v$, we have
$$ \delta = \int_\R \exp\left( \frac{u^2 + \frac{5}{6}}{T-3} \right) \left(\sum_{k=1}^{\tilde K_u} \frac{|C_k(p,u)|}{a^k} + |RS_{\tilde K_u}(u + iT')|\right) \frac{1}{\sqrt{\pi t}} e^{-(u-\sigma)^2/t}\ du,$$
where $\tilde K_u$ is a positive integer parameter that can depend arbitrarily on $u$ (as long as it is measurable, of course).  

We choose $\tilde K_u$ to equal $1$ when $u \geq 0$ and $\max( \lfloor -u \rfloor + 3, \lfloor \frac{T'}{\pi} \rfloor )$ when $u < 0$, so that Proposition \ref{arias-prop} applies.  The expression
$$ \sum_{k=1}^{\tilde K_u} \frac{|C_k(p,u)|}{a^k} + |RS_{\tilde K_u}(u + iT')| $$
is then bounded by
\begin{equation}\label{u0}
 \frac{\sqrt{2}}{2\pi} \frac{9^u \Gamma(1/2)}{2a} + \frac{1}{7} 2^{3u/2} \frac{\Gamma(1)}{(a/1.1)^2}
\leq \frac{0.200 \times 9^u}{a} + \frac{0.173 \times 2^{3u/2}}{a^2} 
\end{equation}
for $u \geq 0$ and
\begin{equation}\label{laf}
 \sum_{1 \leq k \leq \tilde K_u} \frac{2^{\frac{1}{2}-u}}{2\pi} \frac{\Gamma(k/2)}{2\pi ((3-2\log 2)\pi)^{k/2} a^k} + \frac{1}{2} (9/10)^{\lceil -u \rceil} \frac{\Gamma((\tilde K_u + 1)/2)}{(a/1.1)^{\tilde K_u + 1}}
\end{equation}
for $u < 0$.  One can calculate that
$$ \frac{2^{\frac{1}{2}}}{2\pi} \frac{1}{2\pi} \leq 0.036 \leq \frac{1}{2}$$
and
$$ \frac{1}{((3-2\log 2)\pi)^{1/2}} \leq 0.445 \leq 1.1$$
and hence we can bound \eqref{laf} by
$$  (0.036) 2^{-u} \sum_{1 \leq k \leq \frac{T'}{\pi}} (0.445)^k \frac{\Gamma(k/2)}{a^k} 
+ \frac{1}{2} 2^{-u} \sum_{\frac{T'}{\pi} \leq k \leq -u+4} \frac{\Gamma(k/2)}{(a/1.1)^k}.$$

For $u \geq 0$, we can estimate \eqref{u0} by
$$ 0.2 \times 9^u \left(\frac{1}{a} + \frac{0.865}{a^2}\right) \leq \frac{0.2 \times 9^u}{a - 0.865}$$
thanks to Lemma \ref{elem-lem}(i).  For $u<0$, we observe that if $k \leq 2 a^2 = \frac{T'}{\pi}$ then
$$ \frac{\Gamma(\frac{k+2}{2})}{a^{k+2}} = \frac{k}{2 a^2} \frac{\Gamma(k/2)}{a^k} \leq \frac{\Gamma(k/2)}{a^k}$$
and hence by the geometric series formula
$$ \sum_{2 \leq k \leq \frac{T'}{\pi}, k\ \mathrm{even}} (0.445)^k \frac{\Gamma(k/2)}{a^k}  \leq \frac{(0.445)^2}{1-(0.445)^2} \frac{\Gamma(2/2)}{a^2} \leq \frac{0.247}{a^2}$$
and similarly
$$ \sum_{3 \leq k \leq \frac{T'}{\pi}, k\ \mathrm{odd}} (0.445)^k \frac{\Gamma(k/2)}{a^k}  \leq \frac{(0.445)^3}{1-(0.445)^2} \frac{\Gamma(3/2)}{a^3} \leq \frac{0.098}{a^3}$$
and hence we can bound \eqref{laf} by
$$ (0.036) 2^{-u} \left(\frac{0.445 \sqrt{\pi}}{a} + \frac{0.247}{a^2} + \frac{0.098}{a^3}\right) + \frac{1}{2} 2^{-u}  \sum_{\frac{T'}{\pi} \leq k \leq -u+4} \frac{\Gamma(k/2)}{(a/1.1)^k}.$$
By Lemma \ref{elem-lem}(i) we have
$$ 0.036 \left(\frac{0.445 \sqrt{\pi}}{a} + \frac{0.247}{a^2} + \frac{0.098}{a^3}\right) \leq \frac{0.029}{a - 0.353}$$
and thus we can bound \eqref{laf} by
$$ \frac{0.029 \times 2^{-u}}{a - 0.353} + \frac{1}{2} 2^{-u} \sum_{\frac{T'}{\pi} \leq k \leq -u+4} (1.1)^{k} \frac{\Gamma(k/2)}{a^k}.$$

Putting this together, we conclude that
$$
\sum_{k=1}^{\tilde K_u} \frac{|C_k(p,u)|}{a^k} + |RS_{\tilde K_u}(u + iT')| \leq 
\frac{0.2 \times 9^u}{a-0.865} + \frac{0.029 \times 2^{-u}}{a - 0.353} + \frac{2^{-u}}{2} \sum_{\frac{T'}{\pi} \leq k \leq -u+4} (1.1)^{k} \frac{\Gamma(k/2)}{a^k}$$
for all $u$ (positive or negative).  We conclude that $\delta \leq \delta_1 + \delta_2 + \delta_3$, where
\begin{align}
\delta_1 &\coloneqq \int_\R \exp\left( \frac{u^2 + \frac{5}{6}}{T-3} \right) \frac{0.2 \times 9^u}{a-0.865} \frac{1}{\sqrt{\pi t}} e^{-(u-\sigma)^2/t}\ du \nonumber\\
\delta_2 &\coloneqq \int_\R \exp\left( \frac{u^2 + \frac{5}{6}}{T-3} \right) \frac{0.029 \times 2^{-u}}{a - 0.353} \frac{1}{\sqrt{\pi t}} e^{-(u-\sigma)^2/t}\ du \nonumber\\
\delta_3 &\coloneqq \int_\R \exp\left( \frac{u^2 + \frac{5}{6}}{T-3} \right) \frac{2^{-u}}{2} \sum_{\frac{T'}{\pi} \leq k \leq -u+4} (1.1)^{k} \frac{\Gamma(k/2)}{a^k} \frac{1}{\sqrt{\pi t}} e^{-(u-\sigma)^2/t}\ du.\label{delta3-def}
\end{align}
For $\delta_1$, we translate $u$ by $\sigma$ to obtain
$$ \delta_1 = \frac{0.2 \times 9^\sigma}{a-0.865} \int_\R \exp\left( \frac{u^2 + 2 \sigma u + \sigma^2 + \frac{5}{6}}{T'-0.33}  + 2 u \log 3 \right) \frac{1}{\sqrt{\pi t}} e^{-u^2/t}\ du$$
and hence by \eqref{gaussian}
\begin{equation}\label{delta1}
 \delta_1 = \frac{0.2 \times 9^\sigma}{a-0.865} \exp\left( \frac{\sigma^2 + \frac{5}{6}}{T'-0.33} + \frac{t(\log 3 + \frac{\sigma}{T'-0.33})^2}{1 - \frac{t}{T-3}} \right) \left(1 - \frac{t}{T-3}\right)^{-1/2}.
\end{equation}
One can write
\begin{equation}\label{hit}
 \frac{1}{1 - \frac{t}{T-3}} = 1 + \frac{t}{T-3-t} \leq 1 + \frac{t}{T-3.5}
\end{equation}
while by Lemma \ref{elem-lem}(ii) we have
\begin{equation}\label{hit2}
 1 - \frac{t}{T-3} = \exp\left( O_{\leq}\left( \frac{t}{T-3-t} \right) \right) = \exp\left( O_{\leq}\left( \frac{t}{T-3.5} \right) \right).
\end{equation}
We conclude that
$$ \delta_1 \leq \frac{0.2 \times 9^\sigma}{a-0.865} \exp\left( \frac{5+3t+6\sigma^2}{6(T-3.5)} + t\left(\log 3 + \frac{\sigma}{T-3}\right)^2 \left(1 + \frac{t}{T-3.5}\right) \right).$$
From Lemma \ref{elem-lem}(i) and the hypothesis $0 \leq \sigma \leq 1$, we have
\begin{align*}
\left(\log 3 + \frac{\sigma}{T-3}\right)^2 &\leq (\log^2 3) \left(1 + \frac{2 \sigma / \log 3}{T - 3 - \frac{\sigma}{2\log 3}}\right) \\
&\leq  (\log^2 3) \left(1 + \frac{2 \sigma / \log 3}{T - 3.5}\right)
\end{align*}
and therefore by a further application of Lemma \ref{elem-lem}(i)
\begin{align*}
\left(\log 3 + \frac{\sigma}{T-3}\right)^2 \left(1 + \frac{t}{T-3.5}\right) 
&\leq \log^2 3 \left(1 + \frac{\frac{2 \sigma}{\log 3} + t}{T - 3.5 - \frac{2\sigma t/\log 3}{2\sigma/\log 3 + t}}\right) \\
&\leq \log^2 3 \left(1 + \frac{\frac{2 \sigma}{\log 3} + t}{T - 3.5 - t}\right) \\
&\leq \log^2 3 \left(1 + \frac{\frac{2 \sigma}{\log 3} + t}{T - 4}\right) 
\end{align*}
and thus
$$ \delta_1 \leq \frac{0.2 \times 9^\sigma \exp( t \log^2 3 )}{a-0.865} \exp\left( \frac{5+3t+6\sigma^2 + 12 t \sigma \log 3 + 6t^2 \log^2 3}{6(T-4)} \right).$$

By repeating the proof of \eqref{delta1}, we have
$$
 \delta_2 = \frac{0.029 \times 2^{-\sigma}}{a - 0.353} \exp\left( \frac{\sigma^2 + \frac{5}{6}}{T-3} + \frac{t\left(-\log \sqrt{2} + \frac{\sigma}{T-3}\right)^2}{1 - \frac{t}{T-3}} \right) \left(1 - \frac{t}{T-3}\right)^{-1/2}.$$
We can bound $(-\log \sqrt{2} + \frac{\sigma}{T-3})^2$ by $\log^2 \sqrt{2}$.  Using \eqref{hit}, \eqref{hit2} we thus have
$$
 \delta_2 \leq \frac{0.029 \times 2^{-\sigma} \exp( t \log^2 \sqrt{2})}{a - 0.353} \exp\left( \frac{5 + 3t + 6\sigma^2 + 6t^2 \log^2 \sqrt{2}}{6(T-4)} \right).
$$

With $t \leq 1/2$ and $0 \leq \sigma \leq 1$ one has
\begin{align*}
 0.2 \exp(t \log^2 3) &\leq 0.366 \\
0.029 \exp( t \log^2 \sqrt{2}) &\leq 0.031 \\
\frac{5 + 3t + 6\sigma^2 + 6t^2 \log^2 \sqrt{2}}{6} \leq \frac{5 + 3t + 6\sigma^2 + 12 t \sigma \log 3 + 6t^2 \log^2 3}{6} &\leq 3.49
\end{align*}
and hence
$$ \delta_1 \leq \frac{0.366 \times 9^\sigma}{a-0.865} \exp\left( \frac{3.49}{T-4} \right)$$
and
$$ \delta_2 \leq \frac{0.031 \times 2^{-\sigma}}{a-0.353} \exp\left( \frac{3.49}{T-4} \right).$$

Now we turn to $\delta_3$, which will end up being extremely small compared to $\delta_1$ or $\delta_2$. By \eqref{delta3-def} and the Fubini-Tonelli theorem, we have
$$ \delta_3 = \frac{1}{2 \sqrt{\pi t}} \sum_{k \geq \frac{T'}{2.2 \pi}} (1.1)^{k} \frac{\Gamma(k/2)}{a^k} \int_{-\infty}^{4-k} \exp\left( \frac{u^2 + \frac{5}{6}}{T'-0.33}  - \frac{(u-\sigma)^2}{t} - u \log 2\right)\ du.$$
Since $u \leq 4-k$, $k \geq \frac{T'}{2.2\pi}$, and $T' \geq T \geq 100$, we have $k \geq 14$ and $u \leq -10$; since $\sigma \geq 0$, we may thus lower bound $(u-\sigma)^2/t$ by $u^2/t$.  Since $t \leq 1/2$, we can upper bound $\frac{u^2 + \frac{5}{6}}{T'-0.33} - \frac{u^2}{t}$ by (say) $-\frac{u^2}{2t}$, thus
$$ \delta_3 \leq \frac{1}{2 \sqrt{\pi t}} \sum_{k \geq \frac{T'}{2.2 \pi}} (1.1)^{k} \frac{\Gamma(k/2)}{a^k} \int_{-\infty}^{4-k} e^{-\frac{u^2}{2t} - u \log 2}\ du.$$
We can bound $e^{-\frac{u^2}{2t}} \leq e^{\frac{(k-4)u}{2t}}$, in the range of integration and thus
$$ \int_{-\infty}^{4-k} e^{-\frac{u^2}{2t} - u \log 2}\ du \leq \frac{1}{\frac{k-4}{2t} - \log 2} e^{-\frac{(k-4)^2}{2t} + (k-4) \log 2} \leq \frac{1}{\frac{k-4}{2t} - \log 2} e^{-(k-4)^2 + (k-4) \log 2};$$
bounding
$$ \frac{k-4}{2t} - \log 2 = \frac{k-4-2t \log 2}{2t} \geq \frac{k-6}{2t}$$
we conclude that
$$ \delta_3 \leq \frac{\sqrt{t}}{\sqrt{\pi}} \sum_{k \geq \frac{T'}{2.2 \pi}} (1.1)^{k} \frac{\Gamma(k/2)}{(k-6) a^k} e^{-(k-4)^2 + (k-4) \log 2}.$$
For $k \geq 14$ one can easily verify that $(1.1)^{k} \Gamma(k/2) e^{-(k-4)^2 + (k-4) \log 2} \leq 10^{-30}$; discarding the $\frac{\sqrt{t}}{\sqrt{\pi}}$ and $\frac{1}{k-6}$ factors we thus have
$$ \delta_3 \leq \sum_{k \geq 14} \frac{10^{-30}}{a^k} \leq \frac{2 \times 10^{-30}}{a^{14}}$$
(say).   Since
$$ \frac{0.031 \times 2^{-\sigma}}{a-0.353} + \frac{2 \times 10^{-30}}{a^{14}} \leq \frac{0.031 \times 2^{-\sigma}}{a-0.865}$$
we thus have
$$ \delta \leq \delta_1+\delta_2+\delta_3 \leq \frac{0.366 \times 9^\sigma + 0.031 \times 2^{-\sigma}}{a-0.865} \exp\left( \frac{3.49}{T-4} \right).$$
Inserting this and \eqref{eep}, \eqref{cop} into \eqref{rtnst}, and crudely bounding $2^{-\sigma}$ by $9^\sigma$, we obtain the claim.
\end{proof}

\subsection{Combining the estimates}

Combining Propositions \ref{rtn-prop}, \ref{RTN-prop} with \eqref{htz-expand} and the triangle inequality (and noting that $M_0 = M_0^*$, $M_t = M_t^*$ and $\alpha = \alpha^*$, and that $U$ has magnitude $1$), we conclude the following ``$A+B-C$ approximation to $H_t$'':

\begin{corollary}[$A+B-C$ approximation]\label{abc-cor}  Let $t,x,y$ obey \eqref{region}.  Set
\begin{equation}\label{tp-def}
 T' \coloneqq \frac{x}{2} + \frac{\pi t}{8} 
\end{equation}
and then define $a,N,p,U,C_0(p)$ using \eqref{a-def}, \eqref{N-def}, \eqref{U-def}, \eqref{C0-def}.  Define the quantities
\begin{align*}
s_+ = s_+(x+iy) &\coloneqq \frac{1+y-ix}{2} \\
s_- = s_-(x+iy) &\coloneqq \frac{1-y+ix}{2} \\
A_{t,N}(x+iy) &\coloneqq M_t(s_-) \sum_{n=1}^N \frac{b_n^t}{n^{s_- +\frac{t}{2} \alpha(s_-)}}  \\
B_{t,N}(x+iy) &\coloneqq M_t(s_+) \sum_{n=1}^N \frac{b_n^t}{n^{s_+ +\frac{t}{2} \alpha(s_+)}}  \\
C_t(x+iy) &\coloneqq  2 e^{-\pi i y/8} (-1)^{N} \exp\left( \frac{t \pi^2}{64}\right) \Re ( M_0(iT') C_0(p) U e^{\pi i/8} )
\end{align*}
where $M_0, b_n^t$ were defined in \eqref{Mt-def}, \eqref{bn-def}.
Then
$$ H_t(x+iy) = A_{t,N}(x+iy) + B_{t,N}(x+iy) - C_t(x+iy) + O_{\leq}(E_A(x+iy) + E_B(x+iy) + E_C(x+iy))$$
where
\begin{align*}
E_A(x+iy) &\coloneqq \left|M_t(s_-)\right| \sum_{n=1}^N \frac{b_n^t}{n^{\frac{1-y}{2}+\frac{t}{2} \Re \alpha(s_-)}} \eps_{t,n}(s_-) \\
E_B(x+iy) &\coloneqq \left|M_t(s_+)\right|\sum_{n=1}^N  \frac{b_n^t}{n^{\frac{1+y}{2}+\frac{t}{2} \Re \alpha(s_+)}} \eps_{t,n}(s_+)  \\
E_C(x+iy) &\coloneqq \exp\left( \frac{t \pi^2}{64}\right) |M_0(iT')| \left(\tilde \eps(s_-) + \tilde \eps(s_+)\right)  
\end{align*}
and $\eps_{t,n}, \tilde \eps$ were defined in \eqref{eps-def}, \eqref{epsp-def}. 
\end{corollary}

In our applications, we will just use the cruder ``$A+B$'' approximation that is immediate from the above corollary and \eqref{cop}:

\begin{corollary}[$A+B$ approximation]\label{ab-cor} With the notation and hypotheses as in Corollary \ref{abc-cor}, we have
$$ H_t(x+iy) = A_{t,N}(x+iy) + B_{t,N}(x+iy) + O_{\leq}(E_A(x+iy) + E_B(x+iy) + E_{C,0}(x+iy))$$
where
\begin{align*}
E_{C,0}(x+iy) &\coloneqq \exp\left( \frac{t \pi^2}{64}\right) |M_0(iT')| \left(1 + \tilde \eps(s_-) + \tilde \eps(s_+)\right). 
\end{align*}
\end{corollary}

We can now prove Theorem \ref{eff}.  Dividing by the expression $B_t$ from \eqref{bo-def}, and using \eqref{ft-def}, we conclude that
\begin{equation}\label{ratio-form-refined}
\frac{H_t(x+iy)}{B_t(x+iy)} = f_t(x+iy) - \frac{C_t(x+iy)}{B_t(x+iy)} + O_{\leq}\left( e_A + e_B + e_{C} \right)
\end{equation}
and
\begin{equation}\label{ratio-form}
\frac{H_t(x+iy)}{B_t(x+iy)} = f_t(x+iy) + O_{\leq}\left( e_A + e_B + e_{C,0} \right)
\end{equation}
where
\begin{align}
e_A \coloneqq e_A(x+iy) &\coloneqq |\gamma| \sum_{n=1}^N n^y \frac{b_n^t}{n^{\Re s + \Re \kappa}} \eps_{t,n}(s_-) \label{ea-def}\\
e_B \coloneqq e_B(x+iy) &\coloneqq \sum_{n=1}^N  \frac{b_n^t}{n^{\Re s}} \eps_{t,n}(s_+) \label{eb-def} \\
e_{C} \coloneqq e_{C}(x+iy) &\coloneqq \frac{\exp\left( \frac{t \pi^2}{64}\right) |M_0(iT')|}{|M_t(s_+)|} \left(\tilde \eps(s_-) + \tilde \eps(s_+)\right).\label{ecc-def}\\
e_{C,0} \coloneqq e_{C,0}(x+iy) &\coloneqq \frac{\exp\left( \frac{t \pi^2}{64}\right) |M_0(iT')|}{|M_t(s_+)|} \left(1 + \tilde \eps(s_-) + \tilde \eps(s_+) \right),\label{ec-def}
\end{align}
and where $\gamma,s_*,\kappa$ were defined in \eqref{lambda-def}, \eqref{sn-def}, \eqref{kappa-def}.  Note also from \eqref{tp-def}, \eqref{a-def}, \eqref{N-def} that $N$ is given by \eqref{N-def-main}.

To conclude the proof of Theorem \ref{eff} it thus suffices to obtain the following estimates.

\begin{proposition}[Estimates]\label{estimates}  Let the notation and hypotheses be as above. 
\begin{itemize}
\item[(i)]  One has
$$ |\gamma| \leq e^{0.02 y} \left( \frac{x}{4\pi} \right)^{-y/2} $$
\item[(ii)] One has
$$ \Re s_* \geq \frac{1+y}{2} +\frac{t}{4} \log \frac{x}{4\pi} - \frac{(1-3y+\frac{8y(1-y)}{x^2})_+ t}{2x^2}.$$
\item[(iii)]  One has
$$ \kappa = O_{\leq} \left( \frac{ty}{2(x-6)} \right).$$
\item[(iv)]  One has
$$ e_A \leq |\gamma| N^{|\kappa|} \sum_{n=1}^N n^{y} \frac{b_n^t}{n^{\Re s_*}} \left( \exp\left( \frac{\frac{t^2}{16} \log^2 \frac{x}{4\pi n^2} + 0.626}{x-6.66} \right)-1 \right).$$
\item[(v)]  One has
$$ e_B \leq \sum_{n=1}^N  \frac{b_n^t}{n^{\Re s_*}} \left( \exp\left( \frac{\frac{t^2}{16} \log^2 \frac{x}{4\pi n^2} + 0.626}{x-6.66} \right)-1 \right).$$
\item[(vi)] One has
$$ e_{C} \leq \left(\frac{x}{4\pi}\right)^{-\frac{1+y}{4}} \exp\left( - \frac{t}{16} \log^2 \frac{x}{4\pi} + \frac{3 |\log \frac{x}{4\pi} + i \frac{\pi}{2}|+3.58}{x-8.52} \right) \left(\frac{1.24 \times (3^y+3^{-y})}{N-0.125} + \frac{6.92}{x-12}\right).
$$
and
$$ e_{C,0} \leq \left(\frac{x}{4\pi}\right)^{-\frac{1+y}{4}} \exp\left( - \frac{t}{16} \log^2 \frac{x}{4\pi} + \frac{3 |\log \frac{x}{4\pi} + i \frac{\pi}{2}|+3.58}{x-8.52} \right) \left(1 + \frac{1.24 \times (3^y+3^{-y})}{N-0.125} + \frac{6.92}{x-12}\right).
$$
\end{itemize}
\end{proposition}

Note that to obtain the bound \eqref{ec-bound} from Proposition \ref{estimates}(vi) we may simply use the inequality $1+u \leq \exp(u)$ for any $u \in \R$, and then bound $\frac{1}{x-8.52} \leq \frac{1}{x-12}$.

\begin{proof}
From the mean value theorem (and noting that $M_t = M_t^*$, so that $\left|M_t\left(\frac{1+y-ix}{2}\right)\right| = \left|M_t\left(\frac{1+y+ix}{2}\right)\right|$), we have
$$ \log|\gamma| = -y \frac{d}{d\sigma} \log \left|M_t\left( \sigma + \frac{ix}{2}\right)\right| $$
for some $\frac{1-y}{2} \leq \sigma \leq \frac{1+y}{2}$.  From \eqref{alpha-def}, \eqref{Mt-def} we have
$$ \frac{d}{d\sigma} \log \left|M_t\left( \sigma + \frac{ix}{2}\right)\right| = \Re \left( \frac{t}{2} \alpha\left(\sigma+\frac{ix}{2}\right) \alpha'\left(\sigma+\frac{ix}{2}\right) + \alpha\left(\sigma+\frac{ix}{2}\right) \right).$$
From \eqref{alpha-deriv-bound} one has
\begin{equation}\label{alphap-b}
 \alpha'\left(\sigma+\frac{ix}{2}\right) = O_{\leq}\left( \frac{1}{x-6} \right)
\end{equation}
and from Taylor expansion we also have
$$ \alpha(\sigma+\frac{ix}{2}) = \alpha\left(\frac{ix}{2}\right) + O_{\leq}\left( \frac{\sigma}{x-6} \right);$$
from \eqref{alpha-form} one has
$$ \alpha\left(\frac{ix}{2}\right) = O_{\leq}\left(\frac{1}{x}\right) + O_{\leq}\left(\frac{1}{x}\right) + \frac{1}{2} \Log \frac{ix}{4\pi} 
= \frac{1}{2} \log \frac{x}{4\pi} + i \frac{\pi}{4} + O_{\leq}\left( \frac{2}{x} \right) $$
and hence
\begin{equation}\label{asig}
 \alpha(\sigma+\frac{ix}{2}) = \frac{1}{2} \log \frac{x}{4\pi} + i \frac{\pi}{4} + O_{\leq}\left( \frac{2+\sigma}{x-6} \right).
\end{equation}
Inserting these bounds, we conclude that
$$ \log|\gamma| = -y \Re \left( \left(\frac{1}{2} \log \frac{x}{4\pi} + i \frac{\pi}{4} + O_{\leq}\left( \frac{2+\sigma}{x-6} \right)\right) \left(1 + O_{\leq}\left(\frac{t}{2(x-6)}\right)\right) \right).$$
Expanding this out, we have
$$ \log|\gamma| = -y \left(\frac{1}{2} \log \frac{x}{4\pi} + O_{\leq}\left( \frac{2+\sigma + \frac{t}{4} \log \frac{x}{4\pi} + \frac{t\pi}{8} + \frac{t(2+\sigma)}{2(x-6)}}{x-6} \right)\right).$$
In the region \eqref{region}, which implies that $0 \leq \sigma \leq 1$, we have
$$ 2 + \sigma + \frac{t\pi}{8} + \frac{t(2+\sigma)}{2(x-6)} \leq 3.21$$
and thus
$$ \log|\gamma| \leq -\frac{y}{2} \log \frac{x}{4\pi} + y \frac{\frac{t}{4} \log \frac{x}{4\pi} + 3.21}{x-6}.$$
The function $x \mapsto \frac{\log \frac{x}{4\pi}}{x-6}$ is decreasing for $x \geq 200$ thanks to Lemma \ref{elem-lem}(vi), hence
$$ y \frac{\frac{t}{4} \log \frac{x}{4\pi} + 3.21}{x-6} \leq y \frac{\frac{t}{4} \log \frac{200}{4\pi} + 3.21}{200-6} \leq 0.02 y.$$
Claim (i) follows.  We remark that one can improve the $e^{0.02 y}$ factor here by Taylor expanding $\alpha$ to second order rather than first order, but we will not need to do so here.

To prove claim (ii), it suffices by \eqref{res-bound} to show that
$$ \Re \alpha(s_+) \geq \frac{1}{2} \log \frac{x}{4\pi} - \frac{(1-3y)_+}{x^2} - \frac{4y(1+y)}{x^4}.$$
By \eqref{alpha-form} one has
$$ \Re \alpha(s_+) = \frac{1+y}{(1+y)^2+x^2} - \frac{2(1-y)}{(1-y)^2+x^2} + \frac{1}{2} \log \frac{\sqrt{(1+y)^2+x^2}}{4\pi}.$$
We bound $\sqrt{(1+y)^2+x^2} \geq x$ and calculate
\begin{align*}
 \frac{1+y}{(1+y)^2+x^2} - \frac{2(1-y)}{(1-y)^2+x^2} &= - \frac{1-3y}{(1+y)^2+x^2} -\frac{8y(1-y)}{((1+y)^2+x^2)((1-y)^2+x^2)} \\
&\geq - \frac{1-3y+\frac{8y(1-y)}{x^2}}{(1+y)^2+x^2}.
\end{align*}
Lower bounding the numerator by its nonnegative part and then lower bounding $(1+y)^2+x^2$ by $x^2$, we obtain the claim.

Claim (iii) is immediate from \eqref{alphap-b} and the fundamental theorem of calculus.  Now we turn to (iv), (v).  From \eqref{asig} one has
$$ \alpha\left(\frac{1 \pm y + ix}{2}\right) - \log n = \frac{1}{2} \log \frac{x}{4\pi n^2} + i \frac{\pi}{4} + O_{\leq}\left( \frac{3}{x-6}\right)$$
for either choice of sign $\pm$.  In particular, we have
\begin{equation}\label{alphn}
 \left|\alpha\left(\frac{1 \pm y + ix}{2}\right) - \log n\right|^2 = \frac{1}{4} \log^2 \frac{x}{4\pi n^2} + \frac{\pi^2}{16} + 
O_{\leq}\left( \frac{3 |\log \frac{x}{4\pi n^2} + i \frac{\pi}{2}|}{x-6} + \frac{9}{(x-6)^2}\right).
\end{equation}
For any $1 \leq n \leq N$, we have
$$ 1 \leq n^2 \leq N^2 \leq a^2 = \frac{x+\frac{\pi t}{4}}{4\pi};$$
in the region \eqref{region}, the right-hand side is certainly bounded by $(\frac{x}{4\pi})^2$, so that
$$ \frac{4\pi}{x} \leq \frac{x}{4\pi n^2} \leq \frac{x}{4\pi}$$
and hence
$$ \left|\log \frac{x}{4\pi n^2} + i \frac{\pi}{2}\right| \leq \left|\log \frac{x}{4\pi} + i \frac{\pi}{2}\right|.$$
In the region \eqref{region} we have  $x \geq 200$, we see from Lemma \ref{elem-lem}(vi) (after squaring) that $\frac{|\log \frac{x}{4\pi} + i \frac{\pi}{2}|}{x-6}$ is decreasing in $x$.  Thus
\begin{align*}
\frac{\pi^2}{16} + \frac{3 |\log \frac{x}{4\pi n^2} + i \frac{\pi}{2}|}{x-6} + \frac{9}{(x-6)^2}
&\leq \frac{\pi^2}{16} +  \frac{3 |\log \frac{200}{4\pi} + i \frac{\pi}{2}|}{200-6} + \frac{9}{(200-6)^2} \\
&\leq 0.667.
\end{align*}
Similarly, in \eqref{region} we also have
$$ \frac{t^2}{8} \times 0.667 + \frac{t}{4} + \frac{1}{6} \leq 0.313.$$ 
We conclude from \eqref{eps-def} that
$$
\eps_{t,n}\left(\frac{1 \pm y + ix}{2} \right) \leq \exp\left( \frac{\frac{t^2}{32} \log^2 \frac{x}{4\pi n^2} + 0.313}{T-3.33} \right)-1.$$
Inserting this bound into \eqref{ea-def}, \eqref{eb-def}, we obtain claims (iv), (v).

Now we establish (vi).  From \eqref{Mt-def} we have
$$\frac{\exp\left( \frac{t \pi^2}{64}\right) |M_0(iT')|}{|M_t(s_+)|}
= \exp\left( \frac{t \pi^2}{64} - \frac{t}{4} \Re(\alpha(s_+)^2)\right) \frac{|M_0(iT')|}{|M_0(s_+)|}.$$
Note that $\frac{1+y+ix}{2} = iT' + \frac{1+y}{2} - \frac{\pi i t}{8}$.  From \eqref{alpha-deriv-bound} we see that $|\alpha'(s)| \leq \frac{1}{x-6}$ for any $s$ on the line segment between $iT'$ and $\frac{1+y+ix}{2}$.  From Taylor's theorem with remainder applied to a branch of $\log M_0$, and noting that $|M_0(s_+)| = |M_0(\frac{1+y+ix}{2})|$, we conclude that
$$ \frac{|M_0(iT')|}{|M_0(s_+)|} = \exp\left( \Re\left( \left(-\frac{1+y}{2}+\frac{\pi i t}{8}\right) \alpha(iT') \right) + O_{\leq}\left( \frac{|-\frac{1+y}{2}+\frac{\pi i t}{8}|^2}{2(x-6)} \right) \right).$$
For $0 \leq y \leq 1$ and $0 < t \leq \frac{1}{2}$ we have
$$\frac{|-\frac{1+y}{2}+\frac{\pi i t}{8}|^2}{2} \leq 0.52$$
and from \eqref{alpha-form} one has
$$ \alpha(iT') = O_{\leq}\left(\frac{1}{2T'}\right) +  O_{\leq}\left(\frac{1}{T'}\right) + \frac{1}{2} \Log \frac{iT'}{2\pi} = \frac{1}{2} \log \frac{T'}{2\pi} + \frac{i\pi}{4} + O_{\leq}( \frac{3}{2T'} ) $$
and hence
$$
\frac{|M_0(iT')|}{|M_0(s_+)|}
= \exp\left( -\frac{1+y}{4} \log \frac{T'}{2\pi} - \frac{t\pi^2}{32} + O_{\leq}\left( \frac{3|-\frac{1+y}{2}+\frac{\pi i t}{8}|}{2T'} + \frac{0.52}{x-6} \right) \right).$$
Bounding $\frac{1}{2T'} \leq \frac{1}{x-6}$ and $|-\frac{1+y}{2}+\frac{\pi i t}{8}| \leq 1.02$, this becomes
$$
\frac{|M_0(iT')|}{|M_0(s_+)|}
= \left(\frac{T'}{2\pi}\right)^{-\frac{1+y}{4}} \exp\left( - \frac{t\pi^2}{32} + O_{\leq}\left( \frac{3.58}{x-6} \right) \right)$$
and hence
$$\frac{\exp\left( \frac{t \pi^2}{64}\right) |M_0(iT')|}{|M_t(\frac{1+y+ix}{2})|} 
= \left(\frac{T'}{2\pi}\right)^{-\frac{1+y}{4}} \exp\left( - \frac{t \pi^2}{64} - \frac{t}{4} \Re (\alpha\left(\frac{1+y+ix}{2}\right)^2) + O_{\leq}\left( \frac{3.58}{x-6} \right) \right).$$
By repeating the proof of \eqref{alphn} we have
$$
\Re (\alpha(\frac{1 \pm y + ix}{2})^2) = \frac{1}{4} \log^2 \frac{x}{4\pi} - \frac{\pi^2}{16} + 
O_{\leq}\left( \frac{3 |\log \frac{x}{4\pi} + i \frac{\pi}{2}|}{x-6} + \frac{9}{(x-6)^2}\right).
$$
As before, in the region \eqref{region} we have
$$\frac{3 |\log \frac{x}{4\pi n^2} + i \frac{\pi}{2}|}{x-6} + \frac{9}{(x-6)^2} \leq \frac{3 |\log \frac{x}{4\pi} + i \frac{\pi}{2}|}{x-6} + \frac{9}{(x-6)^2}
$$
and thus
\begin{align*}
\frac{\exp\left( \frac{t \pi^2}{64}\right) |M_0(iT')|}{|M_t(s_+)|}
&= \left(\frac{T'}{2\pi}\right)^{-\frac{1+y}{4}} \exp\left( - \frac{t}{16} \log^2 \frac{x}{4\pi} + O_{\leq}\left( \frac{3 |\log \frac{x}{4\pi} + i \frac{\pi}{2}|+3.58}{x-6} + \frac{9}{(x-6)^2} \right) \right)\\
&= \left(\frac{T'}{2\pi}\right)^{-\frac{1+y}{4}} \exp\left( - \frac{t}{16} \log^2 \frac{x}{4\pi} + O_{\leq}\left( \frac{3 |\log \frac{x}{4\pi} + i \frac{\pi}{2}|+3.58}{x-8.52} \right) \right)
\end{align*}
thanks to Lemma \ref{elem-lem}(i).  Finally, since $T' \geq \frac{x}{2} \geq 100$ in \eqref{region}, one has 
$$\exp\left(\frac{3.49}{T-4}\right) \leq 1.037$$
and hence by \eqref{epsp-def}
$$ 
\tilde \eps\left(\frac{1 \pm y+ix}{2}\right) \leq \frac{1.24 \times 3^{\pm y}}{a-0.125} + \frac{1.73}{T-6}.$$
Hence
$$
\tilde \eps(s_+) + \tilde \eps(s_-) \leq \frac{1.24 \times (3^y+3^{-y})}{a-0.125} + \frac{3.46}{T-6}$$
giving the claim (substituting $T' = x/2$ and $a \geq N$).
\end{proof}

\section{Fast evaluation of multiple sums}\label{multiple-sec}

Fix $t \geq 0$.  For the verification of the barrier criterion (Theorem \ref{ubc-0}(iii)) using Corollary \ref{zero-test}, we will need to evaluate the quantity $f_t(s)$ to reasonable accuracy for a large number of values of $s$ in the vicinity of a fixed complex number $X+iy$.  From \eqref{ft-def} we have
\begin{equation}\label{fts}
f_t(s) = \sum_{n=1}^N \frac{n^b b_n^t}{n^{\frac{1+y-iX}{2}}} + \gamma(s) \sum_{n=1}^N \frac{n^a b_n^t}{n^{\frac{1-y+iX}{2}}},
\end{equation}
where $b_n^t$ is given by \eqref{bn-def}, $\gamma(s)$ is given by \eqref{lambda-def}, $N$ is given by \eqref{N-def-main} and
$$ b = b(s) \coloneqq  \frac{1+y-iX}{2} - s_* $$
and
$$ a = a(s) \coloneqq  \frac{1-y-iX}{2} - \overline{s_*} - \kappa$$
with $s_*, \kappa$ defined by \eqref{sn-def}, \eqref{kappa-def}.  In practice the exponents $a,b$ will be rather small, and $N$ will be fixed (in our main verification we will in fact have $N = 69098$).

A naive computation of $f_t(s)$ for $M$ values of $s$ would take time $O(NM)$, which turns out to be somewhat impractical for for the ranges of $N,M$ we will need; indeed, for our main theorem, the total number of pairs $(t,s)$ at which we need to perform the evaluation is $785052$ (spread out over $152$ values of $t$), and direct computation of all this data required $78.5$ hours of computer time, which was still feasible at this order of magnitude of $X$ but would not scale to significantly higher magnitudes.  However, one can significantly speed up the computation (to about $0.025$ hours) to extremely high accuracy by using Taylor series expansion to factorise the sums in \eqref{fts} into combinations of sums that do not depend on $s$ and thus can be computed in advance.

We turn to the details.  To make the Taylor series converge\footnote{One can obtain even faster speedups here by splitting the summation range $\sum_{n=1}^N$ into shorter intervals and using a Taylor expansion for each interval, although ultimately we did not need to exploit this.} faster, we recenter the sum in $n$, writing
$$ \sum_{n=1}^N F(n) = \sum_{h=-\lfloor N/2\rfloor+1}^{\lfloor (N+1)/2\rfloor} F(n_0 + h)$$
for any function $F$, where $n_0 \coloneqq \lfloor N/2 \rfloor$.  We thus have
$$ f_t(s) = B(b) + \gamma(s) A(a)$$
where
$$ B(b) \coloneqq \sum_{h=-\lfloor N/2\rfloor+1}^{\lfloor (N+1)/2\rfloor} \frac{(n_0+h)^b b_{n_0+h}^t}{(n_0+h)^{\frac{1+y-iX}{2}}}$$
and
$$ A(a) \coloneqq \sum_{h=-\lfloor N/2\rfloor+1}^{\lfloor (N+1)/2\rfloor} \frac{(n_0+h)^a b_{n_0+h}^t}{(n_0+h)^{\frac{1-y+iX}{2}}}.$$
We discuss the fast computation of $B(b)$ for multiple values of $b$; the discussion for $A(a)$ is analogous.  We can write the numerator $(n_0+h)^b b_{n_0+h}^t$ as
$$ \exp( b \log(n_0+h) + \frac{t}{4} \log^2(n_0+h) );$$
writing $\log(n_0+h) = \log n_0 + \log(1+\frac{h}{n_0})$, this becomes
$$ n_0^{b + \frac{t}{4} \log n_0} \exp( \frac{t}{4} \log^2(1+\frac{h}{n_0}) ) \exp( (b + \frac{t}{2} \log n_0) \log(1+\frac{h}{n_0}) ).$$
By Taylor expanding\footnote{It is also possible to proceed by just performing Taylor expansion on the second exponential and leaving the first exponential untouched; this turns out to lead to a comparable numerical run time.} the exponentials, we can write this as
$$ n_0^{b + \frac{t}{4} \log n_0} \sum_{i=0}^\infty \sum_{j=0}^\infty \frac{( \frac{t}{4} \log^2(1+\frac{h}{n_0}) )^i}{i!} \log^j(1+\frac{h}{n_0}) \frac{(b+\frac{t}{2} \log n_0)^j}{j!}$$
and thus the expression $B(b)$ can be written as
$$ B(b) = n_0^{b + \frac{t}{4} \log n_0} \sum_{i=0}^\infty \sum_{j=0}^\infty B_{i,j} \frac{(b+\frac{t}{2} \log n_0)^j}{j!}$$
where
$$ B_{i,j} \coloneqq \sum_{h=-\lfloor N/2\rfloor+1}^{\lfloor (N+1)/2\rfloor} \frac{( \frac{t}{4} \log^2(1+\frac{h}{n_0}) )^i}{i!} \frac{\log^j(1+\frac{h}{n_0})}{(n_0+h)^{\frac{1+y-iX}{2}}}.$$
If we truncate the $i,j$ summations at some cutoff $E$, we obtain the approximation
$$ B(b) \approx n_0^{b + \frac{t}{4} \log n_0} \sum_{i=0}^{E-1} \sum_{j=0}^{E-1} B_{i,j}(n_0) \frac{(b+\frac{t}{2} \log n_0)^i}{i!}.$$
The quantities $B_{i,j}, i,j=0,\dots,{E-1}$ may be evaluated in time $O(N E^2)$, and then the sums $B(b)$ for $M$ values of $b$ may be evaluated in time $O(ME^2)$, leading to a total computation time of $O((N+M) E^2)$ which can be significantly faster than $O(NM)$ even for relatively large values of $E$.  We took $E=50$, which is more than adequate to obtain extremely high accuracy\footnote{One can obtain more than adequate analytic bounds for the error (which are several orders of magnitude more than necessary) for the parameter ranges of interest by very crude bounds, e.g., bounding $b$ and $\log(1+\frac{h}{n_0})$ by (say) $O_{\leq}(2)$, and relying primarily on the $i!$ and $j!$ terms in the denominator to make the tail terms small.  We omit the details as they are somewhat tedious.}; for $f_t(s)$; see Figure \ref{fig1}.  The code for implementing this may be found in the file

\centerline{\tt dbn\_upper\_bound/pari/barrier\_multieval\_t\_agnostic.txt}

in the github repository \cite{github}.

\section{A new upper bound for the de Bruijn-Newman constant}\label{newup-sec}

In this section we prove Theorem \ref{new-upper}. 

\subsection{Selection of parameters}\label{select}

As stated in the introduction, it suffices to verify the conditions (i), (ii), (iii) of Theorem \ref{ubc-0} $t_0 \coloneqq 0.2$, $X \coloneqq X_0-0.5$, and $y_0 \coloneqq 0.2$, where $X_0 \coloneqq 6 \times 10^{10} + 83952$.  

The choice $t_0=y_0=0.2$ is due to the limitations of our numerical verifications, particularly the known numerical verification of RH.  We now explain the choice of $X_0$.  Recall the familiar Euler product factorization
$$ \zeta(s) = \prod_p \left(1 - \frac{1}{p^s}\right)^{-1}$$
for the Riemann zeta function.  This leads to the heuristic
$$ H_0(x+iy) \propto \prod_{p \leq P} \left(1 - \frac{1}{p^s}\right)^{-1}$$
for some small prime cutoff $P$, where $s = \frac{1+y+ix}{2}$ and we are extremely vague as to what the proportionality symbol $\propto$ means.  This heuristic extends to non-zero times $t$ as
$$ H_t(x+iy) \propto \prod_{p \leq P} \left(1 - \frac{b_p^t}{p^s}\right)^{-1}$$
and we also have
\begin{equation}\label{oscil}
 f_t(x+iy) \propto \prod_{p \leq P} \left(1 - \frac{b_p^t}{p^s}\right)^{-1}.
\end{equation}
One can non-rigorously justify the latter assertion by by inspecting the first series of $f_t(x+iy)$ in \eqref{ft-def} and ignoring the fact that the sequence $n \mapsto b_n^t$ is not multiplicative when $t \neq 0$.  

We will be relying heavily on Corollary \ref{zero-test}, and therefore seek to ensure that $|f_t(x+iy)|$ is as large as possible.
It would therefore seem to be advantageous to try to work as much as possible in regions where Euler product
$$ \prod_{p \leq P} \left(1 - \frac{b_n^t}{p^s}\right),$$
is small, which heuristically corresponds to $\frac{x}{4\pi} \log p$ being close to an integer for $p \leq P$ (so that $p^s$ has argument close to zero).  If one chooses $x$ to lie in the vicinity of
$$ X \coloneqq 6 \times 10^{10} + 83952 - 0.5$$
then indeed the fractional parts $\{ \frac{X}{4\pi} \log p\}$ for $p \leq 11$ are somewhat close to zero:
\begin{align*}
\{ \frac{X}{4\pi} \log 2 \} &= 0.0275\dots \\
\{ \frac{X}{4\pi} \log 3 \} &= 0.0437\dots \\
\{ \frac{X}{4\pi} \log 5 \} &= 0.0640\dots \\
\{ \frac{X}{4\pi} \log 7 \} &= 0.0774\dots\\
\{ \frac{X}{4\pi} \log 11 \} &= 0.0954\dots
\end{align*}
We found this shift by the following somewhat \emph{ad hoc} procedure.  We first introduced the quantity
$$ \mathrm{eulerprod}(x,p_n) \coloneqq \left|\prod\limits_{p \leq p_n}\frac{1}{1-\frac{1}{p^{1-ix/2}}}\right|,$$
which is the  exponent corresponding to $y=1$ (where the minimum value of $|f_t(x+iy)|$ in the barrier region is expected to occur).  We numerically located candidate integers $1 \leq q \leq 10^5$ for which the quantity
$$ \min_{x - 6 \times 10^{10} - q \in \{-0.5,0,0.5\}} |\mathrm{eulerprod}(x,29)|$$
exceeded a threshold (we chose $4$), to obtain seven candidates for $q$: $1046$, $22402$, $24198$, $52806$, $77752$, $83952$, and $99108$.  Among these candidates, we selected the value of $q$ which maximised the quantity
$$ \min_{x - 6 \times 10^{10} - q \in \{-0.5,0,0.5\}} |f_0(x+i)|,$$
namely $q = 83952$ (this quantity being $\approx 4.32$ for this value of $q$).
\begin{figure}[ht!]
  \includegraphics[width=\linewidth]{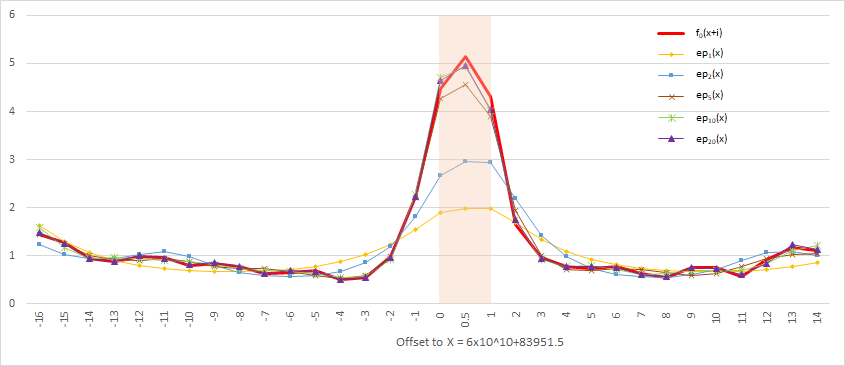}
  \caption{Improving approximation of $\mathrm{ep}_n(x)$ w.r.t. $|f_0(x+iy)|$ as $n$ is increased, where $\mathrm{ep}_n(x)=\mathrm{eulerprod}(x,p_n)$, shown near $X=6 \times 10^{10} + 83951.5$, which was chosen as a barrier location.}
	\label{euler}
\end{figure}

\subsection{Verifications of claims}

Claim (i) of Theorem \ref{ubc-0} is immediate from the result of Platt \cite{platt} that all the non-trivial zeroes of $\zeta$ with imaginary part between $0$ and $3.06 \times 10^{10}$ lie on the critical line $\{ \Re s = 1/2\}$.  For the remaining claims (ii), (iii) of Theorem \ref{ubc-0}, it will suffice to verify that $H_t(x+iy) \neq 0$ for the following three regions of $(x,y,t)$:
\begin{itemize}
\item[(ii)]  $x \geq X_0 - 0.5 + \sqrt{0.96}$, $0.2 \leq y \leq \sqrt{0.6}$, and $t = 0.2$. 
\item[(iii)]  $X_0 - 0.5 \leq x \leq X_0 + 0.5$, $0.2 \leq y \leq \sqrt{0.6}$, and $0 \leq t \leq 0.2$.
\end{itemize}
Here we have enlarged the region (iii) for simplicity.  Both of these regions lie in \eqref{region}.  We can of course replace $H_t(x+iy)$ by $H_t(x+iy)/B_t(x+iy)$.  

Set
$$ N \coloneqq \left\lfloor \sqrt{\frac{x}{4\pi} + \frac{t}{16}} \right\rfloor $$ 
so in particular
\begin{equation}\label{xnn}
 x_N \leq x < x_{N+1}
\end{equation}
where
$$ x_N \coloneqq 4 \pi N^2 - \frac{\pi t}{4}.$$
Write $N_0 \coloneqq 69098$ and $N_1 \coloneqq 1.5 \times 10^6$.
In region (ii) we then have $N \geq N_0$, while in region (iii) we have $N = N_0$.  It will now suffice to verify $\frac{H_t(x+iy)}{B_t(x+iy)} \neq 0$ in the following three regions:
\begin{itemize}
\item[(a)]  When $X_0 - 0.5 \leq x \leq X_0 + 0.5$, $N = N_0$, $0 \leq t \leq 0.2$, and $0.2 \leq y \leq 1$.
\item[(b)]  When $x \geq X_0 - 0.5$, $N_0 \leq N \leq N_1$, $t = 0.2$, and $0.2 \leq y \leq 1$.
\item[(c)]  When $N \geq N_1$, $t = 0.2$, and $0.2 \leq y \leq 1$.
\end{itemize}

In all three regions we use the following approximation:

\begin{proposition}\label{sweep}  Let $(x,y,t)$ lie in one of the regions (a), (b), (c).
Define
\begin{align*}
f_t(x+iy) &\coloneqq \sum_{n=1}^N \frac{b_n^t}{n^{s_*}} + \gamma \sum_{n=1}^N n^y \frac{b_n^t}{n^{\overline{s_*} + \kappa}}\\
b_n^t &\coloneqq \exp( \frac{t}{4} \log^2 n),\\
\end{align*}
where $\kappa, s_*, \gamma$ are as in Theorem \ref{eff}.  Then
\begin{equation}\label{bbb}
\frac{H_t(x+iy)}{B_t(x+iy)} = f_t(x+iy) + O_{\leq}( 1.25 \times 10^{-3} ).
\end{equation}
\end{proposition}

\begin{proof} 
By Theorem \ref{eff} it suffices to show that
$$ e_A + e_B + e_{C,0} \leq 1.25 \times 10^{-3}.$$
From Theorem \ref{eff} again, we have
\begin{equation}\label{eaeb-bound}
 e_A + e_B \leq (e^{\delta_1}-1) (F_{N,t}(\Re s_*) + |\gamma| N^{|\kappa|} F_{N,t}( \Re s_* - y ) )
\end{equation}
where
\begin{equation}\label{fnt-def}
 F_{N,t}( \sigma ) := \sum_{n=1}^N \frac{b_n^t}{n^\sigma}.
\end{equation}
and
\begin{equation}\label{dela}
 \delta_1 \coloneqq \frac{\frac{t^2}{16} \log^2 \frac{x}{4\pi} + 0.626}{x-6.66}.
\end{equation}
From Lemma \ref{elem-lem}(vi), the quantity $\delta_1$ is monotone decreasing in $x$ in the region \eqref{region}.  Thus we have
\begin{equation}\label{delta1-bound}
 \delta_1 \leq \frac{\frac{(0.2)^2}{16} \log^2 \frac{X_0-0.5}{4\pi} + 0.626}{X_0-0.5-6.66}
\end{equation}
whenever $x \geq X_0-0.5 \geq 200$ and $0 \leq t \leq 0.2$.  Computing the right-hand side, we conclude that
$$  \delta_1 \leq 3.12 \times 10^{-11}$$
and hence by Taylor expansion
$$ e^{\delta_1} - 1 \leq 1.001 \delta_1$$
(say).
Also, from Theorem \ref{eff} and \eqref{fnt-def} we can bound
\begin{align*}
|\gamma| F_{N,t}( \Re s_* - y - |\kappa| ) &\leq |\gamma| N^y N^{|\kappa|} F_{N,t}( \Re s_* ) \\
&\leq \exp\left( 0.02 y + y \left(\log N - \frac{1}{2} \log \frac{x}{4\pi}\right) + \frac{ty}{2(x-6)} \log N \right) F_{N,t}( \Re s_* )  \\
&\leq \exp\left( 0.02 y + \left(y + \frac{ty}{2(x-6)}\right) \frac{1}{2} \log(1 + \frac{\pi t}{4x}) + \frac{ty}{4(x-6)} \log \frac{x}{4\pi} \right) F_{N,t}( \Re s_* ).
\end{align*}
For $0.2 \leq y \leq 1$, $0 \leq t \leq 0.2$, and $x \geq X_0 - 0.5$ we see from Lemma \ref{elem-lem}(vi) that
$$ \frac{ty}{4(x-6)} \log \frac{x}{4\pi} \leq \frac{0.2}{4(X_0 - 0.5-6)}\log \frac{X_0 - 0.5}{4\pi} \leq 1.86 \times 10^{-11}$$
and
$$ \left(y + \frac{ty}{2(x-6)}\right) \frac{1}{2} \log(1 + \frac{\pi t}{4x}) \leq \left(1 + \frac{0.2}{2(X_0-0.5-6)}\right) \frac{1}{2} \log\left(1 + \frac{0.2 \pi}{4(X_0-0.5)}\right) \leq 1.31 \times 10^{-12}$$
and thus
\begin{equation}\label{gafn}
 |\gamma| F_{N,t}( \Re s_* - y - |\kappa| ) \leq 1.021  F_{N,t}( \Re s_* ).
\end{equation}
Thus
$$  e_A + e_B \leq 1.023 \delta_1 F_{N,t}( \Re s_* ).$$
To estimate $\Re s_*$, we use Proposition \ref{estimates}(ii) or \eqref{kappa-bound}, together with the inequality
$$ \frac{t}{2x^2} \left(1-3y+\frac{8y(1-y)}{x^2}\right)_+ \leq \frac{0.2}{2 (X_0-0.5)^2} \left(1 - 3 \times 0.2 + \frac{8}{(X_0-0.5)^2}\right)_+
\leq 1.2 \times 10^{-23}$$
to obtain 
$$ \Re s_* \geq 0.5999 + \frac{t}{4} \log \frac{x}{4\pi}$$
(say).  Since $F_{N,t}(\sigma)$ is non-increasing in $\sigma$, we conclude
$$  e_A + e_B \leq 1.023 \delta_1 F_{N,t}\left( 0.5999 + \frac{t}{4} \log \frac{x}{4\pi} \right).$$
Since
$$ N = \left\lfloor \sqrt{\frac{x}{4\pi} + \frac{t}{16}} \right\rfloor,$$
$0 \leq t \leq 0.2$, and $x \geq X_0-0.5$, it is easy to see that
$$ N \leq \frac{x}{4\pi}$$
and hence by \eqref{bn-def}
$$ \frac{b_n^t}{n^{\frac{t}{4} \log \frac{x}{4\pi}}} \leq 1$$
for all $1 \leq n \leq N$.  Therefore
$$ F_{N,t}\left( 0.5999 + \frac{t}{4} \log \frac{x}{4\pi} \right) \leq \sum_{n=1}^N \frac{1}{n^{0.6001}}$$
and hence by the integral test
$$ F_{N,t}\left( 0.5999 + \frac{t}{4} \log \frac{x}{4\pi} \right) \leq \int_1^{N+1} \frac{ds}{s^{0.5999}} = \frac{1}{0.4001} (N+1)^{0.4001} $$
so that (by \eqref{dela})
$$ e_A + e_B \leq 1.023 \frac{\frac{(0.2)^2}{16} \log^2 \frac{x}{4\pi} + 0.626}{x-6.66} \frac{1}{0.5999} (N+1)^{0.5999}.$$
We have
$$ N+1 \leq (1.001) \left(\frac{x-6.66}{4\pi}\right)^{1/2} $$
(say), and hence
$$ e_A + e_B \leq 0.7220 \frac{0.0025 \log^2 \frac{x}{4\pi} + 0.626}{(x-6.66)^{0.80005}}$$
From Lemma \ref{elem-lem}(vi), the right-hand side is monotone decreasing in the region $x \geq X_0-0.5$, thus
\begin{align*}
 e_A + e_B &\leq 0.7220 \frac{0.0025 \log^2 \frac{X_0-0.5}{4\pi} + 0.626}{(X_0-0.5-6.66)^{0.80005}} \\
&\leq 3.444 \times 10^{-9}.
\end{align*}

Meanwhile, from Proposition \ref{estimates}(vi) one has
$$ e_{C,0} \leq \left(\frac{x}{4\pi}\right)^{-\frac{1+y}{4}} \exp\left( - \frac{t}{16} \log^2 \frac{x}{4\pi} + \frac{3 |\log \frac{x}{4\pi} + i \frac{\pi}{2}|+3.58}{x-8.52} \right) \left(1 + \frac{1.24 \times (3^y+3^{-y})}{N-0.125} + \frac{6.92}{x-6.66}\right).$$
From Lemma \ref{elem-lem}(vi), the quantity $\frac{\log^2 \frac{x}{4\pi} + \frac{\pi^2}{4}}{(x-8.52)^2}$ is monotone decreasing in $x$ in \eqref{region}, hence $\frac{|\log \frac{x}{4\pi} + i \frac{\pi}{2}|}{x-8.52}$ is also monotone decreasing.  Also the expression is monotone decreasing in $y$. We conclude that
\begin{equation}\label{ec0-bound}
\begin{split}
 e_{C,0} &\leq \left(\frac{X_0-0.5}{4\pi}\right)^{-\frac{1+0.2}{4}} \exp\left( \frac{3 |\log \frac{X_0-0.5}{4\pi} + i \frac{\pi}{2}|+3.58}{X_0-0.5-8.52} \right) \left(1 + \frac{1.24 \times (3^{\sqrt{0.6}}+3^{-\sqrt{0.6}})}{N_0-0.125} + \frac{6.92}{X_0-0.5-6.66}\right) \\
&\leq 1.249 \times 10^{-3}
\end{split}
\end{equation}
where we have discarded the negative term $- \frac{t}{16} \log^2 \frac{X_0-0.5}{4\pi}$.  Combining the estimates, we obtain the claim.
\end{proof}

Now we attend to the three claims.

\subsection{Proof of claim (c)}\label{c-bound}  

We begin with claim (c), which is the easiest.  
By Proposition \ref{sweep}, it suffices to establish the bound
$$ |f_t(x+iy)| > 1.25 \times 10^{-3}.$$
In fact we will establish the stronger estimate
\begin{equation}\label{ft0}
f_t(x+iy) = 1 + O_{\leq}( 0.955 ).
\end{equation}

In the region (c) we have from \eqref{xnn} that
$$ x \geq x_{N_1} \geq  2.82 \times 10^{13}.$$

Our main tool here is the triangle inequality.  From \eqref{ft-def} one has
$$
f_t(x+iy) = 1 + \sum_{n=2}^N \frac{b_n^t}{n^{s_*}} + \gamma \sum_{n=1}^N n^y \frac{b_n^t}{n^{\overline{s_*} + \kappa}}$$
and hence
$$ f_t(x+y) = 1 + O_{\leq}\left( \sum_{n=2}^N \frac{b_n^t}{n^{\sigma}} + |\gamma| \sum_{n=1}^N n^y \frac{b_n^t}{n^{\sigma - |\kappa|}} \right)$$
where $\sigma \coloneqq \Re s_*$.  Restoring the $n=1$ term in the first sum and recalling that $t=0.2$, it thus suffices to show that
\begin{equation}\label{nan}
 \sum_{n=1}^N \frac{b_n^{0.2}}{n^\sigma} + |\gamma| \sum_{n=1}^N n^y \frac{b_n^{0.2}}{n^{\sigma-|\kappa|}}
< 1.955.
\end{equation}
Our main tool here will be 

\begin{lemma}\label{largen}
Let $N \geq N_0 \geq 1$ be natural numbers, and let $\sigma,t > 0$ be such that
$$ \sigma > \frac{t}{2} \log N.$$
Then
$$ \sum_{n=1}^N \frac{b_n^t}{n^\sigma} \leq \sum_{n=1}^{N_0}
\frac{b_n^t}{n^\sigma}  + 
\max( N_0^{1-\sigma} b_{N_0}^t, N^{1-\sigma} b_N^t ) \log \frac{N}{N_0}.$$
\end{lemma}

\begin{proof}  From the identity
$$ \frac{b_n^t}{n^\sigma} = \frac{\exp\left( \frac{t}{4} (\log N - \log n)^2 - \frac{t}{4} (\log N)^2\right) }{n^{\sigma - \frac{t}{2} \log N}}$$
we see that the summands $\frac{b_n^t}{n^\sigma}$ are decreasing for $1 \leq n \leq N$, hence by the integral test one has
\begin{equation}\label{rado}
 \sum_{n=1}^N \frac{b_n^t}{n^\sigma} \leq \sum_{n=1}^{N_0}
\frac{b_n^t}{n^\sigma}  + \int_{N_0}^N \frac{b_a^t}{a^\sigma}\ da.
\end{equation}
Making the change of variables $a = e^u$, the right-hand side becomes
$$\sum_{n=1}^{N_0} \frac{b_n^t}{n^\sigma} \int_{N_0}^N \exp( (1-\sigma) u + \frac{t}{4} u^2 )\ du.$$
The expression $(1-\sigma) u + \frac{t}{4} u^2$ is convex in $u$, and is thus bounded by the maximum of its values at the endpoints $u = \log N_0, \log N$; thus
$$\exp( (1-\sigma) u + \frac{t}{4} u^2) \leq N_0^{1-\sigma} b_{N_0}^t, N^{1-\sigma} b_N^t.$$
The claim follows. 
\end{proof}

\begin{remark}  The right-hand side of \eqref{rado} can be evaluated exactly as
$$
\sum_{n=1}^{N_0}
\frac{b_n^t}{n^\sigma}  + \frac{\sqrt \pi}{\sqrt t} \exp(\frac{-(\sigma - 1)^2}{t}) \left( \operatorname{erfi}\left(\frac{\frac{t}{2} \log N  - \sigma + 1}{\sqrt t} \right) - \operatorname{erfi}\left(\frac{\frac{t}{2} \log N_0  - \sigma + 1}{\sqrt t}\right) \right)$$
where $\operatorname{erfi}(z) = -i \operatorname{erf}(iz)$ is the imaginary error function, with $\operatorname{erf}(z) \coloneqq \frac{2}{\sqrt{\pi}} \int_0^z e^{-t^2}\ dt$.

In practice, this upper bound for $\sum_{n=1}^N \frac{b_n^t}{n^\sigma}$ is slightly more accurate than the one in Lemma \ref{largen}, and is a good approximation even for relatively small values of $N_0$ (e.g., $N_0=100$).  However, the cruder bound above suffices for the numerical values of parameters needed to establish the bound $\Lambda \leq 0.22$.
\end{remark}

Observe from \eqref{kappa-bound} that
$$ |\kappa| \leq \frac{0.2 \times 1}{2(2.82 \times 10^{13}-6)} \leq 3.55 \times 10^{-15}$$
while from \eqref{gamma-bound} one has
$$
|\gamma| \leq 1.005 \left( \frac{x_N}{4\pi} \right)^{-y/2}$$
so in particular since $0.2 \leq y \leq 1$ and $n \leq 1.0001 (\frac{x_N}{4\pi})^{1/2} \leq 1.0002 N$
$$
|\gamma| n^y \leq 1.006 N^{-0.2} n^{0.2}.$$
Also from \eqref{res-bound} one has
\begin{align*}
\sigma &\geq 0.6 + \frac{0.2}{4} \log \frac{x_N}{4\pi} - \frac{0.2}{2x_{N_1}^2} \left(0.4+\frac{0.96}{x_{N_1}^2}\right)_+ \\
&\geq 0.6 + 0.1 \log N + 0.05 \log \left(1 - \frac{t}{16 (N_1)^2}\right) - \frac{0.2}{2x_{N_1}^2} \left(0.4+\frac{0.96}{x_{1.5\times 10^6}^2}\right)_+ \\
&\geq 0.6 + 0.1 \log N - 5.1 \times 10^{-29}.
\end{align*}
We can then apply Lemma \ref{largen} twice to bound the left-hand side of \eqref{nan} by $A+B$, where
\begin{align*}
A &\coloneqq \sum_{n=1}^{N_0} \frac{b_n^{0.2}}{n^{\sigma'}} + 1.006 \left( \frac{x_N}{4\pi} \right)^{-0.1} \sum_{n=1}^{N_0} \frac{b_n^{0.2}}{n^{\sigma''}}, \\
B &\coloneqq (\max( N_0^{1-\sigma'} b_{N_0}^{0.2}, N^{1-\sigma'} b_N^{0.2} ) + 1.006 N^{-0.2} 
\max( N_0^{1-\sigma''} b_{N_0}^{0.2}, N^{1-\sigma''} b_N^{0.2} ) )\log \frac{N}{N_0} \\
\sigma' &:= 0.6 + 0.1 \log N - 5.1 \times 10^{-29} \\
\sigma'' &:= 0.4 + 0.1 \log N - 7.09 \times 10^{-16}.
\end{align*}
The quantity $A$ is decreasing in $N$, so we may bound it by its value at $N = N_1$.  Performing the sum numerically, we obtain
$$ A \leq 1.88.$$
Finally, the quantity $B$ can also be seen to be decreasing\footnote{This is visually apparent from Figure \ref{bnp}, but to prove it analytically, it suffices to show that the quantities $N_0^{-\sigma'} \log \frac{N}{N_0}$, $N^{1-\sigma'} b_N^{0.2} \log \frac{N}{N_0}$, $N^{-0.2} N_0^{-\sigma''} \log \frac{N}{N_0}, N^{-0.2} N^{1-\sigma''} b_N^{0.2} \log \frac{N}{N_0}$ are decreasing in $N$ for $N \geq N_1$.  This can in turn be established by computing the log-derivative of all these quantities, multiplied by $N$; there will be a negative term $-0.1 \log N_0$ or $-0.1 \log N$ which dominates all the other terms when $N \geq N_1$.  We leave the details to the interested reader.}  in $N$ in the range $N \geq N_1$, and obeys the bound
$$ B \leq 0.075.$$
The claim \eqref{ft0} follows.

\begin{figure}[ht!]
  \includegraphics[width=0.7\linewidth]{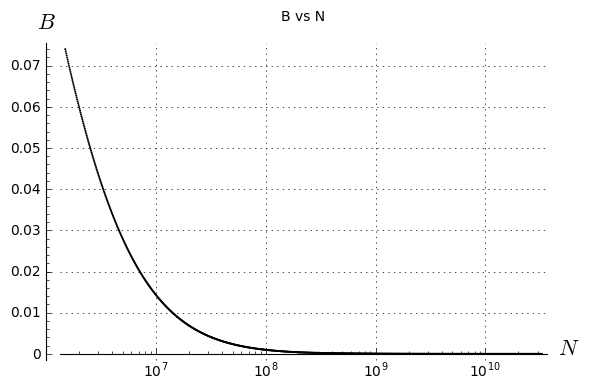}
  \caption{Plot of $B$ vs $N$}\label{bnp}
\end{figure}

\subsection{Proof of claim (a)}\label{barrier-sec}

As $N = N_0$ is constant in this region, the function $f_t(x+iy)$ is holomorphic, so by Rouche's theorem, it suffices to show that for each time $0 \leq t \leq 0.2$, as $x+iy$ traverses the boundary $\partial R$ of the rectangle 
$$ R \coloneqq \{ x+iy: X_0 - 0.5 \leq x \leq X_0 + 0.5; 0.2 \leq y \leq 1 \},$$
the function $f_t(x+iy)$ stays outside of the ball $B \coloneqq \{ z: |z| \leq 1.25 \times 10^{-3} \}$, and furthermore has a winding number of zero around the origin.

To verify this claim numerically for a given value of $t$, we subdivide each edge of $\partial R$ into some number $n$ of equally spaced mesh points, thus approximating $\partial R$ by a discrete mesh $x_j+iy_j$, $j=1,\dots,4n$, with any two adjacent points on this mesh separated by a distance at most $1/n$.  Using the techniques in Section \ref{multiple-sec}, we evaluate $f_t(x_j+iy_j)$ numerically for each such value of $j$.  The polygonal path connecting these points then winds around the origin with winding number
$$ \frac{1}{2\pi} \sum_{j=1}^{4n} \mathrm{arg}( f_t(x_{j+1}+iy_{j+1}) / f_t(x_j + i y_j) )$$
which can be easily computed (and verified to be zero).  To pass from this polygonal path to the true trajectory $f_t(\partial R)$ of $f_t(x+iy)$ on $\partial R$, we again use Rouche's theorem.  If one has a derivative bound $|\frac{\partial}{\partial z} f_t(z)| \leq D_z$ on the boundary of the rectangle, the polygonal path and the true trajectory $f_t(\partial R)$ differ by a distance of at most $\frac{D_z}{2n}$, and the latter will have the same winding number around the origin (and stay outside of the ball $B$) as long as
$$ |f_t(x_j + iy_j)| > 1.25 \times 10^{-3} + \frac{D_z}{2n}.$$
Furthermore, the same is true for nearby times $t \leq t' \leq 0.2$ to $t$, as long as one has the stronger bound
\begin{equation}\label{cond}
 |f_t(x_j + iy_j)| > 1.25 \times 10^{-3} + \frac{D_z}{2n} + \frac{D_t |t'-t|}{2n}
\end{equation}
and a bound of the form $|\frac{\partial}{\partial t} f_{\tilde t}(z)| \leq D_t$
for $t \leq \tilde t \leq 0.2$ and $z \in \partial R$.

This gives the following algorithm to verify (a) for the entire range $0 \leq t \leq 0.2$.  We start with $t=0$ and obtain bounds $D_t, D_z$ for the derivatives of $f_{\tilde t}(z)$ in the indicated ranges.  Because of the way the barrier location $X$ was selected, we expect $|f_t(x+iy)|$ to stay well above $1$ in magnitude.  We thus choose $n$ so that $\frac{D_z}{2n} \leq 1$, and evaluate $f_t(x_j+iy_j)$ at all the mesh points (in particular confirming that $|f_t(x_j+iy_j)|$ does stay well above $1$).  Using the minimum value of $|f_t(x_j+iy_j)|$, we can then use the condition \eqref{cond} to establish the claim for times in the interval $[t,t')$ where $t'>t$ is chosen so that \eqref{cond} holds (or $t'=0.2$, if that is also possible); the most aggressive choice of $t'$ would be one in which \eqref{cond} held with equality, but in practice we can afford to take more conservative values of $t'$ and still obtain good runtime performance.  If $t' < 0.2$, we then repeat the process, replacing $t$ by $t'$, until the entire range $0 \leq t \leq 0.2$ is verified.

To run this algorithm, we need bounds on $D_t$ and $D_z$.  This is achieved by the following lemma, which gives bounds which are somewhat complicated but which can be easily upper bounded numerically on $\partial R$:

\begin{lemma} In the region \eqref{region}, and away from the jump discontinuities of $N$, we have
\begin{align*}
 \left|\frac{\partial f_t}{\partial z}\right| &\leq  \sum_{n=1}^N \frac{b_n^t}{n^{\Re s_*}} \left(\frac{\log n}{2} + \frac{t \log n}{4(x-6)}\right) \\
&\quad + |\gamma| N^{|\kappa|} \sum_{n=1}^N \frac{b_n^t n^{y} }{n^{\Re s_{*}}}
\left( \frac{t \log n}{4(x-6)} + \left(\log \frac{|1+y+ix|}{4\pi} + \pi + \frac{3}{x}\right) \left(\frac{1}{2} + \frac{t}{4(x-6)}\right)\right)
\end{align*}
and
\begin{align*} \left|\frac{\partial f_t}{\partial t}\right| &\leq \sum_{n=1}^N \frac{b_n^t}{n^{\Re s_*}} \left(\frac{1}{4} \log n \log \frac{x}{4\pi n} + \frac{\pi}{8} \log n + \frac{2 \log n}{x-6}\right) \\
&\quad + |\gamma| N^{|\kappa|} \sum_{n=1}^N \frac{b_n^t n^y}{n^{\Re s_{*}}}
\left(\frac{1}{4} \log n \log \frac{x}{4\pi n} + \frac{\pi}{8} \log n + \frac{2 \log n}{x-6} + \frac{1}{4} \left(\frac{\pi}{2} + \frac{8}{x-6}\right) \left(\log \frac{x}{4\pi} + \frac{8}{x-6}\right)\right).
\end{align*}
\end{lemma}

\begin{proof}
We begin with the first estimate.  Write 
$$ s_{**} \coloneqq \overline{s_*} - y + \kappa = \frac{1-y+ix}{2} + \frac{t}{2} \alpha\left(\frac{1-y+ix}{2}\right)$$
then
\begin{equation}\label{ftne}
f_t = \sum_{n=1}^N \frac{b_n^t}{n^{s_*}} + \gamma \sum_{n=1}^N \frac{b_n^t}{n^{s_{**}}}.
\end{equation}
One can check that $s_*, s_{**}, \gamma$ are holomorphic functions of $x+iy$, hence by the Cauchy-Riemann equations
$$ \left|\frac{\partial f_t}{\partial x}\right| = \left|\frac{\partial f_t}{\partial y}\right|.$$
By the product and chain rules, we may calculate
$$ 
\frac{\partial f_t}{\partial x} = - \sum_{n=1}^N \frac{b_n^t}{n^{s_*}} \frac{\partial s_*}{\partial x} \log n + \gamma \sum_{n=1}^N \frac{b_n^t}{n^{s_{**}}}
\left( \frac{\partial}{\partial x} \log \gamma - \frac{\partial s_{**}}{\partial x} \log n\right).$$
From \eqref{sn-def}, \eqref{alpha-deriv-bound} we have
\begin{align*}
 \frac{\partial s_*}{\partial x} &= -\frac{i}{2} - \frac{it}{4} \alpha'\left(\frac{1+y-ix}{2}\right) \\
&= -\frac{i}{2} + O_{\leq}\left( \frac{t}{4(x-6)} \right).
\end{align*}
Similarly we have
$$ \frac{\partial s_{**}}{\partial x} = \frac{i}{2} + O_{\leq}\left( \frac{t}{4(x-6)} \right).$$
Writing $s = \frac{1-y+ix}{2}$, we have from \eqref{lambda-def}, \eqref{Mt-def} that
$$ \log \gamma = \frac{t}{4} (\alpha(s)^2 - \alpha(1-s)^2) + \log M_0(s) - \log M_0(1-s) $$
and hence by \eqref{alpha-def}
$$ \frac{\partial}{\partial x} \log \gamma = \frac{it}{4} (\alpha(s) \alpha'(s) + \alpha(1-s) \alpha'(1-s))
+ \frac{i}{2} \alpha(s) + \frac{i}{2} \alpha(1-s).$$
From the triangle inequality and \eqref{alpha-deriv-bound}, we thus have
\begin{align*}
|\frac{\partial f_t}{\partial x}| &\leq \sum_{n=1}^N \frac{b_n^t}{n^{\Re s_*}} \left(\frac{\log n}{2} + \frac{t \log n}{4(x-6)}\right) \\
&\quad + |\gamma| \sum_{n=1}^N \frac{b_n^t}{n^{\Re s_{**}}}
\left( \frac{t \log n}{4(x-6)} + \frac{|\alpha(s)| + |\alpha(1-s)| - \log n}{2} + \frac{t (|\alpha(s)| + |\alpha(1-s)|)}{4(x-6)}\right).
\end{align*}
We have from \eqref{alpha-form} that
$$ |\alpha(s)|, |\alpha(s) - \frac{1}{2} \log n| \leq \frac{1}{2} \log \frac{|1-y+ix|}{4\pi} + \frac{\pi}{2} + \frac{3}{2x} $$
since $n \leq N \leq \frac{x}{4\pi} \leq \frac{|1-y+ix|}{4\pi}$.  Similarly
$$ |\alpha(1-s)|, |\alpha(1-s) - \frac{1}{2} \log n| \leq \frac{1}{2} \log \frac{|1+y+ix|}{4\pi} + \frac{\pi}{2} + \frac{3}{2x} $$
and thus
$$ |\alpha(s)+\alpha(1-s)|, |\alpha(s)+\alpha(1-s)-\log n| \leq \log \frac{|1+y+ix|}{4\pi} + \pi + \frac{3}{x}.$$
Writing $\Re s_{**} = \Re s_* - y + \Re \kappa$, we then have the first estimate.

Now we estimate the time derivative.  Since
\begin{align*}
 \frac{\partial}{\partial t} \log b_n^t &= \frac{1}{4} \log^2 n \\
 \frac{\partial}{\partial t} s_* &= \frac{1}{2} \alpha(1-s) \\
 \frac{\partial}{\partial t} s_{**} &= \frac{1}{2} \alpha(s) \\
 \frac{\partial}{\partial t} \log \gamma &= \frac{1}{4} \left(\alpha(s)^2 - \alpha^2(1-s)\right)
\end{align*}
we see from differentiating \eqref{ftne} that, we obtain
\begin{align*}
 \frac{\partial f_t}{\partial t} &= \sum_{n=1}^N \frac{b_n^t}{n^{s_*}} \left(\frac{\log^2 n}{4} - \frac{\alpha(1-s)}{2} \log n\right) \\
&+ \gamma \sum_{n=1}^N \frac{b_n^t}{n^{s_{**}}}
\left(\frac{\log^2 n}{4} - \frac{\alpha(s)}{2} \log n + \frac{1}{4} (\alpha(s)^2 - \alpha^2(1-s))\right).
\end{align*}
From \eqref{alpha-deriv-bound}, \eqref{alpha-form} we have
\begin{align*}
 \alpha\left(\frac{1 \pm y+ix}{2}\right) &= \alpha\left(\frac{ix}{2}\right) + O_{\leq}\left( \frac{1}{x-6} \right) \\
&= \frac{1}{2} \log \frac{x}{4\pi} + \frac{\pi i}{4} + O_{\leq}\left( \frac{4}{x-6} \right) 
\end{align*}
and hence (since $\alpha = \alpha^*$)
$$ \alpha\left(\frac{1 \pm y-ix}{2}\right) = \frac{1}{2} \log \frac{x}{4\pi} - \frac{\pi i}{4} + O_{\leq}\left( \frac{4}{x-6} \right) $$
so in particular (recalling that $1-s = \frac{1+y-ix}{2}$ and $s = \frac{1-y+ix}{2}$)
$$ \alpha(s) - \alpha(1-s) = \frac{\pi i}{2} + O_{\leq}\left( \frac{8}{x-6} \right)$$
and
$$ \alpha(s) + \alpha(1-s) = \log \frac{x}{4\pi} + O_{\leq}\left( \frac{8}{x-6} \right)$$
so that
$$ \left|\alpha(s)^2 - \alpha(1-s)^2\right| \leq \left(\frac{\pi}{2} + \frac{8}{x-6}\right) \left(\log \frac{x}{4\pi} + \frac{8}{x-6}\right).$$
We conclude from the triangle inequality that
\begin{align*}
 \left|\frac{\partial f_t}{\partial t}\right| &\leq \sum_{n=1}^N \frac{b_n^t}{n^{\Re s_*}} \left(\frac{1}{4} \log n \log \frac{x}{4\pi n} + \frac{\pi}{8} \log n + \frac{2 \log n}{x-6}\right) \\
&+ |\gamma| \sum_{n=1}^N \frac{b_n^t}{n^{\Re s_{**}}}
\left(\frac{1}{4} \log n \log \frac{x}{4\pi n} + \frac{\pi}{8} \log n + \frac{2 \log n}{x-6} + \frac{1}{4} \left(\frac{\pi}{2} + \frac{8}{x-6}\right) \left(\log \frac{x}{4\pi} + \frac{8}{x-6}\right)\right)
\end{align*}
giving the second claim.
\end{proof}

The next few graphs summarize the numerical output of the algorithm for the following barrier parameters: $x=6\times 10^{10}+83952 \pm 0.5, y = 0.2 \dots 1, t=0 \dots 0.2$.  

The first step in the barrier verification process was to precalculate a `stored sum' of Taylor expansion terms that allows for fast recreation of $f_t(x+iy)$ during the execution of the algorithm as per Section \ref{multiple-sec}. The number of Taylor terms required was determined through an iterative process targeted to achieve a $20$ decimal accuracy.  Figure \ref{fig1} illustrates that the achieved accuracy for all rectangular mesh points at $t=0$.

\begin{figure}[ht!]
  \includegraphics[width=0.7\linewidth]{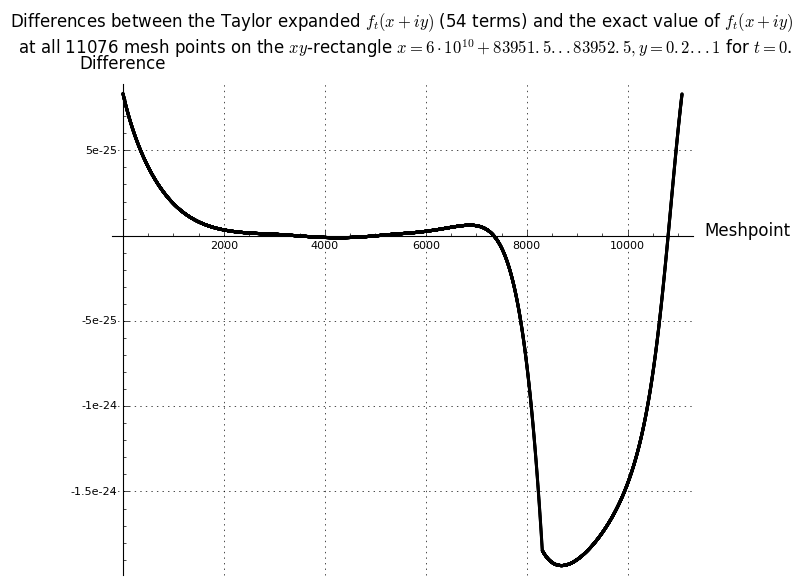}
  \caption{Achieved error term in the Taylor expansion at $t=0$. Target was set at $20$ decimal places accuracy.}\label{fig1}
\end{figure}

The derivative bounds determine the number of mesh points required on each $xy$-rectangle and in the $t$-direction. Figure \ref{fig2} illustrates that these bounds have been chosen quite conservatively:

\begin{figure}[ht!]
  \includegraphics[width=1.0\linewidth]{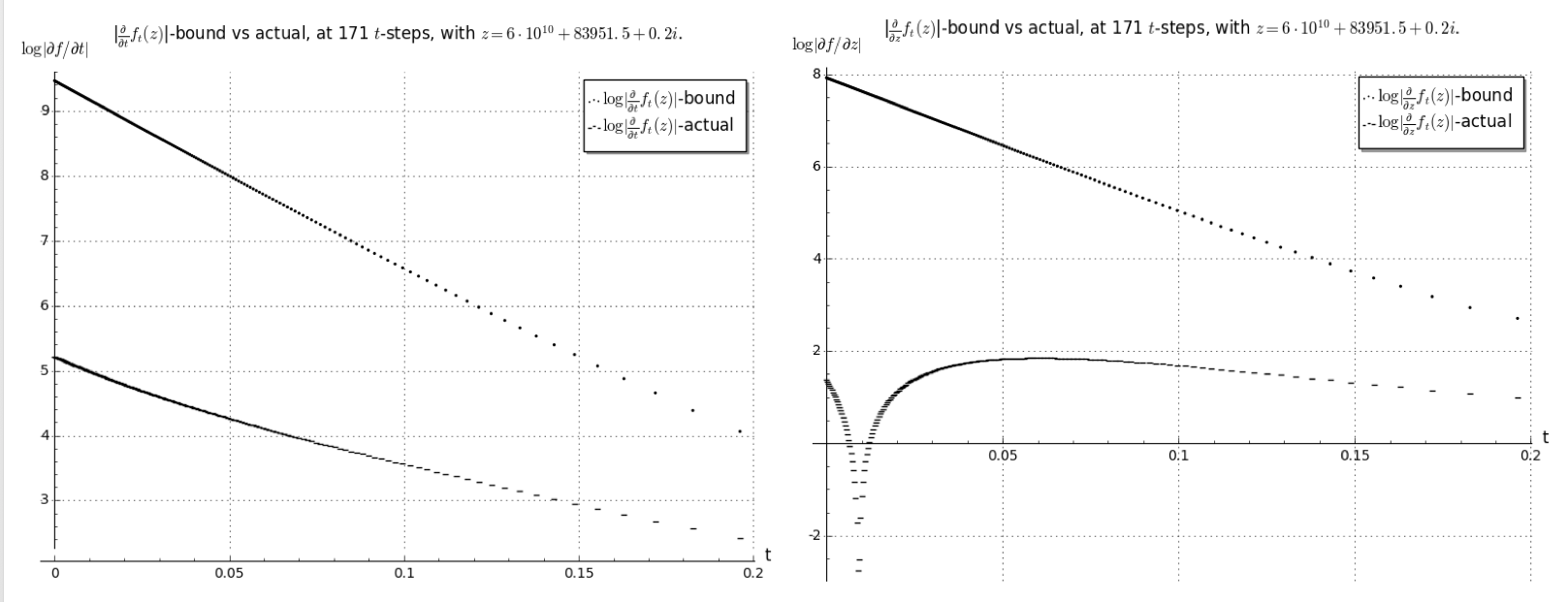}
  \caption{Derivative bounds versus their actual values at all required steps of $t$.}\label{fig2}
\end{figure}

The number of rectangle mesh points varies with $t$ ranging from $11076$ at $t=0$ to $56$ at $t=0.195$; see Figure \ref{fig3}.

\begin{figure}[ht!]
  \includegraphics[width=0.7\linewidth]{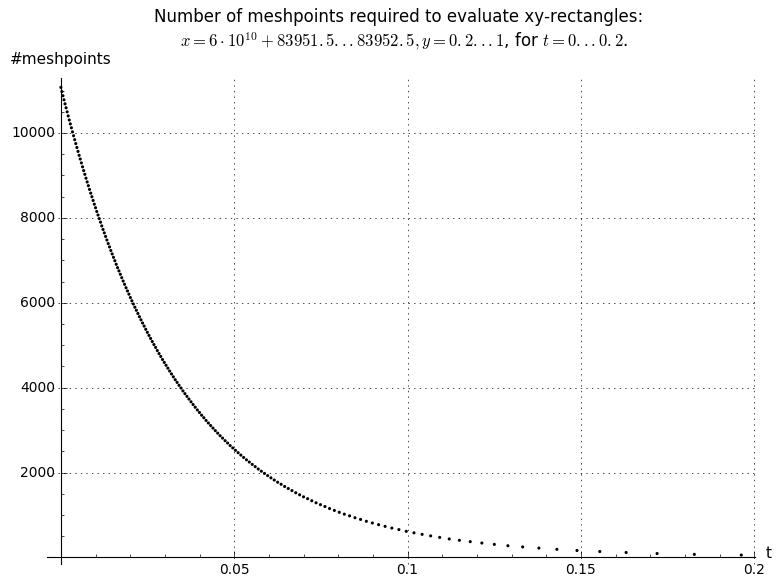}
  \caption{The number of mesh points required per rectangle for each step of $t$.}\label{fig3}
\end{figure}

The overall winding number for the barrier at this specific location came out at $0$. Figures \ref{fig4}, \ref{fig4a} show the winding process at $t=0, 0.2$ respectively. Due to the choice of the barrier location and the small barrier width compared to the wavelength in the $x$ variable, little oscillation is expected to occur in each rectangle.

\begin{figure}[ht!]
  \includegraphics[width=0.7\linewidth]{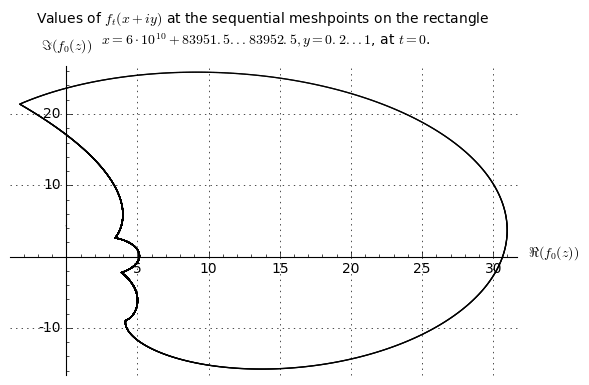}
  \caption{Winding the rectangle at $t=0$.}\label{fig4}
\end{figure}

\begin{figure}[ht!]
  \includegraphics[width=0.7\linewidth]{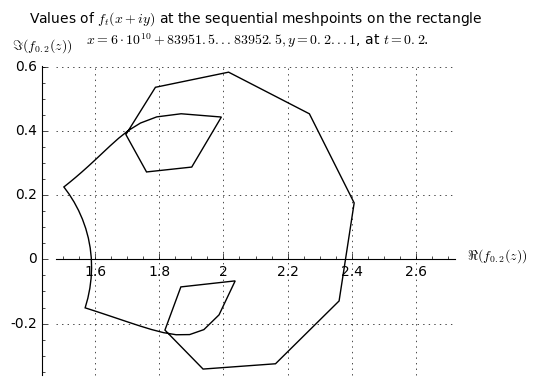}
  \caption{Winding the rectangle at $t=0.2$.}\label{fig4a}
\end{figure}

The code for implementing these computations may be found in the directory

\centerline{\tt dbn\_upper\_bound/arb}

in the github repository \cite{github}.

\subsection{Proof of claim (b)}\label{b-bound}

Fix $t=0.2$, and let $R$ denote the rectangle
$$ R \coloneqq \{ x+iy: 0.2 \leq y \leq 1; x \geq X_0 - 0.5; N \leq N_1 \}.$$
We wish to show that the holomorphic function $H_t(x+iy)/B_t(x+iy)$ does not vanish in this rectangle.  We would like to establish this by the argument principle, however this turns out to be difficult to accomplish due to the oscillation in $H_t(x+iy)/B_t(x+iy) \approx f_t(x+iy)$ indicated by the heuristic \eqref{oscil}.  To damp out this oscillation, we introduce an\footnote{In the literature one also sees other choices of mollifier than this Euler product used, for instance to control the extreme values of Dirichlet polynomials; however our numerical experimentations with alternative mollifiers to $f_t(x+iy)$ turned out to give inferior results for our application.} ``Euler mollifier''
$$ E_{t,5}(x+iy) \coloneqq \prod_{p \leq 5} \left( 1 - \frac{b_p^t}{p^{s_*}}\right),$$
where $s_*$ is given by \eqref{sn-def} the choice to use the first three primes $2,3,5$ was obtained after significant trial and error as giving the best numerical results in this range of $t,x,y$.  We have the upper bound
$$
|E_{t,5}(x+iy)| \leq \prod_{p \leq 5} \left(1 + \frac{b_p^t}{p^{\Re s_*}} \right).
$$
When $N \geq N_0$, $0.2 \leq y \leq 1$ and $t = 0.2$, one easily verifies from \eqref{res-bound} that
$$ \Re s_* \geq 1.7143 $$
and hence
$$ |E_{t,5}(x+iy)| \leq 1.635 $$ 
In particular, by Proposition \ref{sweep} we have
$$ E_{t,5}(x+iy) \frac{H_t(x+iy)}{B_t(x+iy)} = E_{t,5}(x+iy) f_t(x+iy) + O_{\leq}( 2.05 \times 10^{-3} ).$$  
The left-hand side remains holomorphic in $x+iy$ (even though $f_t(x+iy)$ has jump discontinuities).  Thus, by the argument principle, it will suffice to show that the left-hand side avoids the negative real axis $(-\infty,0]$ as $x+iy$ traverses the boundary $\partial R$ of the rectangle $R$.  In other words, it will suffice to show that
\begin{equation}\label{star}
 \mathrm{dist}( E_{t,5}(x+iy) f_t(x+iy), (-\infty,0]) > 2.05 \times 10^{-3}
\end{equation}
for $x+iy \in\partial R$.
For future reference we also observe that for any $x+iy \in R$,
the magnitude of $E_{t,5}(x+iy)$ can be bounded below by
\begin{equation}\label{et5-lower}
 |E_{t,5}(x+iy)| \geq \prod_{p \leq 5} \left(1 - \frac{b_p^t}{p^{\Re s_*}} \right) \geq 0.534 
\end{equation} 
and the argument has magnitude at most
\begin{equation}\label{et5-phase}
|\mathrm{arg}(E_{t,5}(x+iy))| \leq \sum_{p \leq 5} \sin^{-1} \frac{b_p^t}{p^{\Re s_*}} \leq 0.553. 
\end{equation} 

Of the four sides of the rectangle $\partial R$, the required estimate \eqref{star} will hold with significant room to spare, and we can proceed using rather crude estimates.  We first attend to the right edge of $\partial R$, in which $N = N_1$.  
From \eqref{ft0} we have
$$
f_t(x+iy) = 1 + O_{\leq}( 0.955 )$$
in this region.  In particular $f_t(x+iy)$ has argument of magnitude at most $\sin^{-1} 0.955 \leq 1.27$.  Combining this with \eqref{et5-lower}, \eqref{et5-phase}, we conclude that $E_{t,5}(x+iy) f_t(x+iy)$ has magnitude at least $0.024$ and argument at most $1.823$ in magnitude, giving the claim \eqref{star} on this side from elementary trigonometry.

Now we attend to the left edge of $\partial R$, in which $x = X_0 - 0.5$ and $0.2 \leq y \leq 1$.  From the calculations for part (a) (see in particular Figure \ref{fig4a}) one can verify that (for instance) $|f_t(x+iy)| \geq 1$ and $|\mathrm{arg} f_t(x+iy)| \leq \frac{\pi}{2}$ in this region.  Combining this with \eqref{et5-lower}, \eqref{et5-phase}, we conclude that $E_{t,5}(x+iy) f_t(x+iy)$ has magnitude at least $0.534$ and argument at most $2.13$ in magnitude, again giving the claim \eqref{star} on this side from elementary trigonometry.

Now we attend to the upper edge of $\partial R$, in which $N_0 \leq N \leq N_1$ and $y=1$.  From \eqref{res-bound} one now has
$$ \Re s_* \geq 2.1143.$$
In particular\footnote{The sum $\sum_{n=2}^{N_1} \frac{b_n^t}{n^{2.1143}}$ can be numerically computed directly, but one could also use Lemma \ref{largen} (using for instance $N_0 = 69098$) to obtain a usable upper bound as well.  Similarly for the other sums of this type that appear in this argument.}
$$ \sum_{n=1}^N \frac{b_n^t}{n^{s_*}} = 1 + O_{\leq}\left( \sum_{n=2}^{N_1} \frac{b_n^t}{n^{2.1143}} \right) = 1 + O_{\leq}( 0.7 ).$$
Meanwhile, from \eqref{gamma-bound} one has
$$ |\gamma| \leq e^{0.02} \left( \frac{x}{4\pi} \right)^{-1/2} \leq 1.03 N^{-1}$$
and from \eqref{kappa-bound} one has
$$ |\kappa| \leq \frac{ty}{2(x-6)} \leq 4 \times 10^{-13}$$
and hence
$$ |\gamma| \frac{n^y}{n^{s_* + \kappa}} = O_{\leq}\left( \frac{1.03}{N n^{1.1142}} \right)$$
and
$$   \gamma \sum_{n=1}^N n^y \frac{b_n^t}{n^{\overline{s_*} + \kappa}} = O_{\leq}\left( \sum_{n=1}^{N_0} \frac{1.03 b_n^t}{N_0 n^{1.1142}} + \sum_{n=N_0+1}^{N_1} \frac{1.03 b_n^t}{n^{2.1142}} \right) = O_{\leq}( 0.1 )$$
and hence
$$ f_t(x+iy) = 1 + O_{\leq}(0.8).$$
In particular $f_t(x+iy)$ has magnitude at least $0.2$ and argument at most $0.928$ in magnitude, hence $E_t(x+iy) f_t(x+iy)$ has magnitude at least $0.1068$ and argument at most $1.481$, at which point the claim follows from elementary trigonometry.

It remains to attend to the lower edge of $\partial R$, in which $N_0 \leq N \leq N_1$ and $y=0.2$.  This is by far the most delicate side of the rectangle for the purposes of verifying \eqref{star}.  We will split the range $[N_0,N_1]$ into a number of subintervals $[N_-,N_+]$ and obtain a uniform lower bound for $\mathrm{dist}( E_{t,5}(x+iy) f_t(x+iy), (-\infty,0])$ when $N$ is in one of these subintervals $[N_-,N_+]$.

Fix $[N_-,N_+] \subset [N_0,N_1]$, suppose that $N \in [N_-,N_+]$, and write $s_* = \sigma + iT$.  We first deal with the $\kappa$ term in the definition of $f_t(x+iy)$ by writing
$$ n^{-\kappa} = 1 + O_{\leq}( n^{|\kappa|}-1)$$
and hence
\begin{equation}\label{ftxy}
 f_t(x+iy) = \sum_{n=1}^N \frac{b_n^t}{n^{\sigma+iT}} + \gamma \sum_{n=1}^N n^y \frac{b_n^t}{n^{\sigma-iT}}
+ O_{\leq}\left( |\gamma| \sum_{n=1}^N n^y \frac{b_n^t}{n^\sigma} (n^{|\kappa|}-1) \right).
\end{equation}
In particular, using \eqref{gamma-bound} to bound
\begin{align*}
|\gamma|n^y  &\leq e^{0.02y} (x_N/4\pi n^2)^{-y/2} \\
&\leq e^{0.02} \left( (N^2 - \frac{1}{80}) / n^2  \right)^{-y/2} \\
&\leq 1.03 (N^2/n^2)^{-0.1}\\
&\leq 1.03 (n/N_-)^{-0.2}.
\end{align*}
we have
$$ E_{t,5}(x+iy) f_t(x+iy) = E_{t,5}(x+iy) \sum_{n=1}^N \frac{b_n^t}{n^{\sigma+iT}} + O_{\leq}\left( 1.03 |E_{t,5}(x+iy)| 
\left|\sum_{n=1}^N (n/N_-)^{0.2} \frac{b_n^t}{n^{\sigma-iT}}\right| \right)
+ O_{\leq}( Z )$$
where
$$ Z \coloneqq 1.644 \sum_{n=1}^N (n/N)^{0.2} \frac{b_n^t}{n^\sigma} (n^{|\kappa|}-1).$$
We can write
$$ E_{t,5}(x+iy) = \sum_{d|D} \frac{\lambda_d}{d^{\sigma+iT}}$$
where $D \coloneqq 2 \times 3 \times 5$ and
$$ \lambda_d \coloneqq \prod_{p|d} (-b_p^t).$$
As a consequence, we have
$$ E_{t,5}(x+iy) \sum_{n=1}^N \frac{b_n^t}{n^{\sigma+iT}} = \sum_{n=1}^{DN} \frac{\beta_{n}}{n^{\sigma+iT}}$$
where
$$ \beta_{n} \coloneqq \sum_{d|n,D} \lambda_d b_{n/d}^t.$$
Thus for instance $\beta_1 = 1$ and $\beta_p = 0$ for $p=2,3,5$, so that the Dirichlet series $\sum_{n=1}^{DN} \frac{\beta_n}{n^{\sigma+iT}}$ is expected to experience less oscillation than the series $\sum_{n=1}^N \frac{b_n^t}{n^{\sigma+iT}}$.

The product of $E_{t,5}(x+iy)$ and $\sum_{n=1}^N (n/N_-)^{0.2} \frac{b_n^t}{n^{\sigma-iT}}$ is not favorable due to the negative sign in the $\sigma-iT$ exponent.  But since $|z| |w| = |z \overline{w}|$, we have
\begin{align*}
1.03 |E_{t,5}(x+iy)| \left|\sum_{n=1}^N (n/N_-)^{0.2} \frac{b_n^t}{n^{\sigma-iT}}\right|  &=
1.03 \left|E_{t,5}(x+iy) \sum_{n=1}^N (n/N_-)^{0.2} \frac{b_n^t}{n^{\sigma+iT}}\right|  \\
&= \left|\sum_{n=1}^{DN} \frac{N_-^{-0.2} \alpha_{n}}{n^{\sigma+iT}}\right|
\end{align*}
where
$$ \alpha_{n} \coloneqq 1.03 \sum_{d|n,D} \lambda_d (n/d)^{0.2} b_{n/d}^t.$$
Note that the coefficients $\alpha_{n}, \beta_n$ are both real.  We now have
\begin{equation}\label{ets}
 E_{t,5}(x+iy) f_t(x+iy) = \sum_{n=1}^{DN} \frac{\beta_n}{n^{\sigma+iT}} + O_{\leq}\left( \left|\sum_{n=1}^{DN} \frac{N_-^{-0.2} \alpha_{n}}{n^{\sigma+iT}}\right| \right) + O_{\leq}(Z).
\end{equation}
A naive application of the triangle inequality (using $\beta_1=1$) would give the lower bound
\begin{equation}\label{eft}
 \mathrm{dist}(E_{t,5}(x+iy) f_t(x+iy), (-\infty,0]) \geq 1 - N_-^{-0.2} \alpha_{1} - \sum_{n=2}^{DN} \frac{|\beta_n| + |N_-^{-0.2} \alpha_{n}|}{n^{\sigma}} - Z.
\end{equation}
As it turns out, this bound is not quite strong enough to be satisfactory for the numerical ranges of parameters we need.  To do better we need to exploit the fact that when the sum $\sum_{n=1}^{DN} \frac{\beta_n}{n^{\sigma+iT}}$ exhibits significant cancellation, then the sum $\sum_{n=1}^{DN} \frac{N^{-0.2} \alpha_{n}}{n^{\sigma+iT}}$ will also. The key tool here is

\begin{lemma}[Improved triangle inequality]  We have
$$ \mathrm{dist}(E_{t,5}(x+iy) f_t(x+iy), (-\infty,0]) \geq 1 - N_-^{-0.2} \alpha_{1} - \sum_{n=2}^{DN} \frac{\max(|\beta_n-N_-^{-0.2} \alpha_{n}|, \frac{1-N_-^{-0.2} \alpha_{1}}{1+N_-^{-0.2} \alpha_{1}} |\beta_n+N_-^{-0.2} \alpha_{n}|)}{n^{\sigma}} - Z.$$
\end{lemma}

\begin{proof}
Write 
$$ Y \coloneqq \sum_{n=2}^{DN} \frac{\max(|\beta_n-N_-^{-0.2} \alpha_n|, \frac{1-N_-^{-0.2} \alpha_1}{1+N_-^{-0.2} \alpha_1} |\beta_n+N_-^{-0.2} \alpha_n|)}{n^{\sigma}}.$$
We may assume that $N_-^{-0.2} \alpha_1+Y < 1$, otherwise the claim is trivial.  By \eqref{ets} and convexity, it suffices to show that
$$
\mathrm{dist}( \sum_{n=1}^{DN} \frac{\beta_n + e^{i\theta} N_-^{-0.2} \alpha_n}{n^{\sigma+iT}}, (-\infty,0]) \geq 1 - N_-^{-0.2} \alpha_1 - Y$$
for all phases $\theta \in \R$.  We may write
\begin{align*}
\sum_{n=1}^{DN} \frac{\beta_n + e^{i\theta} N_-^{-0.2} \alpha_n}{n^{\sigma+iT}} &= 1 + e^{i\theta} N_-^{-0.2} \alpha_1 + O_{\leq}\left( \sum_{n=2}^{DN}
\frac{|\beta_n + e^{i\theta} N_-^{-0.2} \alpha_n|}{n^{\sigma}} \right) \\
&=
(1 + e^{i\theta} N_-^{-0.2} \alpha_1) \left(1 + O_{\leq}\left( \sum_{n=2}^{DN}
\frac{|\beta_n + e^{i\theta} N_-^{-0.2} \alpha_n|/|1+e^{i\theta} N_-^{-0.2} \alpha_1|}{n^{\sigma}} \right) \right).
\end{align*}
By the cosine rule, we have
$$ \left(|\beta_n + e^{i\theta} N_-^{-0.2} \alpha_n| / |1 + e^{i\theta} N_-^{-0.2} \alpha_1|\right)^2 = \frac{\beta_n^2 + N_-^{-0.4} \alpha_n^2 + 2 N_-^{-0.2} \alpha_n \beta_n \cos \theta}{1 + N_-^{-0.4} \alpha_1^2 + 2 N_-^{-0.2} \alpha_1 \cos \theta}.$$
This is a fractional linear function of $\cos \theta$ with no poles in the range $[-1,1]$ of $\cos \theta$.  Thus this function is monotone on this range and attains its maximum at either $\cos \theta=+1$ or $\cos \theta = -1$.  We conclude that
$$ \frac{|\beta_n + e^{i\theta} N_-^{-0.2} \alpha_n|}{|1 + e^{i\theta} N_-^{-0.2} \alpha_1|} \leq \max\left( \frac{|\beta_n-N_-^{-0.2} \alpha_n|}{1-N_-^{-0.2} \alpha_1}, \frac{|\beta_n+N_-^{-0.2} \alpha_n|}{1+N_-^{-0.2} \alpha_1} \right)$$
and thus
$$ \sum_{n=2}^{DN} \frac{|\beta_n + e^{i\theta} N_-^{-0.2} \alpha_n|/|1+e^{i\theta} N_-^{-0.2} \alpha_1|}{n^{\sigma}} \leq \frac{1}{1-N_-^{-0.2} \alpha_1} Y.$$
We conclude from the triangle inequality that
$$ \sum_{n=1}^{DN} \frac{\beta_n + e^{i\theta} N_-^{-0.2} \alpha_n}{n^{\sigma+iT}} = 1 + e^{i\theta} N_-^{-0.2} \alpha_1 + O_{\leq}\left( \frac{|1+e^{i\theta} N_-^{-0.2} \alpha_1|}{1-N_-^{-0.2} \alpha_1} Y \right).$$
By further application of the triangle inequality
\begin{align*}
\mathrm{dist}( \sum_{n=1}^{DN} \frac{\beta_n + e^{i\theta} N_-^{-0.2} \alpha_n}{n^{\sigma+iT}}, (-\infty,0]) &\geq \mathrm{dist}( 1 + e^{i\theta} N_-^{-0.2} \alpha_1, (-\infty,0]) - \frac{|1+e^{i\theta} N_-^{-0.2} \alpha_1|}{1-N_-^{-0.2} \alpha_1} Y \\
&= |1 + e^{i\theta} N_-^{-0.2} \alpha_1| - \frac{|1+e^{i\theta} N_-^{-0.2} \alpha_1|}{1-N_-^{-0.2} \alpha_1} Y  \\
&= |1 + e^{i\theta} N_-^{-0.2} \alpha_1| \left(1 - \frac{Y}{1-N_-^{-0.2} \alpha_1} \right) \\
&\geq (1-N_-^{-0.2} \alpha_1) \left(1 - \frac{Y}{1-N_-^{-0.2} \alpha_1} \right) \\
&= 1 - N_-^{-0.2} \alpha_1 - Y
\end{align*}
as desired, where we have used the fact that $1+e^{i\theta} N_-^{-0.2} \alpha_1$ lies to the right of the imaginary axis (so that the closest element of $(-\infty,0]$ is the origin).
\end{proof}

Bounding $DN \leq DN_+$, we see that in order to establish \eqref{star} in the range $N \in [N_-, N_+]$, it suffices to verify the inequality
$$ 1 - N_-^{-0.2} \alpha_1 - \sum_{n=2}^{DN_+} \frac{\max(|\beta_n-N_-^{-0.2} \alpha_n|, \frac{1-N_-^{-0.2} \alpha_1}{1+N_-^{-0.2} \alpha_1} |\beta_n+N^{-0.2} \alpha_n|)}{n^{\sigma}} - Z
\geq 2.14 \times 10^{-3}.$$
From \eqref{res-bound}, \eqref{xnn} one has 
\begin{align*}
 \sigma &\geq \frac{1+y}{2} +\frac{t}{4} \log \frac{x_N}{4\pi} - \frac{t}{2x_N^2} \left(1-3y+\frac{4y(1+y)}{x_N^2}\right)_+ \\
&= 0.6 + \frac{1}{20} \log \frac{x_N}{4\pi} - \frac{1}{10 x_N^2} \left( 0.4 + \frac{0.96}{x_N^2} \right) \\
&= 0.6 + \frac{1}{20} \log (N^2 - \frac{1}{80}) - \frac{1}{10 x_{N_0}^2} \left( 0.4 + \frac{0.96}{x_{N_0}^2} \right) \\
&\geq \sigma_{N_-}
\end{align*}
where
$$ \sigma_{N_-} \coloneqq 0.599 + \frac{1}{10} \log N_- $$
while from \eqref{kappa-bound}, \eqref{xnn} one has
$$ |\kappa| \leq \frac{ty}{2(x_N-6)} \leq \frac{0.02}{x_{N_0}-6} \leq 4 \times 10^{-13}$$
and hence it suffices to show that
$$ F_{N_-,N_+} - Z_{N_-,N} \geq 2.14 \times 10^{-3}$$
where
$$ F_{N_-,N_+} \coloneqq 1 - N_-^{-0.2} \alpha_1 - \sum_{n=2}^{DN_+} \frac{\max(|\beta_n-N_-^{-0.2} \alpha_n|, \frac{1-N_-^{-0.2} \alpha_1}{1+N_-^{-0.2} \alpha_1} |\beta_n+N_-^{-0.2} \alpha_n|)}{n^{\sigma_{N_-}}}$$
and
$$ Z_{N_-,N} \coloneqq 1.644 \sum_{n=1}^N (n/N)^{0.2} \frac{b_n^t}{n^{\sigma_{N_-}}} (n^{4 \times 10^{-13}}-1).$$
Since $\sigma_{N_-} \geq 1.714$, we may crudely bound
$$ Z_{N_-,N} \leq 1.644 \sum_{n=1}^{N_1} \frac{b_n^t}{n^{1.714}} (n^{4 \times 10^{-13}}-1) \leq 10^{-10}$$
so it will suffice to show that
$$ F_{N_-,N_+} \geq 2.15 \times 10^{-3}$$
for a collection of intervals $[N_-,N_+]$ covering $[N_0, N_1]$.  This can be done by \emph{ad hoc} numerical experimentation; for instance, one can calculate that
\begin{align*}
F_{69098, 8 \times 10^4} &= 0.0263\dots \\
F_{8 \times 10^4, 1.1 \times 10^5} &= .0470\dots \\
F_{1.1 \times 10^5, 2.2 \times 10^5} &= 0.093\dots \\
F_{2.2 \times 10^5, 1.5 \times 10^6} &= 0.060\dots .
\end{align*}

\section{Asymptotic results}\label{asymptotic-sec}

In this section we use the effective estimates from Theorem \ref{eff} to obtain asymptotic information about the function $H_t$, which improves (and makes more effective) the results of Ki, Kim, and Lee \cite{kkl}, by establishing Theorem \ref{Zero}.

We begin with 

\begin{proposition}[Preliminary asymptotics]\label{asymp}  Let $0 < t \leq 1/2$, $x \geq 200$, and $-10 \leq y \leq 10$.
\begin{itemize}
\item[(i)]  If $x \geq \exp(\frac{C}{t})$ for a sufficiently large absolute constant $C$, then
$$ H_t(x+iy) = (1 + O(x^{-ct})) M_t\left(\frac{1+y-ix}{2}\right) + (1 + O(x^{-ct})) M_t\left(\frac{1-y+ix}{2} \right) $$
for an absolute constant $c>0$, where $M_t$ is defined in \eqref{Mt-def}.  
\item[(ii)]  If instead we have $3 \leq y \leq 4$ and $x \geq C$ for a sufficiently large absolute constant $C$, then
$$ H_t(x+iy) = (1 + O_{\leq}(0.7)) M_t\left(\frac{1+y-ix}{2}\right).$$
\item[(iii)]  If $x = x_0 + O(1)$ for some $x_0 \geq 200$, then
$$ H_t(x+iy) = O\left( x_0^{O(1)} \left|M_t\left(\frac{1+y-ix}{2}\right)\right| \right) = O\left( x_0^{O(1)} \left|M_t\left(\frac{ix_0}{2}\right)\right| \right).$$
\end{itemize}
\end{proposition}

\begin{proof}  We begin with (i).  Since $H_t = H_t^*$ and $M_t = M_t^*$, we may assume without loss of generality that $y \geq 0$.
Using \eqref{lambda-def}, \eqref{bo-def} we may write the desired estimate as
$$ \frac{H_t(x+iy)}{B_t(x+iy)} = 1 + O(x^{-ct}) + \gamma.$$
We apply Theorem \ref{eff}.  (Strictly speaking, the estimates there required $y \leq 1$ rather than $y \leq 10$; however, as remarked at the beginning of Section \ref{initial-sec}, all the estimates in that section would continue to hold under this weaker hypothesis if one adjusted all the numerical constants appropriately.)  This gives
\begin{equation}\label{exp}
\frac{H_t(x+iy)}{B_t(x+iy)} = \sum_{n=1}^N \frac{b_n^t}{n^{s_*}} + \gamma \sum_{n=1}^N n^y \frac{b_n^t}{n^{\overline{s_*} + \kappa}} + O_{\leq}\left( e_A + e_B + e_{C,0} \right)
\end{equation}
where
\begin{align*}
\gamma &= O( x^{-y/2} ) \\
\kappa &= O(x^{-1}) \\
\Re s_* &\geq \frac{1+y}{2} + \frac{t}{2} \log \frac{x}{4\pi}  - O(x^{-2}) \\
e_A &= O\left( x^{-y/2} \sum_{n=1}^N b_n^t n^{-\frac{1-y}{2} - \frac{t}{2} \log \frac{x}{4\pi} - O(x^{-1})} \frac{\log^2 x}{x} \right) \\ 
e_B &= O\left( \sum_{n=1}^N b_n^t n^{-\frac{1+y}{2} - \frac{t}{2} \log \frac{x}{4\pi} + O(x^{-1})} \frac{\log^2 x}{x} \right) \\ 
e_{C,0} &= O\left( x^{-\frac{1+y}{4}} \right)
\end{align*}
Since $N = O(x^{1/2})$, we have $x^{-y/2} n^{y} = O(1)$ and $n^{O(x^{-1})} = O(1)$ for all $1 \leq n \leq N$.  We conclude that
$$\frac{H_t(x+iy)}{B_t(x+iy)} = 1 + \gamma + O\left( \frac{\log^2 x}{x} + \sum_{n=2}^N \frac{b_n^t}{n^{\frac{1+y}{2} + \frac{t}{2} \log \frac{x}{4\pi}}} + x^{-\frac{1+y}{4}} \right) $$
so it will suffice (for $c$ small enough) to show that
$$ \sum_{n=2}^N \frac{b_n^t}{n^{\frac{1+y}{2} + \frac{t}{2} \log \frac{x}{4\pi}}} = O( x^{-ct} ).$$
By \eqref{bn-def} we can write the left-hand side as
$$ \sum_{n=2}^N \frac{1}{n^{\frac{1+y}{2} + \frac{t}{2} \log \frac{x}{4\pi \sqrt{n}}}} = O( x^{-ct} ).$$
For $2 \leq n \leq N$, we have
$$ \frac{1+y}{2} + \frac{t}{2} \log \frac{x}{4\pi \sqrt{n}} \geq c t \log x $$
for some absolute constant $c>0$.  By the integral test, the left-hand side is then bounded by
$$ \frac{1}{2^{c t \log x}} + \int_2^\infty \frac{1}{u^{c t \log x}}\ du$$
which, for $x \geq \exp(C/t)$ and $C$ large, is bounded by $O(2^{-ct\log x})$.  The claim then follows after adjusting $c$ appropriately.

Now we prove (ii).  As before we have the expansion \eqref{exp}.  We have
\begin{align*}
\gamma \sum_{n=1}^N n^y \frac{b_n^t}{n^{\overline{s_*} + \kappa}} &= O\left( x^{-y/2} \sum_{n=1}^N \frac{b_n^t}{n^{\frac{1-y}{2} + \frac{t}{2} \log \frac{x}{4\pi}}} \right) \\
&= O\left( x^{-y/2} \sum_{n=1}^N n^{\frac{y-1}{2}} \right) \\
&= O(x^{-\frac{y-1}{4}});
\end{align*}
similar arguments give $e_A = O( \frac{\log^2 x}{x} x^{\frac{1-y}{4}} )$, while
\begin{align*}
e_B &= O\left( \frac{\log^2 x}{x} \sum_{n=1}^N b_n^t n^{-\frac{1+y}{2}-\frac{t}{2} \log \frac{x}{4\pi}} \right) \\
&= O\left( \frac{\log^2 x}{x} \sum_{n=1}^N n^{-2} \right) \\
&= O\left( \frac{\log^2 x}{x} \right).
\end{align*}
We conclude that
\begin{align*}
\frac{H_t(x+yi)}{B_t(x+yi)} &= \sum_{n=1}^N \frac{b_n^t}{n^{s_*}} + O( x^{-\frac{y-1}{4}} ) \\
&= 1 + O_{\leq}\left( \sum_{n=2}^N n^{-\frac{1+y}{2} - \frac{t}{2} \log \frac{x}{4\pi} - O(x^{-1})} \right) + O( x^{-\frac{y-1}{4}} ) \\
&= 1 + O_{\leq}\left( \sum_{n=2}^N n^{-2} \right) + O( x^{-1/2}) \\
&= 1 + O_{\leq}\left( \frac{\pi^2}{6} - 1 \right) + O( x^{-1/2} ) \\
&= 1 + O_{\leq}( 0.7 )
\end{align*}
as claimed, if $x \geq C$ for $C$ large enough.

Finally, we prove (iii).  Again our starting point is \eqref{exp}.  The right-hand side can be bounded crudely by $O( x^{O(1)}) = O(x_0^{O(1)})$, hence
$$ H_t(x+iy) = O\left( x_0^{O(1)} \left|M_t\left( \frac{1+y+ix}{2} \right)\right| \right).$$
However, from \eqref{Mt-def}, \eqref{M-def}, \eqref{alpha-form} it is not hard to see that the log-derivative of $M_t(s)$ is of size $O( \log x_0 )$ in the region $s = \frac{ix_0}{2} + O(1)$.  Thus
$$ \left|M_t\left( \frac{1+y+ix}{2} \right)\right| = O\left( x_0^{O(1)} \left|M_t\left( \frac{ix_0}{2} \right)\right| \right),$$
giving the claim.
\end{proof}

To understand the behavior of $M_t(x+iy)$, we make the following simple observations:

\begin{lemma}[Behavior of $M_t$]\label{mtform}  Let $0 < t \leq 1/2$, let $x_* > 0$ be sufficiently large, and let $x+iy = x_* + O(1)$.  Then
$$ M_t\left(\frac{1+y+ix}{2}\right) = M_t\left(\frac{1+ix_*}{2}\right) \exp\left( (i(x-x_*)+y) \left(\frac{1}{4} \log \frac{x_*}{4\pi} + \frac{\pi i}{8}\right) + O\left( \frac{\log x_*}{x_*}\right) \right).$$
Also, there is a continuous branch of $\mathrm{arg} M_t\left(\frac{1+ix_*}{2}\right)$ for all large real $x_*$ such that
$$ \mathrm{arg} M_t\left(\frac{1+ix_*}{2}\right) = \frac{t \pi}{16} \log \frac{x_*}{4\pi} + \frac{7\pi}{8} 
+ \frac{x_*}{4} \log \frac{x_*}{4\pi} - \frac{x_*}{4} + O( \frac{\log x_*}{x_*} ).$$
\end{lemma}

\begin{proof}
By \eqref{Mt-def}, \eqref{alpha-def}, the log-derivative of $M_t$ is given by
\begin{equation}\label{mt-deriv}
 \frac{M'_t}{M_t} = \alpha + \frac{t}{2} \alpha \alpha'.
\end{equation}
For $s = \frac{ix_*}{2} + O(1)$, we have from \eqref{alpha-form} that
\begin{equation}\label{alpha-ex}
\alpha(s) = \frac{1}{2} \log \frac{x_*}{4\pi} + \frac{\pi i}{4} + O\left( \frac{1}{x_*}\right) 
\end{equation}
and from this and \eqref{alpha-deriv-bound} we conclude that
$$
 \frac{M'_t(s)}{M_t(s)} = \frac{1}{2} \log \frac{x_*}{4\pi} + \frac{\pi i}{4} + O\left( \frac{\log x_*}{x_*}\right) $$
whenever $s = \frac{ix_*}{2} + O(1)$.  The first claim then follows by applying the fundamental theorem of calculus to a branch of $\log M_t$.

For the second claim, we calculate
\begin{align*}
\mathrm{arg} M_t\left(\frac{1+ix_*}{2}\right) &= \frac{t}{4} \Im(\alpha\left(\frac{1+ix_*}{2}\right)^2) + \pi - \frac{x_*}{4} \log \pi + \Im \left( \frac{-1+ix_*}{4} \log \frac{1+ix_*}{4} - \frac{1+ix_*}{4}\right) \\
&= \frac{t}{4} \left(\frac{\pi}{4} \log \frac{x_*}{4\pi} + O(\frac{\log x_*}{x_*})\right) + \pi - \frac{x_*}{4} \log \pi 
+ \Im \left( \frac{-1+ix_*}{4} \left(\log \frac{x_*}{4} + \frac{i\pi}{2} - \frac{i}{x_*} + O\left(\frac{1}{x_*^2}\right) \right) \right) - \frac{x_*}{4} \\
&= \frac{t \pi}{16} \log \frac{x_*}{4\pi} + \pi - \frac{x_*}{4} \log \pi - \frac{x_*}{4}
+ \frac{x_*}{4} \log \frac{x_*}{4} - \frac{\pi}{8} + O\left( \frac{\log x_*}{x_*} \right) \\
&= \frac{t \pi}{16} \log \frac{x_*}{4\pi} + \frac{7\pi}{8} 
+ \frac{x_*}{4} \log \frac{x_*}{4\pi} - \frac{x_*}{4} + O\left( \frac{\log x_*}{x_*} \right) 
\end{align*}
as desired.
\end{proof}

Now we can prove Theorem \ref{Zero}.  We begin with (ii).  Let $n \geq \exp( \frac{C}{t})$, and suppose that $x+iy = x_n + O(1)$.  
By Proposition \ref{asymp}(i) and Lemma \ref{mtform} we have
\begin{equation}\label{htap}
\begin{split}
 H_t(x+iy) &= \overline{M_t\left(\frac{1+ix_n}{2}\right)} \exp\left( (-i(x-x_n)+y) \left(\frac{1}{4} \log \frac{x_n}{4\pi} - \frac{\pi i}{8}\right) + O( x_n^{-ct} )\right)\\
&\quad  + M_t\left(\frac{1+ix_n}{2}\right) \exp\left( (i(x-x_n)-y) \left(\frac{1}{4} \log \frac{x_n}{4\pi} + \frac{\pi i}{8}\right) + O( x_n^{-ct} )\right).
\end{split}
\end{equation}
From Lemma \ref{mtform} and \eqref{lip} one has
$$ \mathrm{arg} M_t\left(\frac{1+ix_n}{2}\right)  = -\frac{\pi}{2} + O\left( \frac{\log x_n}{x_n} \right)\hbox{ mod } \pi $$
and hence 
\begin{equation}\label{ma}
\overline{M_t\left(\frac{1+ix_n}{2}\right)} = - \exp\left( O( \frac{\log x_n}{x_n} ) \right) M_t\left(\frac{1+ix_n}{2}\right).
\end{equation}
If we now make the further assumption $y = O\left( \frac{1}{\log x_n}\right)$, we can thus simplify the above approximation as
\begin{equation}\label{ht-eff}
\begin{split}
 H_t(x+iy) &= - M_t\left(\frac{1+ix_n}{2}\right) e^{-\pi (x-x_n)/8} \exp\left( (-i(x-x_n)+y) \frac{1}{4} \log \frac{x_n}{4\pi} + O( |y| \log x_n + x_n^{-ct} ) \right)\\
&\quad + M_t\left(\frac{1+ix_n}{2}\right) e^{-\pi (x-x_n)/8} \exp\left( (i(x-x_n)-y) \frac{1}{4} \log \frac{x_n}{4\pi} + O( |y| \log x_n + x_n^{-ct} ) \right)\\
&= 2 i M_t\left(\frac{1+ix_n}{2}\right) e^{-\pi (x-x_n)/8} \left( \sin\left(\frac{x+iy-x_n}{4} \log \frac{x_n}{4\pi}\right) + O( |y| \log x_n + x_n^{-ct} )\right ).
\end{split}
\end{equation}
In particular, if $x+iy$ traverses the circle $\{ x_n + \frac{c}{\log n} e^{i\theta}: 0 \leq \theta \leq 2\pi\}$ once anti-clockwise and $c$ is small enough, the quantity $H_t(x+iy)$ will wind exactly once around the origin, and hence by the argument principle there is precisely one zero of $H_t$ inside this circle.  As the zeroes of $H_t$ are symmetric around the real axis, this zero must be real.  This proves (ii).

Now we prove (i).
Suppose that $H_t(x+iy)=0$ and $x \geq \exp(\frac{C}{t})$.  We can assume $|y| \leq 1$ since it is known (e.g., from \cite[Theorem 13]{debr}) that there are no zeroes with $|y|>1$.  

Let $n$ be a natural number that minimises $|x-x_n|$, then $x = x_n+O\left(\frac{1}{\log x_n}\right)$ since the derivative of the left-hand side of \eqref{lip} in $x_n$ is comparable to $\log x_n$.  From \eqref{htap} we have
\begin{align*}
0 &= \overline{M_t\left(\frac{1+ix_n}{2}\right)} \exp\left( (-i(x-x_n)+y) \left(\frac{1}{4} \log \frac{x_n}{4\pi} - \frac{\pi i}{8}\right) + O( x_n^{-ct} )\right) \\
&\quad + M_t(\frac{1+ix_n}{2}) \exp\left( (i(x-x_n)-y) \left(\frac{1}{4} \log \frac{x_n}{4\pi} + \frac{\pi i}{8}\right) + O( x_n^{-ct} )\right).
\end{align*}
Thus both summands on the right-hand side have the same magnitude, which on taking logarithms and canceling like terms implies that
$$ y \frac{1}{4} \log \frac{x_n}{4\pi} + O( x_n^{-ct} ) = -y \frac{1}{4} \log \frac{x_n}{4\pi} + O( x_n^{-ct} )$$
and hence $y = O\left( \frac{x_n^{-ct}}{\log x_n} \right)$.  We can now apply \eqref{ht-eff} to conclude that
$$ \sin\left(\frac{x+iy-x_n}{4} \log \frac{x_n}{4\pi}\right) + O( x_n^{-ct} ) = 0$$
which (when combined with the hypothesis that $|x-x_n|$ is minimal) forces $x - x_n = O\left( \frac{x_n^{-ct}}{\log x_n} \right)$.  This gives the claim.

Next, we prove (iii).  In view of parts (i) and (ii), and adjusting $C$ if necessary, we may assume that $X$ takes the form $X = x_n+\frac{c}{\log x_n}$ for some $n \geq \exp( \frac{C}{t})$.  By the argument principle, $N_t(X)$ is equal to $\frac{-1}{2\pi}$ times the variation in the argument of $H_t$ on the boundary of the rectangle $\{ x+iy: 0 \leq x \leq X; -3 \leq y \leq 3 \}$ traversed clockwise, since there are no zeroes with imaginary part of magnitude greater than one.  By compactness, the variation on the left edge $\{ iy: -3 \leq y \leq 3 \}$ is $O(1)$, and similarly for any fixed portion $\{ x+3i: 0 \leq x \leq C \}$ of the upper edge.  From Proposition \ref{asymp} (and \eqref{htap}), we see that the variation of $H_t(x+iy) / M_t(\frac{1+y-ix}{2})$ on the remaining upper edge $\{ x+3i: C \leq x \leq X \}$ and on the top half $\{ X+iy: 0 \leq y \leq 3 \}$ of the right edge are both equal to $O(1)$.  Since $H_t = H_t^*$, the variation on the lower half of the rectangle is equal to that of the upper half. We thus conclude that
$$ N_t(X) = -\frac{1}{\pi} \mathrm{arg} M_t\left(\frac{1-iX}{2}\right)+ O(1)$$
where we use a continuous branch of the argument of $M_t\left(\frac{1-iX}{2}\right)$ that is bounded at $3i$.  The claim now follows from Lemma \ref{mtform} (using $M_t = M_t^*$ to work with $\frac{1+iX}{2}$ instead of $\frac{1-iX}{2}$).

Finally, we prove (iv).  From the Hadamard factorization theorem as in the proof of Proposition \ref{dynam} we have
\begin{equation}\label{hata}
 \frac{H'_t(z)}{H_t(z)} = \sum_{n>0} \left(\frac{1}{z-z_n} + \frac{1}{z+z_n}\right)
\end{equation}
where the zeroes of $H_t$ are indexed in pairs $\pm z_n$.
Setting $z = X+4i$, we see from Proposition \ref{asymp} and the generalized Cauchy integral formula that the logarithmic derivative of $H_t(x+iy)/M_t\left(\frac{1+y-ix}{2}\right)$ is equal to $O(1)$ at $X+4i$ for all sufficiently large $X$, and hence for all $X$ by symmetry and compactness.  On the other hand, from Stirling's formula (or the logarithmic growth of the digamma function) one easily verifies that the logarithmic derivative of $M_t\left(\frac{1+y-ix}{2}\right)$  is equal to $O( \log(2+X) )$ at $X+4i$.  Hence $\frac{H'_t(X+4i)}{H_t(X+4i)} = O(\log(2+X))$.  Taking imaginary parts, we conclude that
$$ \sum_{n>0} -\frac{4-y_n}{(X-x_n)^2 + (4-y_n)^2} - \frac{4+y_n}{(X+x_n)^2 + (4+y_n)^2} = O(\log(2+X))$$
where we write $z_n = x_n + iy_n$; equivalently one has
$$ \sum_n \frac{(4-y_n)}{(X-x_n)^2 + (4-y_n)^2} = O(\log(2+X))$$
where the sum now ranges over all zeroes, including any at the origin. Since $|y_n| \leq 1$, every zero in $[X,X+1]$ makes a contribution of at least $\frac{1}{100}$ (say).  As the summands are all positive, the first part of claim (iv) follows.  To prove the second part, we may assume by compactness that $x \geq C$.  Repeating the proof of (iii), and reduce to showing that the variation of $\mathrm{arg} H_t$ on the short vertical interval $\{ X+iy: 0 \leq y \leq 3 \}$ is $O( \log X )$.  If we let $\theta$ be a phase such that $e^{i\theta} H_t(X+3i)$ is real and positive, we see that this variation is at most $\pi(m+1)$, where $m$ is the number of zeroes of $\Re (e^{i\theta} H_t(X+yi))$ for $0 \leq y \leq 3$, since every increment of $\pi$ in $\mathrm{arg} e^{i\theta} H_t$ must be accompanied by at least one such zero.  As $H_t = H_t^*$, this is also the number of zeroes of $e^{i\theta} H_t(X+yi) + e^{-i\theta} H_t(2X - (X+yi))$.  On the other hand, from Proposition \ref{asymp}(ii), (iii) and Jensen's formula we see that the number of such zeroes is $O( \log X )$, and the claim follows.

\begin{remark}\label{time-complexity}  Theorem \ref{Zero} gives good control on $H_t(x+iy)$ whenever $x \geq \exp( C/t )$.  As a consequence (and assuming for sake of argument that the Riemann hypothesis holds), then for any $\Lambda_0 > 0$, the bound $\Lambda \leq \Lambda_0$ should be numerically verifiable in time $O( \exp( O(1/\Lambda_0) ))$, by applying the arguments of previous sections with $t$ and $y$ set equal to small multiples of $\Lambda_0$.  We leave the details to the interested reader.
\end{remark}

\begin{remark} Our discussion here will be informal.  In view of the results of \cite{csv}, it is expected that the zeroes $z_j(t)$ of $H_t(x+iy)$ should evolve according to the system of ordinary differential equations
$$ \frac{d}{dt} z_k(t) = 2 \sum_{j \neq k}^{\prime} \frac{1}{z_k(t) - z_j(t)}$$
where the sum is evaluated in a suitable principal value sense, and one avoids those times where the zero $z_k(t)$ fails to be simple; see \cite[Lemma 2.4]{csv} for a verification of this in the regime $t > \Lambda$.  In view of the Riemann-von Mangoldt formula (as well as variants such Corollary \ref{Zero}, it is expected that the number of zeroes in any region of the form $\{ x+iy: x+iy = x_* + O(1) \}$ for large $x_*$ should be of the order of $\log x_*$.  As a consequence, we expect a typical zero $z_k(t)$ to move with speed $O( \log |z_k(t)| )$, although one may occasionally move much faster than this if two zeroes are exceptionally close together, or less than this if the zeroes are close to being evenly spaced.  As a consequence, if the Riemann hypothesis fails and there is a zero $x+iy$ of $H_0$ with $y$ comparable to $1$, it should take time comparable to $\frac{1}{\log x}$ for this zero to move towards the real axis, leading to the heuristic lower bound $\Lambda \gg \frac{1}{\log x}$.  Thus, in order to obtain an upper bound $\Lambda \leq \Lambda_0$, it will probably be necessary to verify that there are no zeroes $x+iy$ of $H_0$ with $y$ comparable to $1$ and $|x| \leq c \log \frac{1}{\Lambda}$ for some small absolute constant $c>0$.  This suggests that the time complexity bound in Remark \ref{time-complexity} is likely to be best possible (unless one is able to prove the Riemann hypothesis, of course).

In \cite[Lemma 2.1]{csv} it is also shown that the velocity of a given zero $z(t)$ is given by the formula
$$ \frac{d}{dt} z(t) = \frac{H''_t(z(t))}{H'_t(z(t))}$$
assuming that the zero is simple.  By using the asymptotics in Proposition \ref{asymp} and Corollary \ref{Zero} together with the generalized Cauchy integral formula to then obtain asymptotics for $H'_t$ and $H''_t$, it is possible to show that for the zeroes $x(t)$ that are real and larger than $\exp(C/t)$, and move leftwards with velocity
$$ \frac{d}{dt} x(t) = - \frac{\pi}{4} + O( x^{-ct} );$$
we leave the details to the interested reader.  

\begin{figure}[ht!]
  \includegraphics[width=0.9\linewidth]{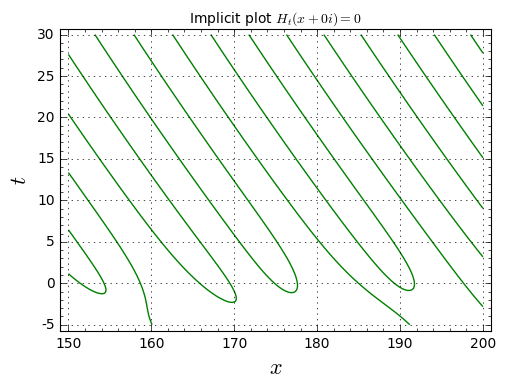}
  \caption{Real zeros moving up leftwards and getting `solidified'.}
\end{figure}

\begin{figure}[ht!]
  \includegraphics[width=0.9\linewidth]{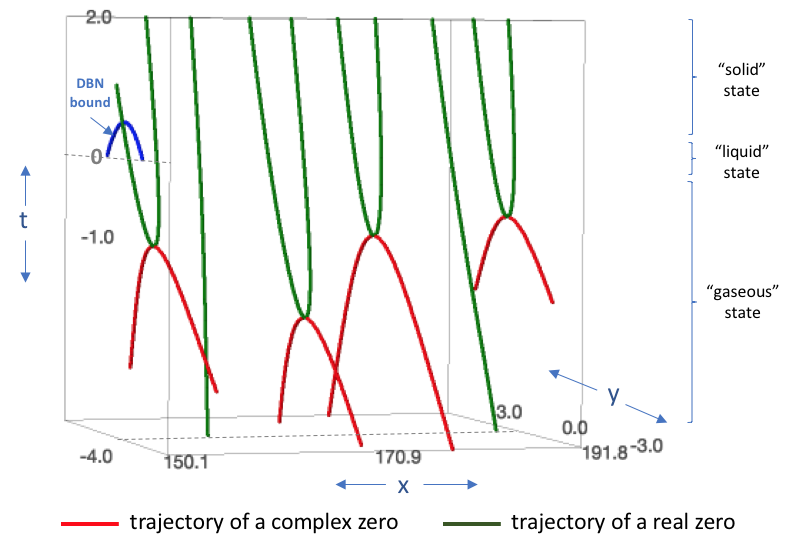}
  \caption{Actual trajectories of some real and complex zeros.}
\end{figure}

\end{remark}

\section{Further numerical results}\label{further-sec}

By Theorem \ref{Zero}, one can verify the second hypothesis of Theorem \ref{ubc-0} when $X \geq \exp(C/t_0)$ for a large constant $C$.  If we ignore for sake of discussion the third hypothesis of Theorem \ref{ubc-0} (which turns out to be relatively easy to verify numerically in practice), this suggests that one can obtain a bound of the form $\Lambda \leq O(t_0)$ provided that one can verify the Riemann hypothesis up to a height $\exp(C/t_0)$.  In other words, if one has numerically verified the Riemann hypothesis up to a large height $T$, this should soon lead to a bound of the form $\Lambda \leq O \left( \frac{1}{\log T} \right )$.

Aside from improving the implied constant in this bound, it does not seem easy to improve this sort of implication without a major breakthrough on the Riemann hypothesis (such as a massive expansion of the known zero-free regions for the zeta function inside the critical strip).  We shall justify this claim heuristically as follows. Suppose that there was a counterexample to the Riemann hypothesis at a large height $T$, so that $H_0(2T + iy) = 0$ for some positive $y$, which for this discussion we will take to be comparable to $1$.  The Riemann von Mangoldt formula indicates that the number of zeroes of $H_0$ within a bounded distance of this zero should be comparable to $\log T$; the majority of these zeroes should obey the Riemann hypothesis and thus stay at roughly unit distance from our initial zero $2T+iy$.  Proposition \ref{dynam} then suggests that as time $t$ advances, this zero should move at speed comparable to $\log T$.  Thus one should not expect this zero to reach the real axis until a time comparable to $\frac{1}{\log T}$.  This heuristic analysis therefore indicates that it is unlikely that one can significantly improve the bound $\Lambda \leq O \left( \frac{1}{\log T} \right )$ without being able to exclude significant violations of the Riemann hypothesis at height $T$.

The table below collects some numerical results verifying the second two hypotheses of Theorem \ref{ubc-0} for larger values of $X$, and smaller values of $t_0,y_0$, than were considered in Section \ref{newup-sec}.  This leads to improvements to the bound $\Lambda \leq 0.22$ conditional on the assumption that the Riemann Hypothesis can be numerically verified beyond the height $T \approx 3.06 \times 10^{10}$ used in Section \ref{newup-sec}.  For instance, the final row of the table implies that one has the bound $\Lambda \leq 0.1$ assuming that the Riemann hypothesis is verified up to the height $T \approx 4.5 \times 10^{21}$.  Note that this is broadly consistent with the previous heuristic that the upper bound on $\Lambda$ is proportional to $\frac{1}{\log T}$.

\begin{table}[ht!]
  \begin{center}
    \caption{Conditional $\Lambda$ Results}
    \label{tab:table1}
    \begin{tabular}{l|r|r|r|c|r|c} 
      $X$ & $t_{0}$ & $y_{0}$ & $\Lambda$ & $\textbf{Winding Number}$ & $N_{0}$ & $|f_t(x+iy)|$ lower bound\\
      \hline
      $2 \times 10^{12} + 129093$ & 0.198 & 0.15492 & 0.21 & 0 & 398942 & 0.0341\\
      $5 \times 10^{12} + 194858$ & 0.186 & 0.16733 & 0.20 & 0 & 630783 & 0.0376\\
      $2 \times 10^{13} + 131252$ & 0.180 & 0.14142 & 0.19 & 0 & 1261566 & 0.0349\\
      $6 \times 10^{13} + 123375$ & 0.168 & 0.15492 & 0.18 & 0 & 2185096 & 0.0377\\
      $3 \times 10^{14} + 188911$ & 0.161 & 0.13416 & 0.17 & 0 & 4886025 & 0.0369\\
      $2 \times 10^{15} + 122014$ & 0.153 & 0.11832 & 0.16 & 0 & 12615662 & 0.0532\\
      $7 \times 10^{15} + 68886$ & 0.139 & 0.14832 & 0.15 & 0 & 23601743 & 0.0350\\
      $6 \times 10^{16} + 156984$ & 0.132 & 0.12649 & 0.14 & 0 & 69098829 & 0.0307\\
      $6 \times 10^{17} + 88525$ & 0.122 & 0.12649 & 0.13 & 0 & 218509686 & 0.0347\\
      $9 \times 10^{18} + 35785$ & 0.113 & 0.11832 & 0.12 & 0 & 846284375 & 0.0318\\
      $2 \times 10^{20} + 66447$ & 0.102 & 0.12649 & 0.11 & 0 & 3989422804 & 0.0305\\
      $9 \times 10^{21} + 70686$ & 0.093 & 0.11832 & 0.1 & 0 & 26761861742 & 0.0321\\
    \end{tabular}
  \end{center}
\end{table}

The selection of parameters in this table proceeded as follows.  One first located parameters $t_0, y_0, N_0$ (with the quantity $\Lambda = t_0 + \frac{1}{2} y_0^2$ as small as possible) for which one could obtain a good lower bound for $f_{t_0}(x+iy_0)$ when $x=N_0$; we arbitrarily chose a target lower bound of $|f_{t_0}(x+iy_0)| \geq 0.03$ to provide an adequate safety margin.  From \eqref{ftxy} one had
\begin{equation}\label{ftxy-2}
 f_{t_0}(x+iy_0) = \sum_{n=1}^N \beta_n + O_{\leq}( |\gamma| |\sum_{n=1}^N \alpha_n| ) + O_{\leq}\left( |\gamma| \sum_{n=1}^N n^{y_0} \frac{b_n^{t_0}}{n^\sigma} (n^{|\kappa|}-1) \right)
\end{equation}
where
$$ \beta_n \coloneqq \frac{b_n^{t_0}}{n^{\sigma+iT}}$$
and
$$ \alpha_n \coloneqq n^{y_0} \frac{b_n^{t_0}}{n^{\sigma+iT}}.$$
The final term on the right-hand side of \eqref{ftxy-2} can be estimated as in Section \ref{b-bound} and is negligible in practice.  To control the other two terms, we use the following lemma (which roughly speaking corresponds to a simplified version of the ``Euler 2 mollifier'' version of the ``Euler 5 mollifier'' analysis in Section \ref{b-bound}):

\begin{lemma}\label{trib2}  Let $\alpha_1,\dots,\alpha_N$ be complex numbers, and let $\beta_2$ be a number such that whenever $1 \leq n \leq N$ is even, $\beta_2 \alpha_{n/2}$ lies on the line segment $\{ \theta \alpha_n: 0 \leq \theta \leq 1\}$ connecting 0 with $\alpha_n$.  Then we have the lower bound
$$ |1-\beta_2| \left|\sum_{n=1}^N \alpha_n\right| \geq 2 |\alpha_1| - (1-|\beta_2|) \sum_{n=1}^N |\alpha_n| - 2 |\beta_2| \sum_{N/2 < n \leq N} |\alpha_n|$$
and the upper bound
$$ |1-\beta_2| \left|\sum_{n=1}^N \alpha_n\right| \leq (1-|\beta_2|) \sum_{n=1}^N |\alpha_n| + 2 |\beta_2| \sum_{N/2 < n \leq N} |\alpha_n|.$$
\end{lemma}

\begin{proof} The quantity $|1-\beta_2| |\sum_{n=1}^N \alpha_n|$ can be written as
$$ \left|\sum_{n=1}^{2N} (1_{n \leq N} \alpha_n - 1_{2|n} \alpha_{n/2} \beta_2)\right|.$$
By the triangle inequality, this is bounded above by
$$ \sum_{n=1}^{2N} |1_{n \leq N} \alpha_n - 1_{2|n} \alpha_{n/2} \beta_2|$$
and below by
$$ 2 |\alpha_1| - \sum_{n=1}^{2N} |1_{n \leq N} \alpha_n - 1_{2|n} \alpha_{n/2} \beta_2|.$$
We have
$$ \sum_{n=1}^{2N} |1_{n \leq N} \alpha_n - 1_{2|n} \alpha_{n/2} \beta_2| = \sum_{n=1}^N |\alpha_n| - 1_{2|n} |\alpha_{n/2}| |\beta_2| + \sum_{n=N+1}^{2N} 1_{2|n} |\alpha_{n/2}| |\beta_2|$$
which we can rearrange as
$$ (1 - |\beta_2|) \sum_{n=1}^N |\alpha_n| + 2 |\beta_2| \sum_{N/2 < n \leq N} |\alpha_n|$$
and the claim follows.
\end{proof}

Using this lemma to lower bound $|\sum_{n=1}^N \beta_n|$ and upper bound $|\sum_{n=1}^N \alpha_n|$, and then using the triangle inequality, yields a lower bound on $|f_{t_0}(x+iy_0)|$ when $N=N_0$.  These quantities can be readily computed for many values of $t_0,y_0,N_0$, leading to an envelope for $\Lambda$ and $x \approx 4 \pi N_0^2$ that is depicted in Figure \ref{tradeoff}.

\begin{figure}[!ht]
  \includegraphics[width=1.0\linewidth]{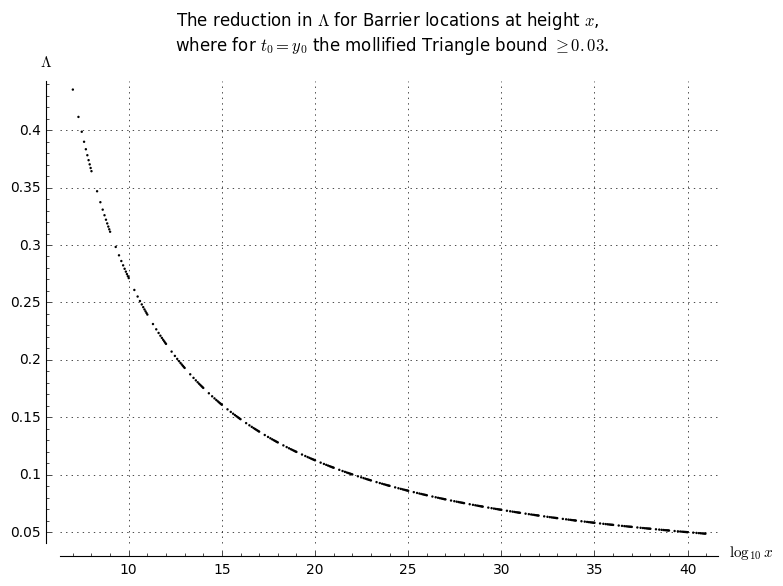}
  \caption{The envelope of potential choices of $x, \Lambda$.  Note the approximate inverse relationship between $\Lambda$ and $\frac{1}{\log x}$.}
  \label{tradeoff}
\end{figure}

By working in intervals $N \in [N_-,N_+]$ for some finite number of intervals $[N_-,N_+]$ covering $[N_0,N_1]$ for some large $N_1$ as in Section \ref{b-bound}, and then using a crude triangle inequality bound for $N \geq N_1$ as in Section \ref{c-bound}, we thus (in view of the conservative safety margin in our lower bounds for $|f_{t_0}(x+iy_0)|$) expect to be able to verify the hypothesis in Theorem \ref{ubc-0}(iii) for any choice of parameters $t_0, y_0, N_0$ as above.  The main remaining difficulty is then to verify the barrier hypothesis (Theorem \ref{ubc-0}(ii)).  This is by far the most numerically intensive step, and we proceed as in Section \ref{barrier-sec}, after using the \emph{ad hoc} procedure in Section \ref{select} to select $X_0$.   The graphs in figure \ref{fig:meshbarrier} illustrate that for increasing $x$, the number of xy-rectangles to be evaluated within the barrier, as well as the number of mesh points required per rectangle (measured at $t=0$), increase exponentially. 

\begin{figure}[!ht]
  \includegraphics[width=1.0\linewidth]{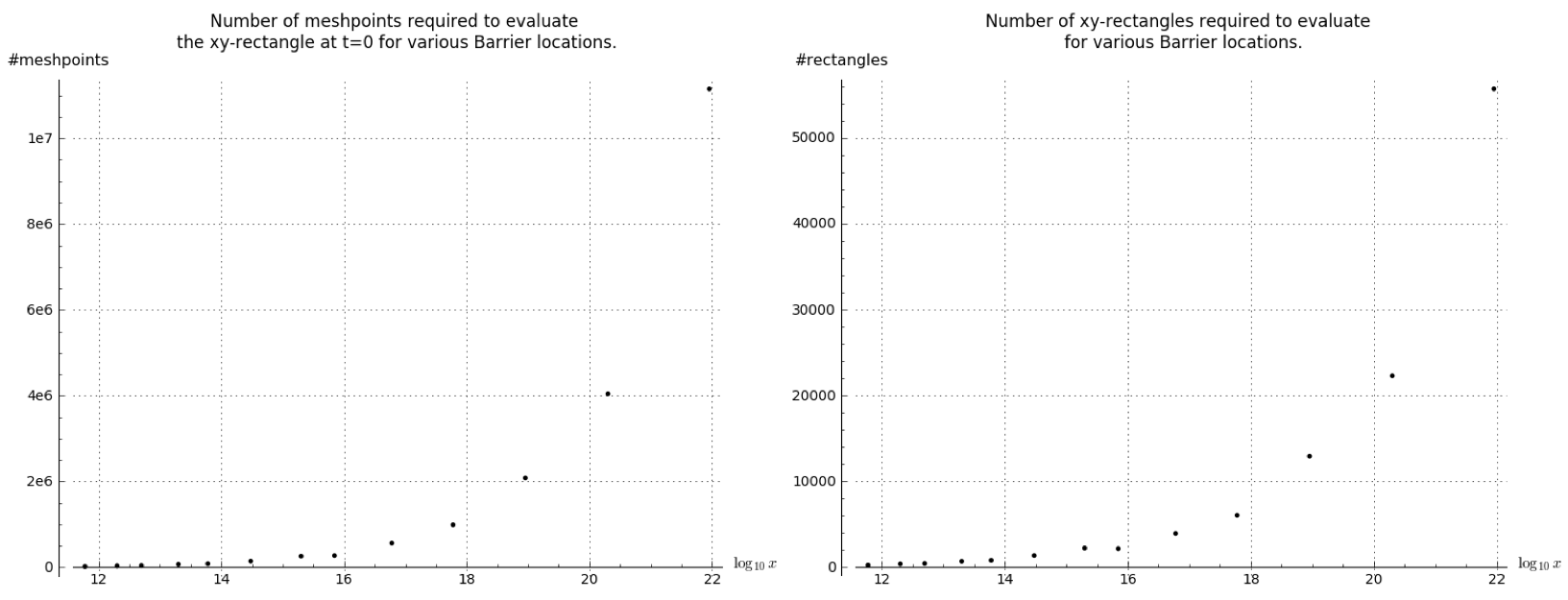}
  \caption{The left graph shows how the number mesh points of the xy-rectangle at $t=0$ increases with $x$ for each barrier. The graph on the right does the same, but now for the total number of xy-rectangles that need to be evaluated per barrier. }
  \label{fig:meshbarrier}
\end{figure}

All barrier runs generated a winding number of zero for each rectangle and the scripts completed successfully without any errors. For all barrier locations, the computations of the mesh points where calculated at $20$ digits accuracy except for the highest two where $10$ digits where used (to be able to compute it within a reasonable time). Checks where made before each formal run to assure the target accuracy would be achieved. 

The computations for $X=2 \times 10^{20} + 66447$ and $X=9 \times 10^{21} + 70686$ in the above table were massive, and performed using a Boinc \cite{anderson} based grid computing setup, in which a few hundred volunteers participated. Their contributions can be tracked at {\tt anthgrid.com/dbnupperbound}.

\section{Conflict of interest statement}

There are no conflicts of interest for this paper.

\end{document}